\documentclass[11pt]{elsarticle}
\usepackage{color}
\usepackage{graphicx,multirow}
\usepackage{amssymb,mathtools}
\usepackage{epstopdf}
\usepackage{bm}
\usepackage{float}
\usepackage{subcaption}
\usepackage{mathtools}
\usepackage[margin=1in,papersize={8.5in,11in}]{geometry}
\usepackage{times}
\usepackage{amsthm}

\usepackage{algorithm}% <=================http://ctan.org/pkg/algorithms
\usepackage{algpseudocode}% <========== http://ctan.org/pkg/algorithmicx

\newtheorem{lemma}{\bf Lemma}[section]
\newtheorem{theorem}{\bf Theorem}[section]

\newtheorem{proposition}{\bf Proposition}[section]

\def\vs{\vspace{0.2cm}} 
\def\vse{\vspace{0.1cm}}

\journal{arXiv}

\begin{document}
\begin{frontmatter}

\title{Dynamically orthogonal tensor methods for high-dimensional nonlinear PDEs}

\author[ucsc]{Alec Dektor}
\author[ucsc]{Daniele Venturi\corref{correspondingAuthor}}
\address[ucsc]{Department of Applied Mathematics\\
University of California Santa Cruz\\ Santa Cruz, CA 95064}
\cortext[correspondingAuthor]{Corresponding author}
\ead{venturi@ucsc.edu}

\begin{abstract} 
We develop new dynamically orthogonal 
tensor methods to approximate multivariate 
functions and the solution of high-dimensional 
time-dependent nonlinear partial differential 
equations (PDEs). 
The key idea relies on a hierarchical decomposition 
of the approximation space obtained by splitting 
the independent variables of the problem into 
disjoint subsets. This process, which can be conveniently 
be visualized in terms of binary trees, yields series expansions 
analogous to the classical Tensor-Train and Hierarchical Tucker 
tensor formats. By enforcing dynamic orthogonality 
conditions at each level of binary tree, we obtain 
coupled evolution equations for the modes spanning each 
subspace within the hierarchical decomposition. 
This allows us to effectively compute the solution to 
high-dimensional time-dependent nonlinear PDEs on 
tensor manifolds of constant rank, with no need 
for rank reduction methods. We also propose new 
algorithms for dynamic addition and removal of 
modes within each subspace. Numerical examples 
are presented and discussed for high-dimensional 
hyperbolic and parabolic PDEs in bounded domains. 
\end{abstract}

\end{frontmatter}

\section{Introduction}
\label{sec:intro}

High-dimensional partial differential equations (PDEs) arise 
in many areas of engineering, physical sciences and 
mathematics.  Classical examples are equations involving 
probability density functions (PDFs) such as 
the Fokker-Plank equation \cite{Risken}, the 
Liouville equation \cite{Venturi_PRS,HeyrimJCP_2014}, or 
the Boltzmann equation \cite{cercignani1988,dimarco2014}.  
Other types of high-dimensional PDEs can be obtained
as finite-dimensional approximations of functional 
differential equations \cite{venturi2018numerical},
such as the Hopf equation of turbulence \cite{Hopf,Hopf1,Monin2},
the Schwinger-Dyson equation of quantum mechanics \cite{Itzykson}, 
or the Martin-Siggia-Rose formulation of classical statistical dynamics \cite{Martin,Jensen,Jouvet,Phythian}.
Computing the solution to high-dimensional PDEs is a 
challenging problem that requires approximating 
high-dimensional functions, i.e., the solution to the PDE, and 
then developing appropriate numerical schemes 
to compute such functions accurately. 
Classical numerical methods based on tensor product 
representations are not viable in high-dimensions, 
as the number of degrees of freedom grows 
exponentially fast with the dimension. 
To address this problem there have been substantial 
research efforts in recent years on 
high-dimensional numerical approximation theory.
Techniques such as sparse collocation 
\cite{Bungartz,Chkifa,Barthelmann,Foo1,Akil}, 
high-dimensional model representations (HDMR) 
\cite{Li1,CaoCG09,Baldeaux} and, more recently, 
deep neural networks \cite{Raissi,Raissi1,Zhu2019} 
and tensor methods 
\cite{khoromskij,Bachmayr,parr_tensor,Hackbusch_book,ChinestaBook,Kolda} 
were proposed to mitigate the exponential growth of 
the degrees of freedom, the computational cost and 
memory requirements. 

In this paper, we develop a new dynamically orthogonal 
tensor method to approximate multivariate 
functions and the solution of high-dimensional 
time-dependent nonlinear PDEs. The key idea relies on 
a hierarchical decomposition of the function space 
in terms of a sequence of nested subspaces of smaller 
dimension. Such decomposition is induced by 
by splitting the independent variables of the problem 
recursively into two disjoint subsets which 
can conveniently be visualized by binary trees.
In particular, we study two classes of trees which are 
analogous to the Tensor-Train  (TT) \cite{OseledetsTT} 
and Hierarchical Tucker (HT) \cite{Grasedyck2018}
tensor formats. By enforcing dynamic orthogonality (DO) 
\cite{do} or bi-orthogonality (BO) \cite{Cheng2013}
conditions at each level of the TT or the HT binary tree, 
we obtain coupled evolution equations for the modes 
spanning each subspace within the hierarchy. 
This allows us to represent the time evolution of high-dimensional 
functions and compute the solution of high-dimensional 
time-dependent nonlinear PDEs on a tensor manifold with 
constant rank. 
This formulation has several advantages over 
classical numerical tensor methods. In particular, 
the hard-to-compute nonlinear projection 
\cite{Lubich2018,hierar} that maps the solution 
of high-dimensional PDEs onto a tensor manifold with 
constant rank \cite{h_tucker_geom} here is 
represented explicitly by the hierarchical DO/BO 
propagator\footnote{The hierarchical DO/BO propagator 
is nonlinear even for linear PDEs. Such nonlinearity implicitly 
represnts the projection onto a tensor manifold with 
constant rank.}, i.e., by a system of coupled one-dimensional 
nonlinear PDEs. 
In other words, there is no need to perform tensor 
rank reduction \cite{hsvd_tensors_grasedyk,Grasedyck2018,Kressner2014}, 
rank-constrained temporal integration \cite{Lubich2018,hierar}, or 
Riemannian optimization \cite{DaSilva2015}, 
when solving high-dimensional PDEs with the hierarchical 
subspace decomposition method we 
propose\footnote{Classical numerical tensor methods 
for high-dimensional PDEs with explicit time stepping 
schemes require rank-reduction to project the solution 
back into the tensor manifold with prescribed rank (see 
\cite{venturi2018numerical} \S 5.5), the so-called 
retraction step \cite{DaSilva2015}. 
This can be achieved, e.g., by a sequence of suitable 
matricizations followed by hierarchical singular value 
decomposition \cite{hsvd_tensors_grasedyk,Grasedyck2018,Kressner2014}, 
or by optimization 
\cite{DaSilva2015,Kolda,parr_tensor,Silva,Rohwedder,Karlsson}.  
Rank reduction can be computationally intensive, 
especially if performed at each time step. 
Tensor methods with implicit time stepping suffer 
from similar issues. In particular, the nonlinear system 
that yields the solution at the next time step needs 
to be solved on a tensor manifold with constant rank 
by using, e.g., Riemannian optimization algorithms 
\cite{DaSilva2015,h_tucker_geom,Etter}.}.

This paper is organized as follows. In Section 
\ref{sec:rec_bi_orth} we introduce a recursive 
bi-orthogonal decomposition method for time-independent
multivariate functions, develop error estimates 
and provide simple examples of application.   
In Section \ref{sec:time_evolution} we extend the 
recursive bi-orthogonal decomposition to time-dependent 
functions, and develop a hierarchy of nested 
time-dependent orthogonal projections generalizing the 
DO and BO conditions to tensor formats with multiple 
levels. We also prove that the approximations resulting 
from the BO and the DO conditions are equivalent, in the 
sense that they span the same function spaces. 
In Section \ref{sec:PDEs}, we apply the recursive 
subspace decomposition method to compute 
the solution of high-dimensional nonlinear PDEs. In 
Section \ref{sec:numerics} we provide numerical 
examples demonstrating the accuracy and computational 
effectiveness of the recursive subspace decomposition 
method we propose. Specifically we study high-dimensional 
hyperbolic and parabolic PDEs. The main findings are 
summarized in Section \ref{sec:summary}.

\section{Recursive bi-orthogonal decomposition of time-independent multivariate functions}
\label{sec:rec_bi_orth}

Let $\Omega$ be a subset of $\mathbb{R}^d$ ($d\geq 2$) that 
contains an open set\footnote{If $\Omega\subseteq \mathbb{R}^d$ 
contains an open set then $\text{dim}(\Omega)=d$.}, and let 
\begin{equation}
\label{space_sol}
u: \Omega\to \mathbb{R}
\end{equation} 
be a multivariate function which we assume 
to be an element of a separable Hilbert space 
$\mathcal{H}(\Omega)$. A possible choice 
of such Hilbert space is the Sobolev space
\begin{equation}
H^{k}(\Omega) =\left\{u\in L^2(\Omega)\, : \, D^{\bm \alpha} u\in L^2(\Omega) \quad \text{for all}\quad  |\bm \alpha|\leq k \right\} \ , \qquad k=0,1,2,\dots
\label{Sobolev}
\end{equation}
where $\bm \alpha = (\alpha_1,\dots, \alpha_d)$ is a 
multi-index and 
\begin{equation}
D^{\bm \alpha} u =\frac{\partial^{|\bm \alpha |} u}
{\partial x_1^{\alpha_1}\dots \partial x_d^{\alpha_d}}, 
\qquad |\bm \alpha|= \alpha_1+\dots +\alpha_d.
\end{equation}
Note that \eqref{Sobolev} includes the classical Lebesgue space 
$L^2(\Omega)=H^{0}(\Omega)$. 
We equip \eqref{Sobolev} with the standard inner product 
\begin{equation}
\langle f,g\rangle_{H^{k}(\Omega)} = \sum_{|\bm \alpha| \leq k} \int_{\Omega} D^{\bm \alpha} f(\bm x) D^{\bm \alpha}g(\bm x)dx_1\dots dx_d.
\label{iprod}
\end{equation}
If needed, this inner product can be weighted 
by a non-negative separable density 
$\rho_1(x_1)\cdots \rho_d(x_d)$.
Any separable Hilbert space is isomorphic to $L^2$, and 
it can be represented as a tensor product of two Hilbert 
spaces \cite[p.51]{math_phys_vol1}, i.e.,
\begin{equation}
    \mathcal{H} \cong \mathcal{H}_1 \otimes \mathcal{H}_2.
    \label{hilbert_iso}
\end{equation}
The spaces $\mathcal{H}_1$ and $\mathcal{H}_2$ may be specified 
by partitioning the spatial variables $\{x_1, \ldots, x_d\}$ into 
two disjoint subsets. This is equivalent to represent the domain $\Omega$ 
as a Cartesian product of two sub-domains (whenever possible). 
For instance, consider the partition
\begin{equation}
\Omega=\Omega^{(1,\ldots,p)}\times \Omega^{(p+1,\ldots,d)}
\end{equation}
induced by the following splitting of the spatial variables 
\begin{equation}
\underbrace{(x_1, x_2, \ldots, x_d)}_{\text{in}\,\, \Omega} = 
(\underbrace{(x_1, \ldots, x_p)}_{\text{in}\,\, \Omega^{(1,\ldots,p)}}, 
\underbrace{(x_{p+1}, \ldots, x_d)}_{\text{in}\,\, \Omega^{(p+1,\ldots,d)}}). 
\end{equation}  
In this setting, the Sobolev space \eqref{Sobolev} admits 
the following decomposition
\begin{equation}
\label{sobolov_iso}
H^{k}(\Omega) \cong  H^{k}\left(\Omega^{(1,\ldots,p)}\right) 
\otimes H^{k}\left(\Omega^{(p+1,\ldots,d)}\right).
\end{equation}
The  inner products within each subspace 
$H^{k}\left(\Omega^{(1,\ldots,p)}\right)$ and 
$H^{k}\left(\Omega^{(p+1,\ldots,d)}\right)$ 
can be defined, respectively, as 
\begin{equation}
\langle f,g \rangle_{H^{k}\left(\Omega^{(1,\ldots,p)}\right)}
=\sum_{\alpha_1+\dots+\alpha_p \leq k}
\int_{\Omega^{(1,\ldots,p)}}
\frac{\partial^{\alpha_1+\dots+\alpha_p}f}{\partial x_1^{\alpha_1}\dots \partial x_p^{\alpha_p}} 
 \frac{\partial^{\alpha_1+\dots+\alpha_p} g}{\partial x_1^{\alpha_1}\dots \partial x_p^{\alpha_p}} 
dx_1 \cdots dx_p,
\label{iprod1}
\end{equation}
and
\begin{equation}
\langle f,g \rangle_{H^{k}\left(\Omega^{(p+1,\ldots,d)}\right)}
=\sum_{\alpha_{p+1}+\dots+\alpha_d \leq k}
\int_{\Omega^{(p+1,\ldots,d)}}
\frac{\partial^{\alpha_{p+1}+\dots+\alpha_d}f}{\partial x_{p+1}^{\alpha_{p+1}}\dots \partial x_d^{\alpha_d}} 
\frac{\partial^{\alpha_{p+1}+\dots+\alpha_d}g}{\partial x_{p+1}^{\alpha_{p+1}}\dots \partial x_d^{\alpha_d}} 
dx_{p+1} \cdots dx_d.
\label{iprod2}
\end{equation}
A representation of the multivariate function \eqref{space_sol} 
in the tensor product space \eqref{sobolov_iso} has the general 
form
\begin{equation}
\label{not_diag}	
 u(x_1, \ldots, x_d) = \sum_{i,j=1}^{\infty}a_{ij}
\varphi_{i}^{(1,\ldots,p)}(x_1, \ldots, x_p)
\varphi_{j}^{(p+1,\ldots,d)}(x_{p+1},\ldots,x_d) \ ,
\end{equation}
where $\varphi_i^{(1,\ldots,p)}$ and $\varphi_j^{(p+1,\ldots,d)}$ are 
orthonormal basis functions in $H^{k}(\Omega^{(1,\ldots,p)})$ and 
$H^{k}(\Omega^{(p+1,\ldots,d)})$, 
respectively\footnote{Orthonormality is relative to the 
inner products 
in $H^{k}(\Omega^{(1,\ldots,p)})$ and 
$H^{k}(\Omega^{(p+1,\ldots,d)})$.}.
The superscripts in \eqref{not_diag} denote which 
spatial components the function depends on.  This will be the 
case throughout this paper and for notational simplicity, the 
spatial arguments will be often omitted when there is no 
ambiguity.

With the isomorphism \eqref{sobolov_iso} and the inner products \eqref{iprod1}-\eqref{iprod2} set, it is 
straightforward to develop an operator
framework which guarantees the existence of 
a diagonalized bi-orthogonal representation of the field 
$u(x_1,\ldots,x_d)$. To this end, following Aubry {\em et. al.} 
\cite{aubry_1,aubry_2,aubry_3} and 
Venturi \cite{venturi_bi_orthogonal}
(see also \cite{venturi2006,venturi2008}), 
we define the integral operator
\begin{equation}
\begin{aligned}
&U_u : H^{k}\left(\Omega^{(1,\ldots,p)}\right) \to H^{k}\left(\Omega^{(p+1,\ldots,d)}\right), \\
&\left(U_u \psi \right)(x_{p+1}, \ldots, x_d) =
\langle u,\psi\rangle_{H^{k}\left(\Omega^{(1,\ldots,p)}\right)}.
\end{aligned}
\end{equation}
The formal adjoint of $U_u$, denoted by $U^\dagger$, 
is a linear operator defined by the requirement
\begin{equation}
\langle U_u \psi,\varphi\rangle_{H^{k}
\left(\Omega^{(p+1,\ldots,d)}\right)}=
\langle \psi ,U^\dagger_u \varphi
\rangle_{H^{k}\left(\Omega^{(1,\ldots,p)}\right)} \ ,
\end{equation}
for all 
$\psi\in H^{k}\left(\Omega^{(1,\ldots,p)}\right)$,
and all 
$\varphi\in H^{k}\left(\Omega^{(p+1,\ldots,d)}\right)$. 
By using integration by parts and discarding boundary 
conditions (formal adjoint operator) we obtain 
\begin{equation}
\begin{aligned}
&U_u^{\dagger} : H^{k}\left(\Omega^{(p+1,\ldots,d)}\right) 
\to H^{k}\left(\Omega^{(1,\ldots,p)}\right)  \ ,\\
&\left(U_u^{\dagger}\varphi\right)(x_1, \ldots, x_p) = 
\langle {u}, \varphi \rangle_{ H^{k}\left(\Omega^{(p+1,\ldots,d)}\right)} \ , 
\end{aligned}
\end{equation}
The subscript ``$u$'' in $U_u$ and $U^{\dagger}_u$ 
identifies the kernel of the integral operators.
We will shortly define a hierarchy of such operators 
and it will be important to distinguish them by their kernels. 
Next, we introduce the following correlation operators
\begin{equation}
R_u = U_u U_u^{\dagger},\qquad 
R_u : H^{k}(\Omega^{(p+1,\ldots,d)}) \to H^{k}(\Omega^{(p+1,\ldots,d)}) \ ,
\end{equation}
and 
\begin{equation}
L_u = U_u^{\dagger} U_u, \qquad 
L_u : H^{k}(\Omega^{(1,\ldots,p)}) \to H^{k}(\Omega^{(1,\ldots,p)}).
\end{equation}
Note that $L_u$ and $R_u$ are self-adjont 
relative to \eqref{iprod1} and \eqref{iprod2}, 
respectively. Moreover, if $U_u$ is compact (e.g., if we consider 
a decomposition in $H^{0}=L_2$), then $U_u^{\dagger}$
is compact, and therefore $L_u$ and $R_u$ 
are compact. Hence, by the Riesz-Schauder theorem, 
they have the same discrete spectra (see e.g. \cite[p.185]{kato}). 
By a direct calculation, it can be show that 
\begin{align}
 \left(R_u\varphi\right) (x_{p+1}, \ldots, x_d)
 = \langle r_u,\varphi \rangle_{H^k\left(\Omega^{(p+1,\ldots,d)}\right)} \qquad \varphi\in H^{k}(\Omega^{(p+1,\ldots,d)}) \ ,
\label{Ru}
\end{align}
where the correlation function $r_u$ is defined by
\begin{align}
r_u(x_{p+1}, \ldots, x_d, x_{p+1}^{\prime}, \ldots, x_d^{\prime})
=\langle u, u \rangle_{H^k\left(\Omega^{(1,\ldots,p)}\right)}.
\label{r_u}
\end{align}
Similarly,
\begin{equation}
\label{Lu}
\begin{aligned}
\left(L_u\psi \right)(x_1,\ldots,x_p) = \langle l_u,\psi \rangle_{H^k\left(\Omega^{(1,\ldots,p)}\right)}
\qquad \psi\in H^{k}(\Omega^{(1,\ldots,p)}) \ , 
\end{aligned}
\end{equation}
where the correlation function $l_u$ is defined by
\begin{align}
l_u(x_1,\ldots,x_p,x_1^{\prime},\ldots,x_p^{\prime}) =\langle u, u \rangle_{H^k\left(\Omega^{(p+1,\ldots,d)}\right)}.
\label{l_u}
\end{align}
It is a classical result in the spectral theory of 
compact operators (see e.g., \cite{aubry_1,aubry_3}) 
that there exists a canonical decomposition of the 
field \eqref{space_sol} of the form 
\begin{equation}
\label{biorthogonal}
 u= \sum_{k=1}^{\infty} \lambda_k \psi_k^{(1,\ldots,p)} \psi_k^{(p+1,\ldots, d)} \ ,
\end{equation}
where the modes $\psi_k^{(1,\ldots,p)}$ and 
$\psi_k^{(p+1,\ldots, d)}$ satisfy the eigenvalue problem
\begin{equation}
\left[
\begin{array}{cc}
U_u & 0\\
0  & U_u^\dagger
\end{array}
\right]
\left[
\begin{array}{c}
\psi_k^{(1,\ldots,p)} \\
\psi_k^{(p+1,\ldots,d)} 
\end{array}
\right]
= \lambda_k
\left[
\begin{array}{cc}
0 & 1\\
1  & 0
\end{array}
\right] \left[
\begin{array}{c}
\psi_k^{(1,\ldots,p)} \\
\psi_k^{(p+1,\ldots,d)} 
\end{array}
\right].
\label{dispersion}
\end{equation}
Moreover, it can be shown that 
$\lambda_1 \geq \lambda_2 \geq \cdots \geq 0$, and
\begin{equation}
\label{biorthogonal_properties}
\begin{aligned}
\langle \psi^{(1,\ldots,p)}_i \psi^{(1,\ldots,p)}_j 
\rangle_{H^k\left(\Omega^{(1,\ldots, p)}\right)} = 
\langle \psi^{(p+1,\ldots,d)}_i \psi^{(p+1,\ldots,d)}_j 
\rangle_{H^k\left(\Omega^{(p+1,\ldots, d)}\right)}  = \delta_{ij}.
\end{aligned}
\end{equation}
The series \eqref{biorthogonal} is usually called 
bi-orthogonal (or Schmidt) decomposition of the multivariate 
function $u$, and it converges in norm. The modes 
$\psi_k^{(1,\ldots,p)}$ are eigenfunctions 
of the operator $L_u$ with corresponding eigenvalue $\lambda_k^2$, 
while the modes $\psi_k^{(p+1,\ldots,d)}$ are eigenfunctions 
of the operator $R_u$ with corresponding eigenvalues $\lambda_k^2$, i.e., 
\begin{equation}
\begin{aligned}
L_u \psi_k^{(1,\ldots,p)} &= \lambda_k^2 \psi_k^{(1,\ldots,p)} \ , \\
R_u \psi_k^{(p+1,\ldots,d)} &= \lambda_k^2 \psi_k^{(p+1,\ldots, d)}.
\end{aligned}
\label{eigenP}
\end{equation}
In practice, since $\psi_k^{(1,\ldots,p)}$ and 
$\psi_k^{(p+1,\ldots,d)}$ are determined up to two 
unitary transformations \cite{aubry_2}, to compute 
\eqref{biorthogonal} we can solve one of the two eigenvalue 
problems in \eqref{eigenP} (the one with smaller dimension), 
and then use one of the dispersion relations 
\eqref{dispersion}, i.e., 
\begin{equation}
\label{eig_fun_proj}
\psi^{(p+1,\ldots,d)}_k = \frac{1}{\lambda_k}U_u
\psi^{(1,\ldots,p)}_k, \qquad
\psi^{(1,\ldots,p)}_k = \frac{1}{\lambda_k}U^\dagger_u
\psi^{(p+1,\ldots,d)}_k.
\end{equation}

\subsection{Hierarchical subspace decomposition and tensor formats} 
\label{sec:recursive}
To obtain a series expansion of the multivariate function 
$u(x_1,\ldots,x_d)$ in terms of univariate functions, 
we apply the bi-orthogonal decomposition method 
discussed in the previous Section recursively.  
The way in which the variables are split in each step of the 
recursive procedure, i.e., the choice of $p$ in \eqref{sobolov_iso}, 
can be conveniently visualized by binary trees. 
In Figure \ref{fig:trees} we provide two simple examples 
of such binary trees corresponding to the 
Tensor-Train (TT) \cite{OseledetsTT} 
and the Hierarchical Tucker (HT) \cite{Grasedyck2018} 
tensor formats (see also \cite{Hackbusch_book,venturi2018numerical}, 
and the references therein).
\begin{figure}[t]
	\centerline{
	\rotatebox{90}{\hspace{1.3cm}  \footnotesize}
\includegraphics[width=0.49\textwidth]{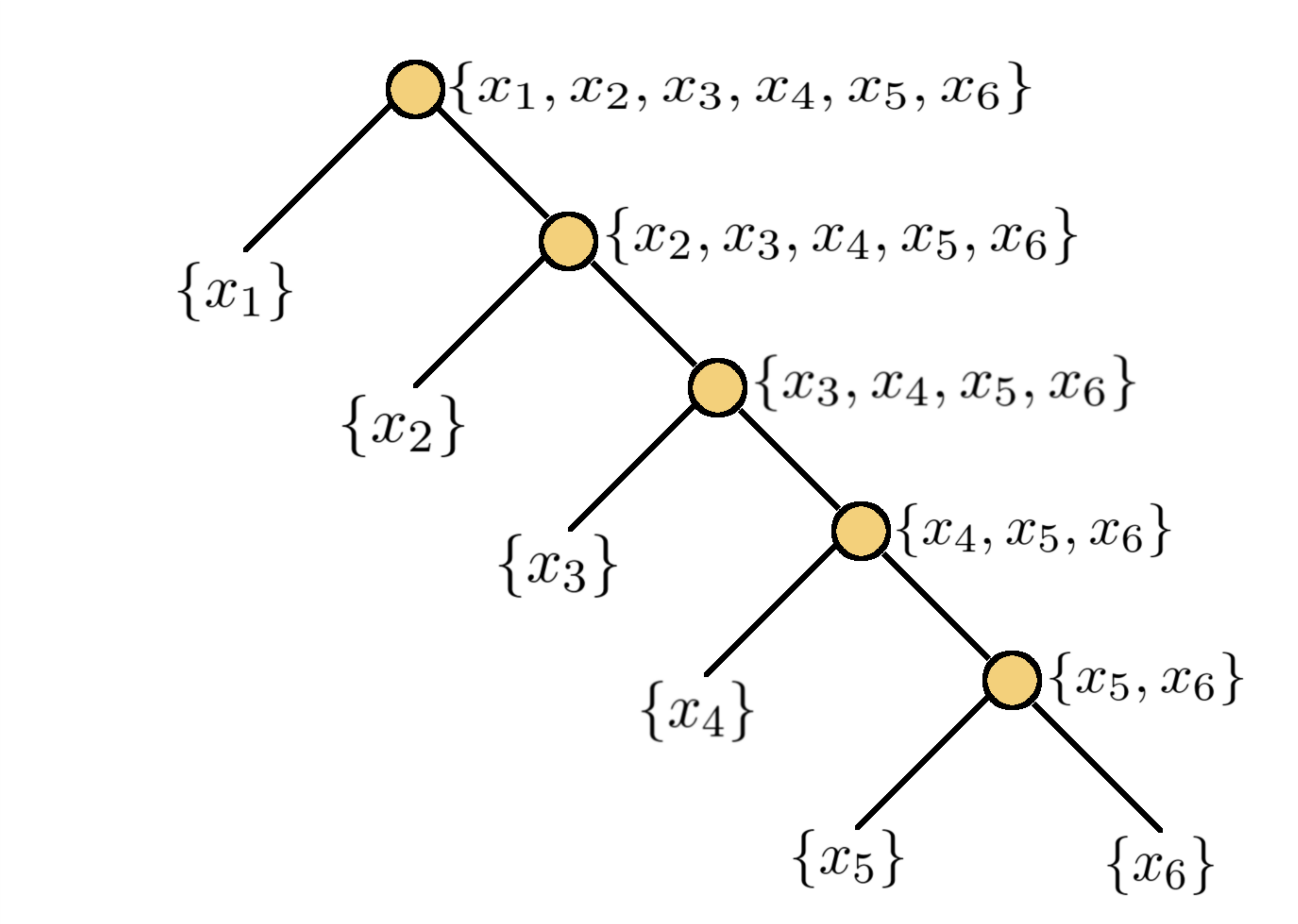}
\includegraphics[width=0.49\textwidth]{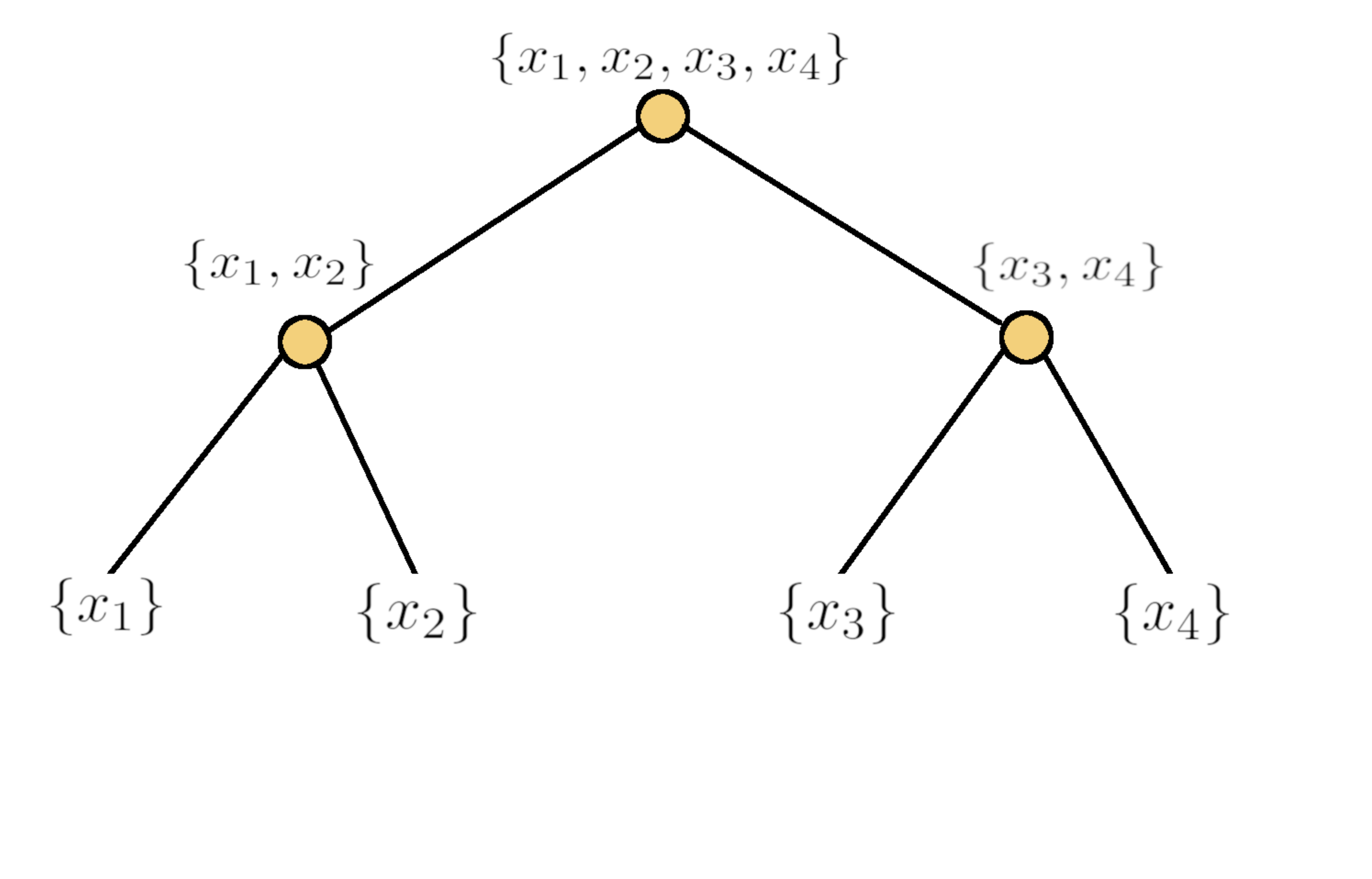}		
}
\caption{Binary trees corresponding to different tensor formats. 
Left: Tensor Train (TT) decomposition of a six-dimensional function.  Right: Hierarchical Tucker (HT) decomposition of a four-dimensional function.}
\label{fig:trees}
\end{figure}

\subsubsection{Tensor Train (TT) format}
\label{sec:TTdec}
The Tensor-Train  format singles out one variable at a time, 
resulting in a binary tree with depth $(d-1)$ when decomposing 
$d$-variate functions  $u(x_1,\ldots, x_d)$ (see Figure \ref{fig:trees}). 
This corresponds to the following hierarchical 
subspace decomposition of the Sobolev space \eqref{Sobolev}
\begin{align}
H^k(\Omega)= &H^k\left(\Omega^{(1)}\right)\otimes H^k\left(\Omega^{(2,\ldots,d)}\right), \nonumber \\
=&  H^k\left(\Omega^{(1)}\right)\otimes 
\left[H^k\left(\Omega^{(2)}\right)
\otimes H^k\left(\Omega^{(3,\ldots,d)}\right)\right],\nonumber\\
=&  H^k\left(\Omega^{(1)}\right)\otimes 
\left[H^k\left(\Omega^{(2)}\right)
\otimes \left\{H^k\left(\Omega^{(3)}\right)\otimes 
H^k\left(\Omega^{(4,\ldots,d)}\right)\right\}\right] \ ,\nonumber\\
\cdots &. \nonumber
\end{align}
in which we diagonalize each tensor product representation using 
the bi-orthogonal decomposition method. This yields the 
following TT expansion of the multivariate function 
$u(x_1,\ldots, x_d)$
\begin{align}
 \label{TT_1}
u &= \sum_{i_1 = 1}^{\infty} \lambda_{i_1} \psi_{i_1}^{(1)} \psi_{i_1}^{(2, \ldots, d)} \ , \\
\label{TT_2}
\psi_{i_1}^{(2, \ldots, d)}  &= \sum_{i_2 = 1}^{\infty} \lambda_{i_1 i_2} \psi_{i_1 i_2}^{(2)} \psi_{i_1 i_2}^{(3, \ldots, d)} \ ,\\
\nonumber
 &\vdots \\
 \label{TT_j}
 \psi_{i_1 \cdots i_{j-1}}^{(j, \ldots, d)} &= \sum_{i_{j} = 1}^{\infty} \lambda_{i_1 \cdots i_j} \psi_{i_1 \cdots i_j}^{(j)} \psi_{i_1 \cdots i_j}^{(j+1, \ldots, d)} \ , \\
 \nonumber
 &\vdots \\
 \label{TT_d-1}
 \psi_{i_1 \cdots i_{d-2}}^{(d-1, d)} &= \sum_{i_{d-1} = 1}^{\infty} \lambda_{i_1 \cdots i_{d-1}} \psi_{i_1 \cdots i_{d-1}}^{(d-1)} \psi_{i_1 \cdots i_{d-1} }^{(d)} \ , 
\end{align}
i.e., 
\begin{equation}
\label{series_tt}
   u= \sum_{i_1 = 1}^{\infty} \sum_{i_2 = 1}^{\infty}\cdots \sum_{i_{d-1} = 1}^{\infty} \lambda_{i_1} \lambda_{i_1 i_2} \cdots \lambda_{i_1 \cdots i_{d-1}} \psi_{i_1}^{(1)} \psi_{i_1 i_2}^{(2)} \cdots \psi_{i_1 \cdots i_{d-1}}^{( d-1)} \psi_{i_1 \cdots i_{d-1}}^{(d)}.
\end{equation}
Each of the bi-orthogonal modes can be obtained by solving a 
sequence of one-dimensional eigenvalue problems followed by projections.  Specifically, the eigenvalue problems are
\begin{equation}
\label{1d_eig_fun}
L_u \psi_{i_1}^{(1)} = \lambda_{i_1}^2 \psi_{i_1}^{(1)}, \qquad 
L_{\psi_{i_1 \cdots i_{k-1}}^{(k-1)}} \psi_{i_1 \cdots i_k}^{(k)}= 
\lambda_{i_1 \cdots i_k}^2 \psi_{i_1 \cdots i_k}^{(k)}, 
\qquad k = 2, \ldots, d-1,
\end{equation}
(see Eq. \eqref{Lu}) while the corresponding 
projections are defined as
\begin{equation}
\label{eig_fun_projection}
\psi^{(2,\ldots,d)}_{i_1} = \frac{1}{\lambda_{i_1}} 
\langle u, \psi^{(1)}_{i_1}\rangle_{H^k\left(\Omega^{(1)}\right)},
\qquad
\psi^{(j+1,\ldots,d)}_{i_1\cdots i_j} =  \frac{1}{\lambda_{i_1 \cdots i_j}}
\langle u, \psi^{(j)}_{i_1\cdots i_j}\rangle_{H^k\left(\Omega^{(j+1)}\right)}
\qquad j=2,\ldots, d-1.
\end{equation}

\subsubsection{Hierarchical Tucker (HT) format}
The Hierarchical Tucker format splits variables 
into disjoint subsets of equal size, whenever possible. 
In the case $d$-variate functions $u(x_1,\ldots, x_d)$, where 
$d = 2^n$ for some natural number $n$, the tree is balanced.  
In general, the Hierarchical Tucker tree is more shallow 
than a Tensor Train tree for the same number of variables $d$.  
In fact, the depth of the HT tree for $d = 2^n$ 
is $n = \log_2(d)$, while the corresponding 
TT tree has depth $2^n-1$. The HT 
format is based on the following hierarchical decomposition 
of the Sobolev space $H^k(\Omega)$
\begin{align}
H^k(\Omega)= &H^k\left(\Omega^{(1,\ldots,d/2)}\right)\otimes 
H^k\left(\Omega^{(d/2+1,\ldots,d)}\right), \nonumber \\
=&\left[H^k\left(\Omega^{(1,\ldots,d/4)}\right)\otimes 
H^k\left(\Omega^{(1+d/4,\ldots,d/2)}\right) \right]\otimes 
\left[H^k\left(\Omega^{(d/2+1,\ldots,3d/4)}\right)\otimes 
H^k\left(\Omega^{(1+3d/4,\ldots,d)}\right) \right]
\ ,\nonumber\\
\cdots &. \nonumber
\end{align}
As before, we diagonalize each tensor 
product representation as we proceed splitting variables down the tree. This yields the following 
sequence of bi-orthogonal decompositions
\begin{align}
\label{HT_1}
u&= \sum_{i_1 = 1}^{\infty} 
\lambda_{i_1}^{(1,\ldots, d/2)} 
\psi_{i_1}^{(1,\ldots, d/2)} 
\psi_{i_1}^{(d/2 + 1,\ldots,d)} \ , \\
\label{HT_2}
\psi_{i_1}^{(1,\ldots, d/2)} &= \sum_{i_2 = 1}^{\infty} 
\lambda_{i_1 i_2}^{(1,\ldots,d/4)} 
\psi_{i_1 i_2}^{(1,\ldots,d/4)} 
\psi_{i_1 i_2}^{(d/4+1,\ldots,d/2)} \ , \\
\label{HT_3}
 \psi_{i_1}^{(d/2  + 1,\ldots,d)} &= 
 \sum_{i_2=1}^{\infty} \lambda_{i_1 i_2}^{(d/2 +1, \ldots, 3d/4)} 
 \psi_{i_1 i_2}^{(d/2+1, \ldots, 3d/4)} 
 \psi_{i_1 i_2}^{(3d/4+1, \ldots, d)} \ , \\
 \nonumber &\vdots \\
 \psi_{i_1 \cdots i_{n-1}}^{(1,2)} &= 
 \sum_{i_n = 1}^{\infty} 
 \lambda_{i_1 \cdots i_n}^{(1)} \psi_{i_1 \cdots i_n}^{(1)} 
 \psi_{i_1 \cdots i_n}^{(2)} \ , \\
 \nonumber &\vdots \\
 \label{HT_d-1}
 \psi_{i_1 \cdots i_{n-1}}^{(d-1,d)} &= 
 \sum_{i_n=1}^{\infty} \lambda_{i_1 \cdots i_n}^{(d-1)} 
 \psi_{i_1 \cdots i_n}^{(d-1)} \psi_{i_1 \cdots i_n}^{(d)} \ ,
\end{align}
and the expansion
\begin{equation}
   u = \sum_{i_1 = 1}^{\infty} \cdots \sum_{i_n = 1}^{\infty} \lambda_{i_1}^{(1,\ldots,\frac{d}{2})} \cdots \lambda_{i_1 \cdots i_{n}}^{(d-1)}\psi_{i_1 \cdots i_n}^{(1)} \psi_{i_1 \cdots i_n}^{(2)} \cdots \psi_{i_1 \cdots i_n}^{(d)}.
    \label{series_ht}
\end{equation}
Similar to the TT format, a sequence of eigenfunction problems 
followed by projections are used to obtain the modes spanning 
the hierarchical subspaces.  However, in the HT case 
the eigenfunction problems are higher dimensional 
and not tractable for large $d$.
 
\paragraph{Remark}
Clearly, we may decompose a multivariate function $u(x_1,\ldots,x_d)$ 
by splitting variables in various ways at different levels 
of the decompositions.  Any binary tree which has leaves 
containing one index leads to a series expansion in terms of 
functions of one spatial variable. Separable Hilbert spaces 
defined on a Cartesian product of one-dimensional domains 
always allow such reduction.

\subsection{Error analysis}
\label{sec:error_analysis}
In this Section we develop an error analysis for 
the recursive biorthogonal decomposition we 
discussed in Section \ref{sec:recursive}. To this end we first 
state a Lemma which will be useful in Section \ref{sec:thresholding} for establishing a thresholding criterion to truncate the 
infinite sums in \eqref{series_tt} and \eqref{series_ht}.
\begin{lemma}
\label{prop:eig_val}
If $u \in H^k(\Omega)$ admits 
the bi-orthogonal expansion 
\begin{equation}
u  = \sum_{i=1}^{\infty} \lambda_i \psi^{(1,\ldots,p)}_i \psi^{(p+1,\ldots, d)}_i \qquad p\in\{2,\ldots,d-1\} \ ,
\end{equation}
then $\displaystyle\sum_{i=1}^{\infty} \lambda_i^2 = \| u \|^2_{H^k(\Omega)}$.
\end{lemma}
\begin{proof}
This result follows immediately from the orthonormality 
of the modes $\psi_i^{(1,\ldots,p)}$ and  
$\psi_i^{(p+1,\ldots,d)}$ relative to the inner products
\eqref{iprod1}-\eqref{iprod2}.

\end{proof}

Next, we analyze the error in the $H^k(\Omega)$ norm
between the Tensor Train series expansion 
\eqref{series_tt} and the truncated expansion
 \begin{equation}
\label{series_tt_trunc}
\tilde{u} = \sum_{i_1 = 1}^{r_1} \sum_{i_2 = 1}^{r_2(i_1)}
\cdots \sum_{i_{d-1} = 1}^{r_{d-1}(i_1, \ldots, i_{d-2})} 
\lambda_{i_1} \cdots \lambda_{i_1 \cdots i_{d-1}} 
\psi_{i_1}^{(1)} \psi_{i_1 i_2}^{(2)} \cdots 
\psi_{i_1 \cdots i_{d-1}}^{( d-1)} \psi_{i_1 \cdots i_{d-1}}^{(d)} \ ,
\end{equation}
where $r_1,r_2, \ldots , r_{d-1}$ are truncation ranks. 
To simplify indexing and array bounds 
for truncated TT expansions such as \eqref{series_tt_trunc}, 
we will omit the array indices in the rank arrays and write, e.g.,  
$\psi^{(j_1, \ldots, j_p)}_{i_1 \cdots i_k} ,\  k = 1,\ldots, r_{k}$ 
instead of $k = 1, \ldots, r_{k}(i_1,\ldots,i_{k-1})$, since the rank 
array indices are clear from the subscripts of the mode 
$\psi^{(j_1, \ldots, j_p)}_{i_1 \cdots i_k}$.  In this simplified 
notation, the truncated TT expansion \eqref{series_tt_trunc} 
can be written as 
 \begin{equation}
\label{series_tt_trunc_simple}
   \tilde{u}^{(1,\ldots,d)} = \sum_{i_1 = 1}^{r_1} \sum_{i_2 = 1}^{r_2}\cdots \sum_{i_{d-1} = 1}^{r_{d-1}} \lambda_{i_1} \cdots \lambda_{i_1 \cdots i_{d-1}} \psi_{i_1}^{(1)} \psi_{i_1 i_2}^{(2)} \cdots \psi_{i_1 \cdots i_{d-1}}^{( d-1)} \psi_{i_1 \cdots i_{d-1}}^{(d)}.
\end{equation}
\begin{proposition}
\label{thm:errors}
Let $u \in H^k(\Omega)$. The error incurred by truncating the infinite expanion \eqref{series_tt} to the finite expansion \eqref{series_tt_trunc_simple} is given by
\begin{equation}
\begin{aligned}
\| u - \tilde{u} \|^2_{H^k(\Omega)} 
= &\sum_{i_1 = r_1+1}^{\infty} \lambda_{i_1}^2 + \sum_{i_1 = 1}^{r_1}\sum_{i_2 = r_2 + 1 }^{\infty} \lambda_{i_1}^2 \lambda_{i_1 i_2}^2 + \cdots \\
&+ \sum_{i_1 = 1}^{r_1} \sum_{i_2 = 1}^{r_2} \cdots \sum_{i_{d-2}=1}^{r_{d-2}} \sum_{i_{d-1} = r_{d-1} +1}^{\infty} \lambda_{i_1}^2 \lambda_{i_1 i_2}^2 \cdots \lambda_{i_1 \cdots i_{d-1}}^2 .
\end{aligned}
\label{errors}
\end{equation}
\end{proposition}
\begin{proof}
Let us rewrite \eqref{series_tt} as 
\begin{equation}
   u  = \sum_{i_1 = 1}^{\infty} \lambda_{i_1} \psi_{i_1}^{(1)} \sum_{i_2=1}^{\infty} \lambda_{i_1 i_2} \psi_{i_1 i_2}^{(2)} \cdots \sum_{i_{d-1} = 1}^{\infty} \lambda_{i_1 \cdots i_{d-1}} \psi_{i_1 \cdots i_{d-1}}^{( d-1)} \psi_{i_1 \cdots i_{d-1}}^{(d)}
\end{equation}
and split each infinite sum into the superimposition of a 
finite sum and an infinite sum, i.e., 
\begin{equation}
\label{binomials}
\begin{aligned}   
   u = &\left(\sum_{i_1 = 1}^{r_1} \lambda_{i_1} \psi_{i_1}^{(1)} + 
   \sum_{i_1 = r_1+1}^{\infty} \lambda_{i_1} \psi_{i_1}^{(1)}\right) 
   \left( \sum_{i_2=1}^{r_2} \lambda_{i_1 i_2} \psi_{i_1 i_2}^{(2)} + 
   \sum_{i_2=r_2+1}^{\infty} \lambda_{i_1 i_2} \psi_{i_1 i_2}^{(2)} \right)  
   \cdots \\
  & \cdots \left( \sum_{i_{d-1} = 1}^{r_{d-1}} \lambda_{i_1 \cdots i_{d-1}} 
  \psi_{i_1 \cdots i_{d-1}}^{( d-1)} \psi_{i_1 \cdots i_{d-1}}^{(d)} + 
  \sum_{i_{d-1} = r_{d-1}+1}^{\infty} \lambda_{i_1 \cdots i_{d-1}} \psi_{i_1 
  \cdots i_{d-1}}^{( d-1)} \psi_{i_1 \cdots i_{d-1}}^{(d)} \right) .
  \end{aligned}
\end{equation}
Expanding the products in \eqref{binomials} yields the following expression
\begin{equation}
\begin{aligned}
u =& \sum_{i_1=1}^{r_1} \sum_{i_2 = 1}^{r_2} \cdots 
\sum_{i_{d-1}=1}^{r_{d-1}} \lambda_{i_1} \cdots 
\lambda_{i_1 \cdots i_{d-1}} \psi^{(1)}_{i_1} \cdots 
\psi^{(d-1)}_{i_1 \cdots i_{d-1}} \psi^{(d)}_{i_1 \cdots i_{d-1}} \\
&+ \sum_{i_1 = r_1 + 1}^{\infty} \sum_{i_2 = 1}^{\infty} \cdots 
\sum_{i_{d-1}=1}^{\infty} \lambda_{i_1} \cdots 
\lambda_{i_1 \cdots i_{d-1}} \psi^{(1)}_{i_1} \cdots 
\psi^{(d-1)}_{i_1 \cdots i_{d-1}} \psi^{(d)}_{i_1 \cdots i_{d-1}} \\
&+ \sum_{i_1 = 1}^{r_1} \sum_{i_2 = r_2 + 1}^{\infty} 
\sum_{i_3 = 1}^{\infty} \cdots \sum_{i_{d-1}=1}^{\infty} 
\lambda_{i_1} \cdots \lambda_{i_1 \cdots i_{d-1}} \psi^{(1)}_{i_1} 
\cdots \psi^{(d-1)}_{i_1 \cdots i_{d-1}} \psi^{(d)}_{i_1 \cdots i_{d-1}} \\
& + \sum_{i_1 = 1}^{r_1} \sum_{i_2 = 1}^{r_2} 
\sum_{i_3 = r_3+1}^{\infty} \sum_{i_4 = 1}^{\infty} \cdots 
\sum_{i_{d-1}=1}^{\infty} \lambda_{i_1} \cdots 
\lambda_{i_1 \cdots i_{d-1}} \psi^{(1)}_{i_1} 
\cdots \psi^{(d-1)}_{i_1 \cdots i_{d-1}} 
\psi^{(d)}_{i_1 \cdots i_{d-1}} \\
& \hspace{2cm} \vdots \\
& + \sum_{i_1 = 1}^{r_1} \sum_{i_2 = 1}^{r_2} \cdots 
\sum_{i_{d-2} = 1}^{r_{d-2}} \sum_{i_{d-1}=r_{d-1}+1}^{\infty} 
\lambda_{i_1} \cdots \lambda_{i_1 \cdots i_{d-1}} 
\psi^{(1)}_{i_1} \cdots \psi^{(d-1)}_{i_1 \cdots i_{d-1}} 
\psi^{(d)}_{i_1 \cdots i_{d-1}}. \\
\end{aligned}
\end{equation}
Using the orthogonality of each set of modes we obtain
\begin{equation*}
\begin{aligned}
\| u - \tilde{u} \|_{H^k(\Omega)}^2 =& \sum_{i_1 = r_1 + 1}^{\infty} \sum_{i_2 = 1}^{\infty} \cdots \sum_{i_{d-1}=1}^{\infty} \lambda_{i_1} \cdots \lambda_{i_1 \cdots i_{d-1}} \| \psi^{(1)}_{i_1} \cdots \psi^{(d-1)}_{i_1 \cdots i_{d-1}} \psi^{(d)}_{i_1 \cdots i_{d-1}} \|_{H^k(\Omega)}^2 \\
+& \sum_{i_1 = 1}^{r_1} \sum_{i_2 = r_2 + 1}^{\infty} \sum_{i_3 = 1}^{\infty} \cdots \sum_{i_{d-1}=1}^{\infty} \lambda_{i_1} \cdots \lambda_{i_1 \cdots i_{d-1}} \| \psi^{(1)}_{i_1} \cdots \psi^{(d-1)}_{i_1 \cdots i_{d-1}} \psi^{(d)}_{i_1 \cdots i_{d-1}} \|_{H^k(\Omega)}^2 \\
 +& \sum_{i_1 = 1}^{r_1} \sum_{i_2 = 1}^{r_2} \sum_{i_3 = r_3+1}^{\infty} \sum_{i_4 = 1}^{\infty} \cdots \sum_{i_{d-1}=1}^{\infty} \lambda_{i_1} \cdots \lambda_{i_1 \cdots i_{d-1}} \| \psi^{(1)}_{i_1} \cdots \psi^{(d-1)}_{i_1 \cdots i_{d-1}} 
 \psi^{(d)}_{i_1 \cdots i_{d-1}} \|_{H^k(\Omega)}^2 \\
& \hspace{2cm} \vdots \\
& + \sum_{i_1 = 1}^{r_1} \sum_{i_2 = 1}^{r_2} \cdots \sum_{i_{d-2} = 1}^{r_{d-2}} \sum_{i_{d-1}=r_{d-1}+1}^{\infty} \lambda_{i_1} \cdots \lambda_{i_1 \cdots i_{d-1}} \| \psi^{(1)}_{i_1} \cdots \psi^{(d-1)}_{i_1 \cdots i_{d-1}} \psi^{(d)}_{i_1 \cdots i_{d-1}} \|_{H^k(\Omega)}^2. \\
\end{aligned}
\end{equation*}
i.e.,  
\begin{equation*}
\begin{aligned}
\| u - \tilde{u} \|_{H^k(\Omega)}^2 = &\sum_{i_1 = r_1+1}^{\infty} 
\lambda_{i_1}^2 + \sum_{i_1 = 1}^{r_1}\sum_{i_2 = r_2 + 1 }^{\infty} 
\lambda_{i_1}^2 \lambda_{i_1 i_2}^2 + \cdots \\ 
&+ \sum_{i_1 = 1}^{r_1} \sum_{i_2 = 1}^{r_2} \cdots 
\sum_{i_{d-2}=1}^{r_{d-2}} \sum_{i_{d-1} = r_{d-1} +1}^{\infty} 
\lambda_{i_1}^2 \lambda_{i_1 i_2}^2 \cdots 
\lambda_{i_1 \cdots i_{d-1}}^2 .
\end{aligned}
\end{equation*}
\end{proof}

\paragraph{Remark} Proposition \ref{thm:errors} 
can be generalized to tensor formats corresponding to 
arbitrary binary trees, e.g., the HT format \eqref{series_ht}. 
In some sense, the equality \eqref{errors}
represents the infinite-dimensional 
version of well-known finite-dimensional 
results which bound the overall squared approximation 
error of multilinear singular value decompositions 
in the 2-norm by the sum (over the whole tree) of 
squares of deleted singular values. These types of 
results were first proven by De 
Lathauwer {\em et al.} in \cite{multilinear_svd_lathauwer}, 
and later generalized by Grasedyck \cite{hsvd_tensors_grasedyk} 
(see also Schneider and Uschmajew \cite{approx_rates}).

\paragraph{Remark} Recent error estimates by Griebel 
and Li \cite{Griebel2019} on the decay rate of singular 
values allow us to develop sharp upper bounds 
for \eqref{errors}  depending only on the 
multivariate rank, the smoothness 
of the function $u(x_1,\ldots,x_d)$ and 
other computable quantities. To obtain such estimates, 
it is sufficient to bound each 
eigenvalue $\lambda_{i_1}^2$, $\lambda_{i_1i_2}^2$, etc.,
with the corresponding sharp upper bound recently 
obtained in \cite{Griebel2019}.

\subsection{Computational aspects of TT and HT series expansions}
\label{sec:compute_biorthogonal}
To compute a recursive bi-orthogonal decomposition of 
the multivariate function $u(x_1,\ldots x_d)$ 
we first need to identify the self-adjoint operator 
\eqref{Lu} (or \eqref{Ru}) at each level of 
the binary tree (see Figure \ref{fig:trees}).
In other words, we need to compute the kernel \eqref{l_u}, 
(or \eqref{r_u}) and then solve the corresponding eigenfunction 
problem.  In the HT tensor format, computing such kernel 
requires evaluating a multivariate integral of dimension $d/2$ 
(at the first level of the tree) 
and then solving an eigenfunction problem of dimension $d/2$, which 
can be extremely challenging when $d$ is large. 
In the TT tensor format, this problem can be mitigated 
substantially. In fact, at the first level of the TT tree we have 
that the kernel of $L_u$ is a $(d-1)$-dimensional integral 
which can be evaluated, e.g., by using Quasi-Monte Carlo 
or more general lattice cubature rules \cite{qmc}. The corresponding 
eigenvalue problem \eqref{1d_eig_fun} and the projection
\eqref{eig_fun_projection} are both one-dimensional. 
From a numerical viewpoint, this is extremely advantageous, as 
we can accurately solve one dimensional eigenvalue problems 
in a collocation or a Galerkin setting at a low computational cost, 
once the kernels of the operators in \eqref{1d_eig_fun} are 
available.

\subsection{Thresholding hierarchical bi-orthogonal series expansions}
\label{sec:thresholding}
In this Section we develop a new thresholding criterion 
to truncate the series expansion  \eqref{series_tt} and 
\eqref{series_ht} to finite rank. 
To this end, we first notice that the amplitude of each term in the series is represented  
by products of eigenvalues from each level of the 
binary tree, since all eigenmodes are normalized.  
With this in mind, it is clear that a reasonable 
criterion to truncate the bi-orthogonal series expansion 
\eqref{series_tt} (or \eqref{series_ht}) to finite rank 
is to ensure that each of these eigenvalue products 
remains above a specified threshold $\sigma$. Hereafter, 
we develop this criterion for the TT format \eqref{series_tt}. 
The same technique can be applied to any other tensor format.  
We begin by setting some threshold value $\sigma$ 
for which we enforce $\lambda_{i_1} \lambda_{i_1 i_2} \cdots 
\lambda_{i_1 \cdots i_{d-1}} \geq \sigma$.  In the first level of 
the TT tree (Figure \eqref{fig:trees}), i.e., Eq. \eqref{TT_1}, 
we keep all modes with eigenvalues $\lambda_{i_1} \geq \sigma$, 
of which there will be a finite number $r_1$ because of property 
\eqref{biorthogonal_properties}.  Then we proceed to the second 
level of the TT tree and decompose $\psi^{(2,\ldots,d)}_{i_1}$ 
($1 \leq i_1 \leq r_1$) as in \eqref{TT_2}.  Here we set new 
thresholds $\sigma_{i_1} = \sigma/\lambda_{i_1}$ and keep 
all modes $\psi^{(2)}_{i_1 i_2}, \psi^{(3,\ldots,d)}_{i_1 i_2}$ 
with eigenvalues $\lambda_{i_1 i_2 } \geq \sigma_{i_1}$.  
Proceeding recursively in this way down to 
the $j^{\text{th}}$ level of the TT tree 
we have the thresholds $\sigma_{i_1 \cdots i_j} = \sigma_{i_1 
\cdots i_{j-1}}/\lambda_{i_1 \cdots i_j}$.  It is reasonable to 
disregard modes corresponding to eigenvalues smaller 
than $\sigma$ in the first bi-orthogonal decomposition 
since 
\begin{equation}
\label{eig_inequality}
\lambda_{i_1} \cdots \lambda_{i_1 \cdots i_{j-1}} \geq \lambda_{i_1} \cdots \lambda_{i_1 \cdots i_j} 
\end{equation}
for all $j = 2, \ldots, d-1$.
Indeed, Lemma \ref{prop:eig_val} implies that 
$\lambda_{i_1 \cdots i_j} \leq 1$ for all $j = 2, \ldots, d-1$ 
from which \eqref{eig_inequality} immediately follows.  Another 
desirable consequence of Lemma \ref{prop:eig_val} is that 
$\sigma_{i_1 \cdots i_{j-1}} \leq \sigma_{i_1 \cdots i_j}$ 
for all $j = 2,\ldots,d-1$. As a result, bi-orthogonal 
decompositions at different levels of the binary trees  are 
truncated to a different number of modes. Let us now 
summarize the thresholding algorithm for Tensor Train formats.
On the first level of the tree we decompose 
$u(x_1,\ldots,x_d)$ as in \eqref{TT_1}, for which 
we keep $r_1$ modes, identified by the criterion 
$\lambda_{i_1}\geq \sigma$. For each of the 
modes $\psi_{i_1}^{(2,\ldots,d)}$ ($i_1 = 1,\ldots,r_1$), 
we perform the decomposition \eqref{TT_2} on 
the second level of the binary tree with mode-specific 
thresholds $\sigma_{i_1}=\sigma/\lambda_{i_1}$. 
Hence, the bi-orthogonal decomposition of 
$\psi_{i_1}^{(2,\ldots,d)}$, has $r_2(i_1)$ 
modes, i.e.,  the ranks in the second level are 
described by the vector $r_2$.  
For each of the 
modes\footnote{The total number of modes 
 $\psi^{(3,\ldots,d)}_{i_1 i_2}$ is 
$r_1 \displaystyle\sum_{i_1=1}^{r_1} r_2(i_1)$.} 
$\psi^{(3,\ldots,d)}_{i_1 i_2}$, 
the decomposition at the third level of the tree is 
performed with thresholds $\sigma_{i_1 i_2}=\sigma/\lambda_{i_1i_2}$. 
This yields the truncation rank $r_3(i_1, i_2)$ in the 
decomposition of $\psi_{i_1 i_2}^{(3,\ldots,d)}$. 
Thus, the bi-orthogonal ranks for the third level of the TT 
tree are described by a matrix.  In general, on the 
$j^{\text{th}}$ level of the tree, the rank is described 
by a tensor $r_j(i_1,\ldots,i_{j-1})$ of dimension $j-1$. 

%DA QUI

\subsubsection{An example: recursive bi-orthogonal decomposition of a 3D function}
\label{sec:ex_1}
In this Section, we apply the recursive 
bi-orthogonal decomposition method to a simple 
a three-dimensional function\footnote{For 
two- and three-dimensional functions TT and 
HT tensor formats are equivalent.}
defined on the cube $\Omega =  [-1,1]^3$.
Specifically, we consider  
\begin{equation}
u(x_1,x_2,x_3) = e^{\sin(x_1 + 2x_2 + 3x_3)} + x_2 x_3 \ , \qquad (x_1,x_2,x_3) \in \Omega.
\label{3d_example}
\end{equation}
This function is shown in Figure \ref{fig:3Dfunction} together 
with the binary tree representing the tensor format we use in the recursive bi-orthogonal decomposition. 
\begin{figure}[t]
	\centerline{
	\rotatebox{90}{\hspace{1.3cm} \footnotesize}
		\includegraphics[width=0.45\textwidth]{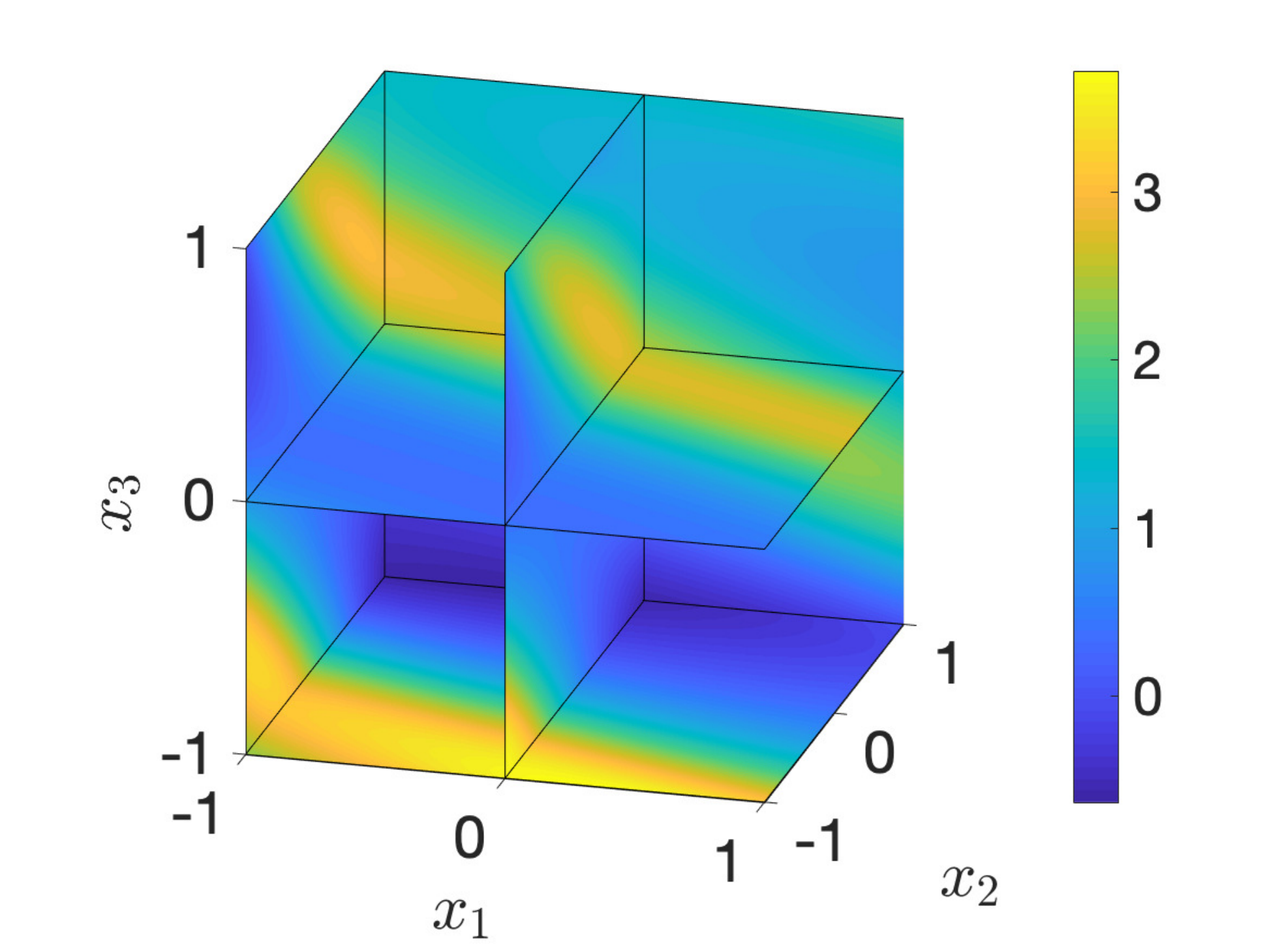}\hspace{1cm}
		\includegraphics[width=0.25\textwidth]{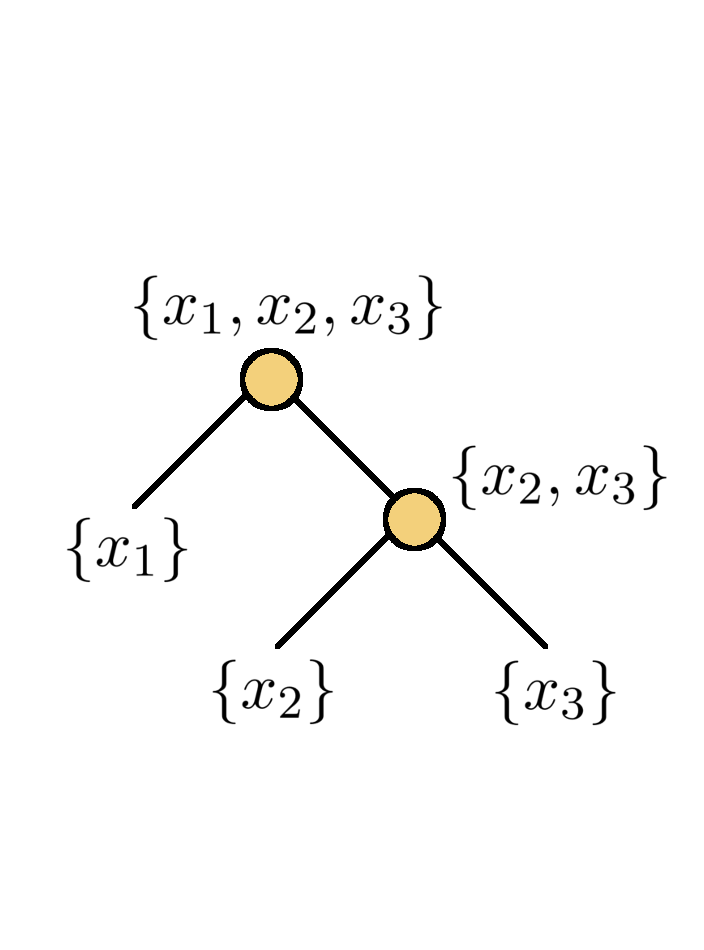}	
}
\caption{Volumetric plot of the three-dimensional function \eqref{3d_example} and binary tree representing the tensor format 
used in the recursive bi-orthogonal decomposition.}
\label{fig:3Dfunction}
\end{figure}
We discretize \eqref{3d_example} on a three-dimensional 
tensor product grid with $50$ Gauss-Legendre collocation points 
\cite{spectral} in each variable (125000 points total). 
Regarding the function space in which we perform the decomposition, 
in this example we set $k=0$ in \eqref{Sobolev}, i.e., we 
consider the classical  $L^2(\Omega)=H^{0}(\Omega)$ 
function space. In this setting, the kernel of 
the integral operator $L_u$ 
in \eqref{1d_eig_fun} (first level of the binary tree), 
reduces to 
\begin{equation}
l_u(x_1,x_1^{\prime}) = \int_{-1}^1 \int_{-1}^1 u(x_1,x_2,x_3) u(x_1^{\prime},x_2,x_3) dx_2 dx_3.
\end{equation}
This integral is computed with the Gauss-Legendre quadrature rule 
corresponding to the chosen grid points. The $x_1$-modes 
are solutions of the eigenvalue problem
\begin{equation}
\label{2d_eig_fun_problem}
\int_{-1}^1 l_u(x_1, x_1^{\prime}) \psi_{i_1}^{(1)}(x_1^{\prime}) dx_1^{\prime} = \lambda_{i_1}^2 \psi^{(1)}_{i_1}(x_1).
\end{equation}

\begin{figure}[t]
\centerline{\footnotesize\hspace{0.25cm} (a) \hspace{4.6cm}(b)\hspace{4.6cm} (c)}
	\centerline{
	\rotatebox{90}{\hspace{1.3cm} \footnotesize}
		\includegraphics[width=0.3\textwidth]{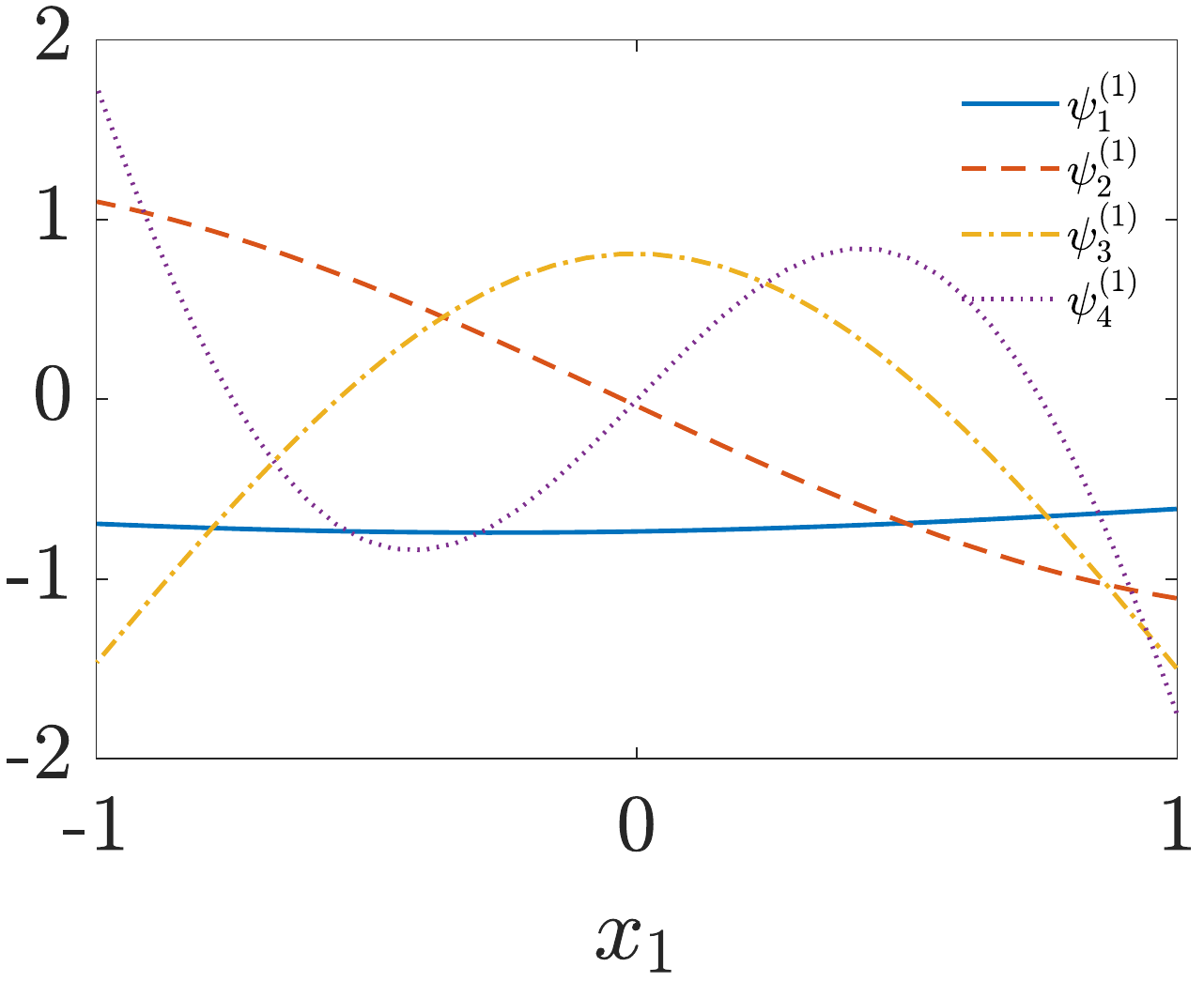}
		\includegraphics[width=0.3\textwidth]{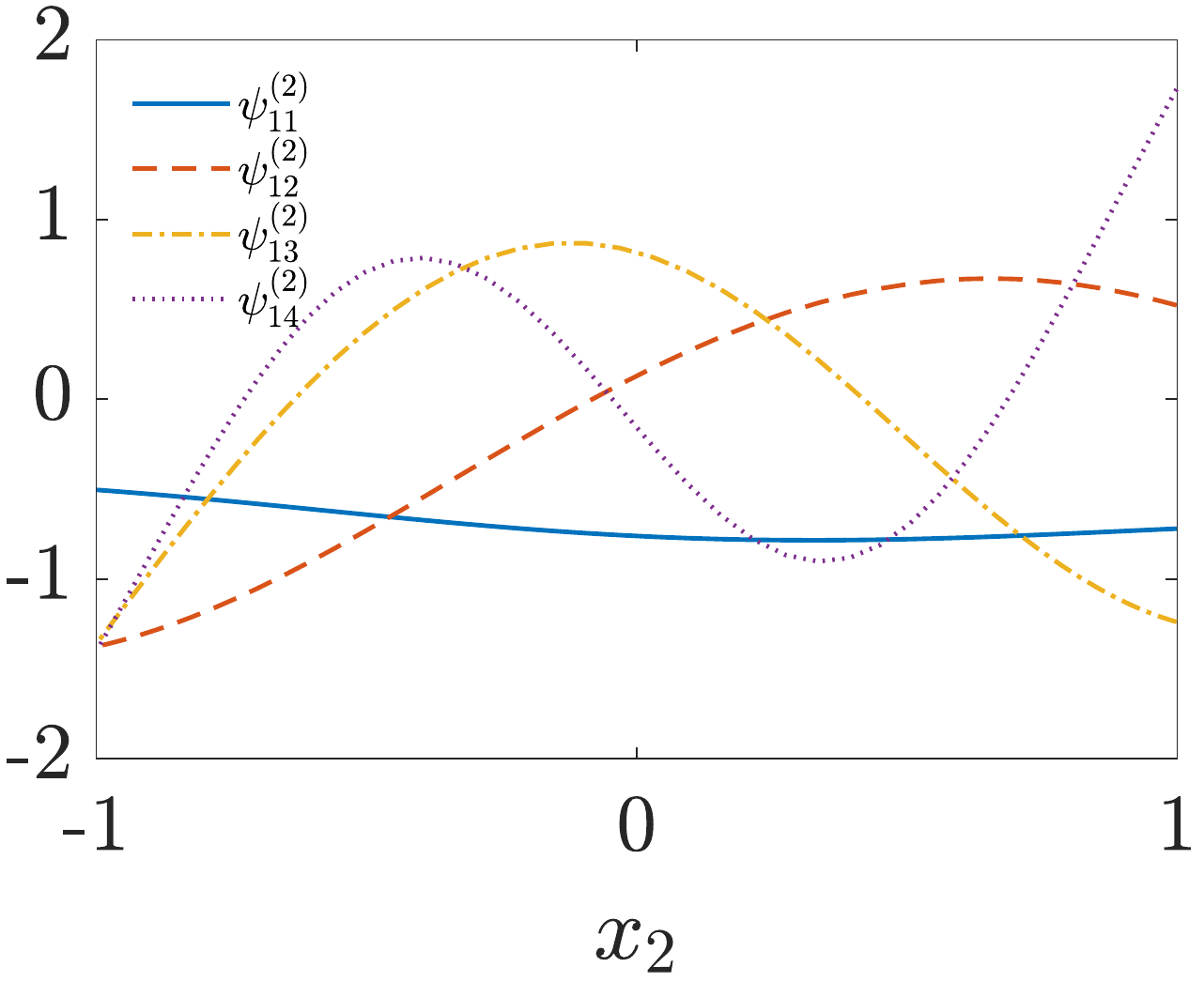}
		\includegraphics[width=0.3\textwidth]{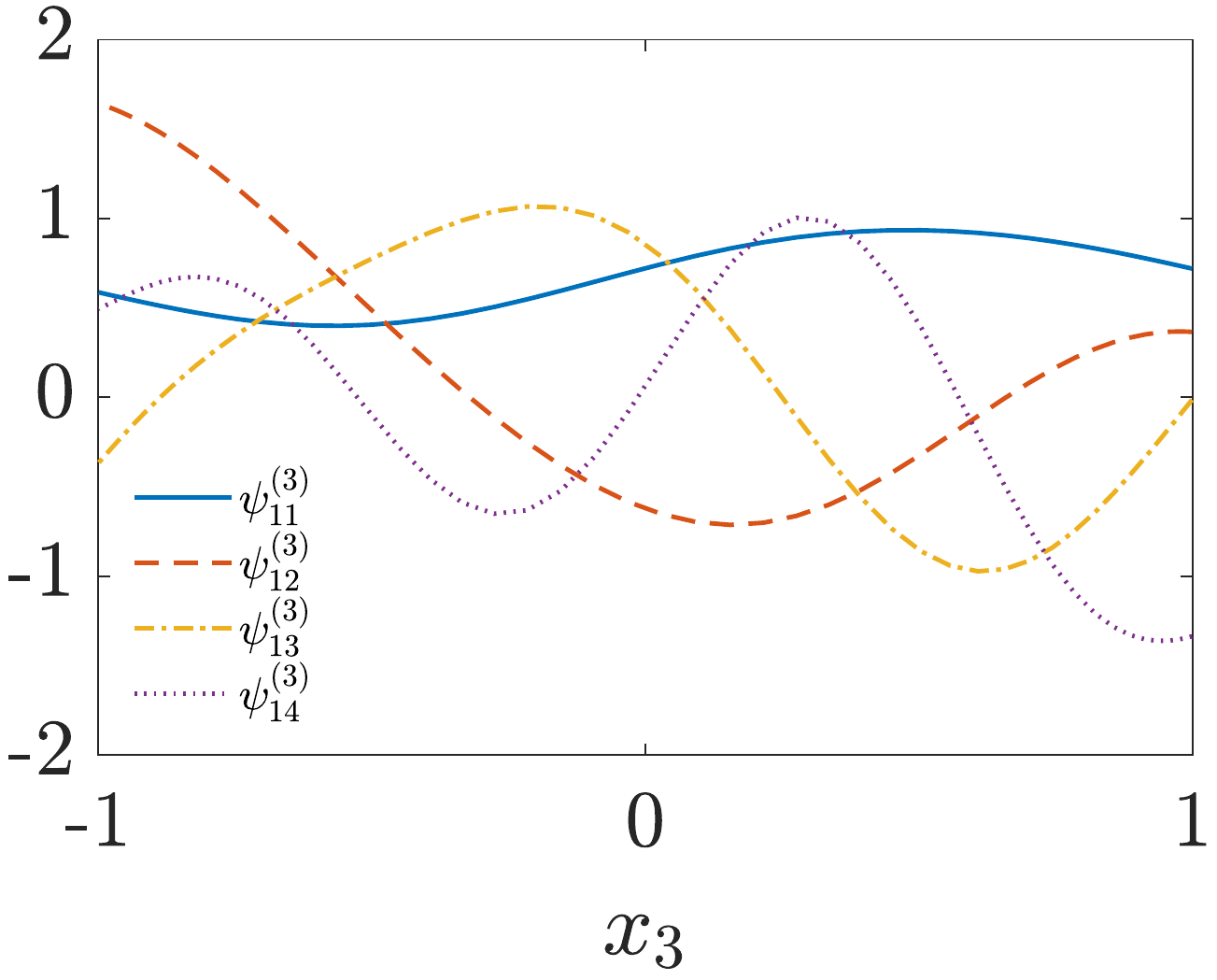}	
}
\caption{Recursive bi-orthogonal decomposition of the function \eqref{3d_example}. Shown are a few relevant modes 
$\psi^{(1)}_{i_1}(x_1)$ (a), $\psi^{(2)}_{i_1i_2}(x_2)$ (b), 
and  $\psi^{(3)}_{i_1i_2}(x_3)$ (c).}
\label{fig:3d_modes}
\end{figure}
\begin{figure}[t]
\centerline{\footnotesize\hspace{0.6cm} (a) \hspace{5cm}(b)\hspace{5cm} (c)}
	\centerline{
		\includegraphics[width=0.33\textwidth]{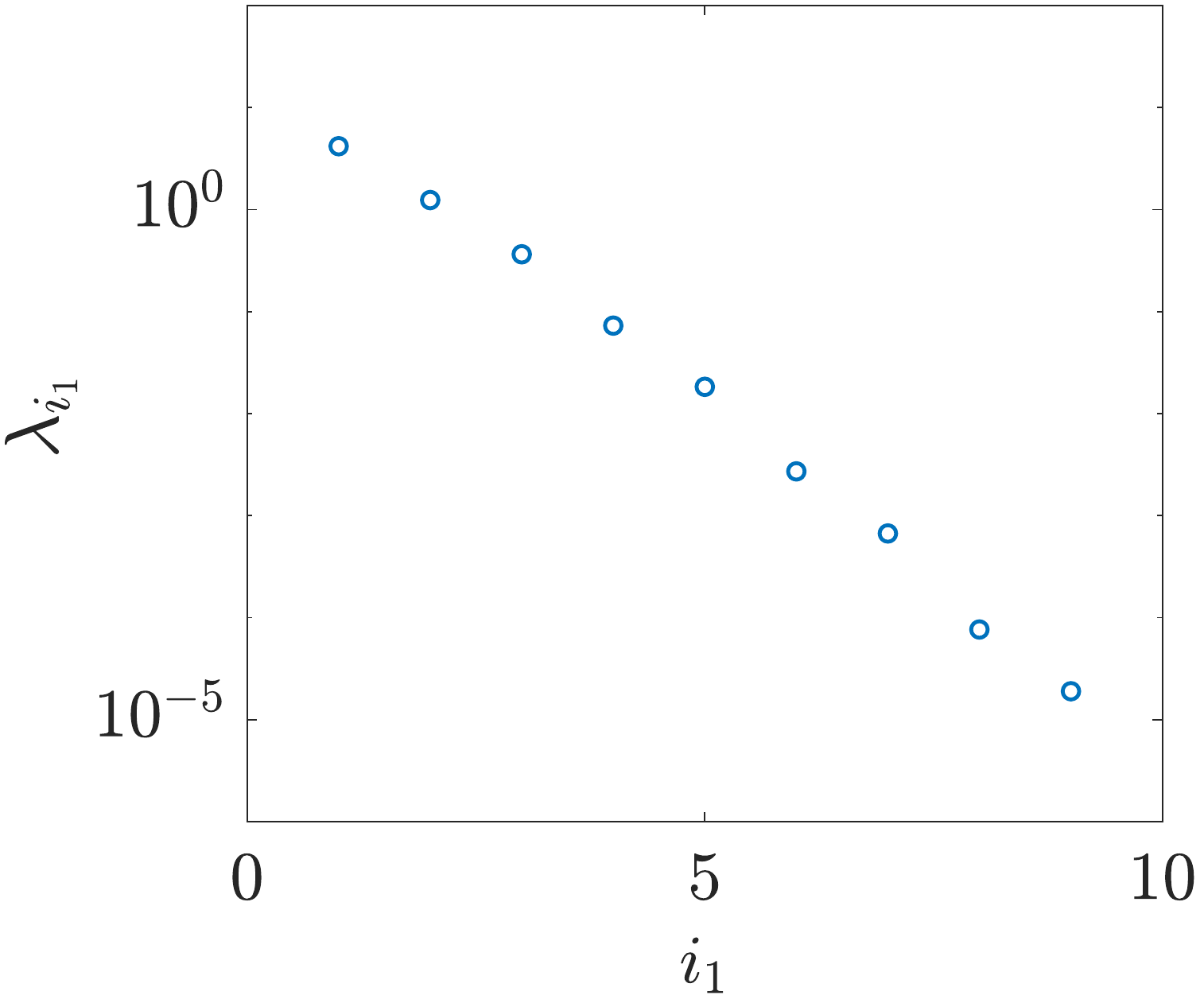}
		\includegraphics[width=0.33\textwidth]{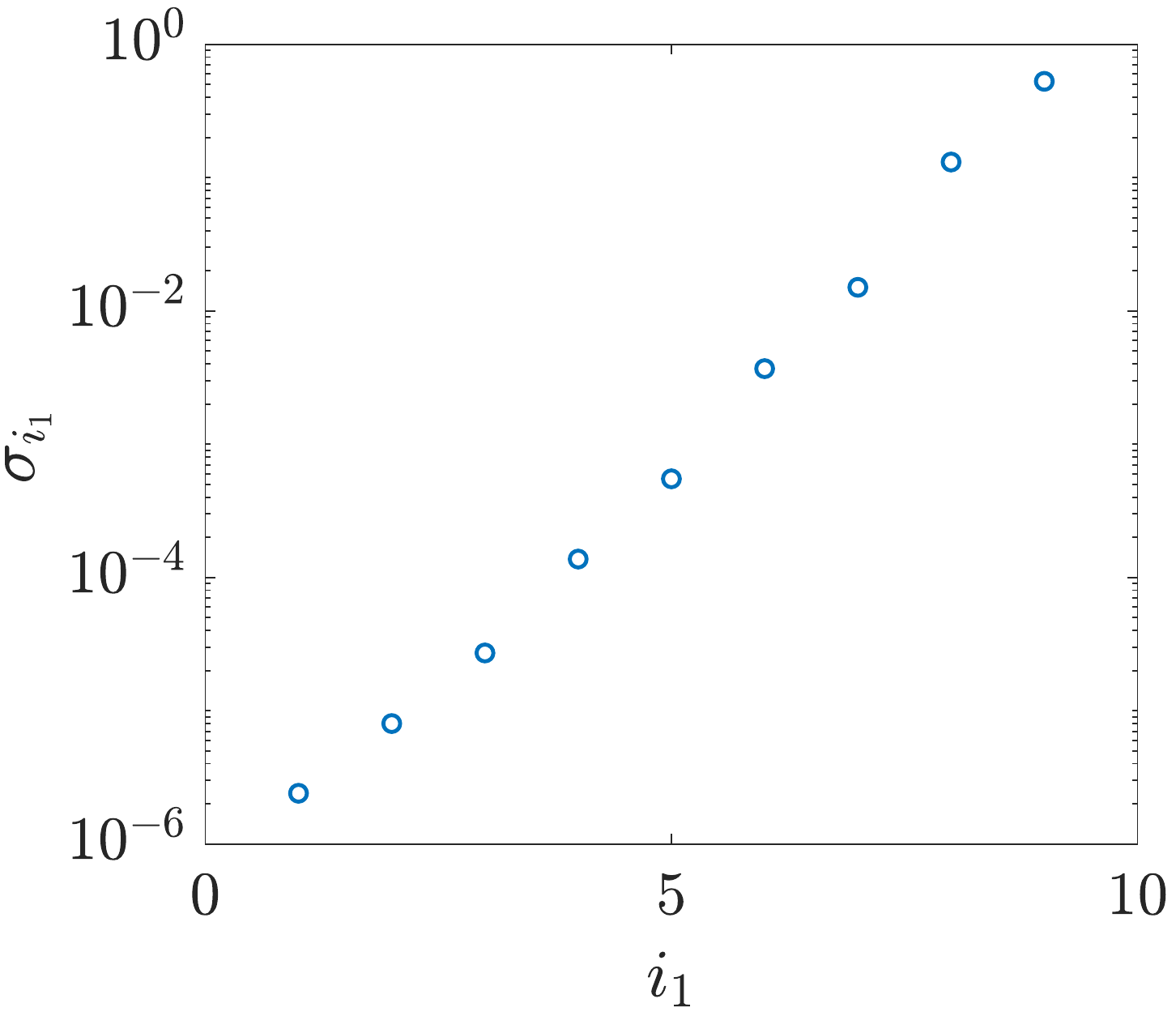}
		\includegraphics[width=0.33\textwidth]{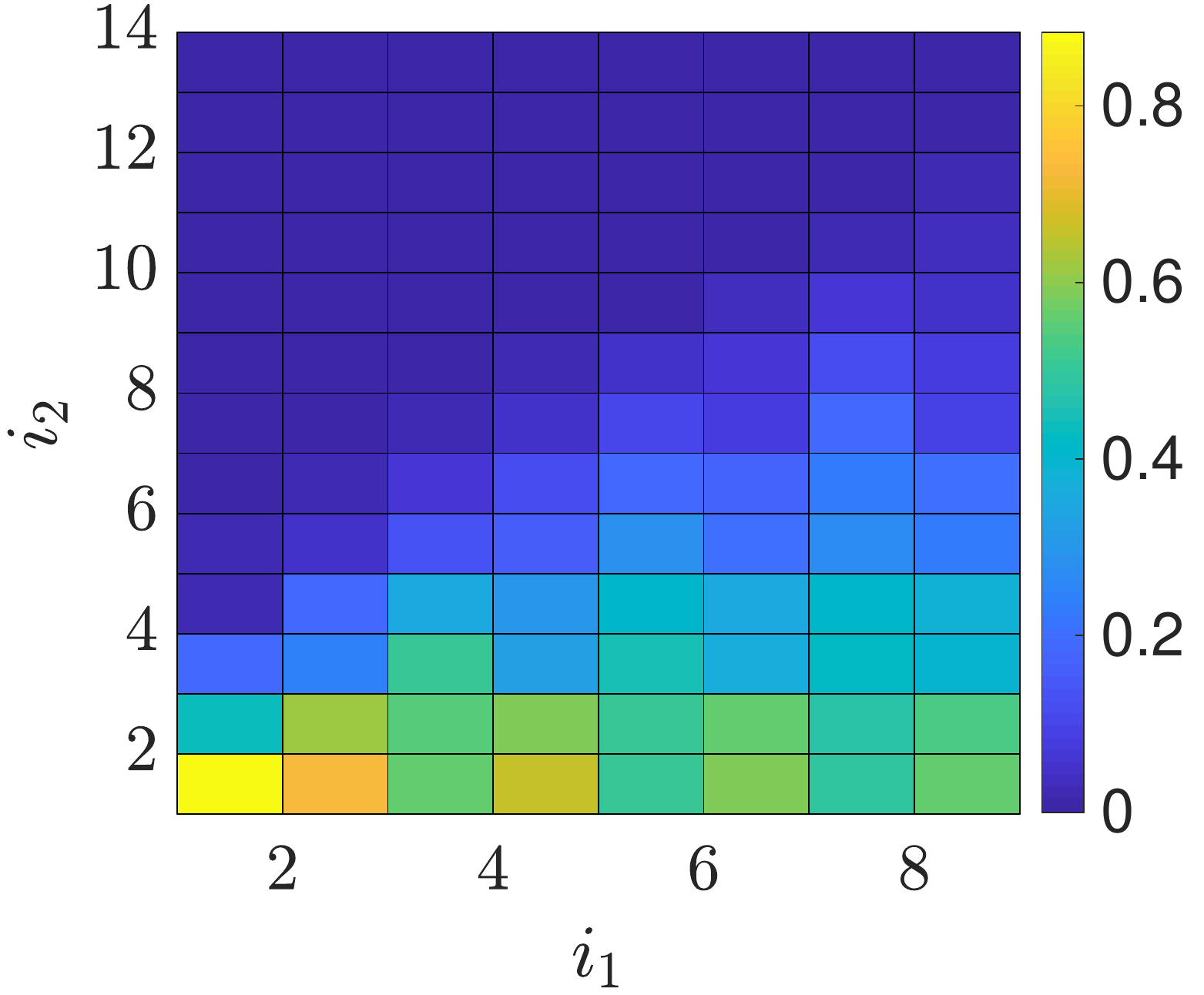}	
}
\caption{Recursive bi-orthogonal decomposition of the function \eqref{3d_example}. We plot the level-1 spectrum
$\lambda_{i_1}$ (a),  the level-2 thresholds 
corresponding to $\sigma = 10^{-5}$ (b), and the level-2 
spectrum $\lambda_{i_1i_2}$ (c).}
\label{fig:3d_approx}
\end{figure}

\noindent 
We discretize \eqref{2d_eig_fun_problem} using  
Gauss-Legendre spectral collocation with 50 points. 
This yields $50$ leading eigenvalues and corresponding 
eigenfunctions.  Following the thresholding technique discussed in Section \ref{sec:thresholding}, we set 
$\sigma = 10^{-5}$ to determine how many level-1 eigenvalues 
and eigenfunctions to keep. It turns out that only $9$ eigenvalues 
are larger than $\sigma$ which determines the first truncation rank as 
$r_1 = 9$.  These $9$ eigenvalues $\{\lambda_1, \ldots, \lambda_9\}$ 
constitute the level $1$ spectrum which is shown in 
Figure \ref{fig:3d_approx}(a). The corresponding eigenfunctions are 
$\psi^{(1)}_1, \ldots, \psi^{(1)}_9$.  The modes 
$\psi^{(2,3)}_1,\ldots,\psi^{(2,3)}_9$ can 
now be obtained through projection 
as in \eqref{eig_fun_projection}
\begin{equation}
\psi^{(2,3)}_{i_1} = \frac{1}{\lambda_{i_1}} \int_{-1}^1 
u(x_1,x_2,x_3) \psi^{(1)}_{i_1} dx_1 \ ,
\end{equation}
which we compute with Gauss-Legendre quadrature.
For each of the $9$ modes $\psi^{(2,3)}_{i_1}$ we 
follow the same procedure we used to compute 
$\psi^{(1)}_{i_1}$, i.e., we build the kernels
\begin{equation} 
l_{\psi^{(2,3)}_{i_1}}(x_2,x_2^{\prime}) = \int_{-1}^1 
\psi_{i_1}^{(2,3)}(x_2,x_3) 
\psi_{i_1}^{(2,3)}(x_2^{\prime}, x_3) dx_3
\end{equation}
and then solve the eigenvalue problems
\begin{equation}
\int_{-1}^1 l_{\psi^{(2,3)}_{i_1}}(x_2,x_2^{\prime}) 
\psi_{i_1i_2}^{(2)}(x_2^{\prime}) dx_2^{\prime} = 
\lambda_{i_1i_2}^2 \psi_{i_1i_2}^{(2)}(x_2) \ ,
\end{equation} 
to obtain $50$ eigenvalues and eigenfunctions of 
for each ${i_1} = 1, \ldots, 9$.  To decide how many eigenvalues 
and eigenfunctions to keep we use the thresholds 
$\sigma_{i_1} = \sigma/\lambda_{i_1}$ 
(see Figure \ref{fig:3d_approx}(b)). This yields the following 
vector of level-$2$ ranks $ r_2 = [11, 11, 11, 11, 11, 11, 10, 6, 0]$. 
The level-$2$ spectra $\lambda_{i_1i_2}$ is shown in 
Figure \ref{fig:3d_approx}(c) for the level-$1$ modes above 
the threshold $\sigma=10^{-5}$.

\section{Recursive bi-orthogonal decomposition of time-dependent multivariate functions}
\label{sec:time_evolution}

Let us consider the time-dependent multivariate function 
\begin{equation}
\label{time_varying_fun}
u : \Omega \times [0,T] \to \mathbb{R}.
\end{equation}
At any fixed $t \in [0,T]$ we assume that 
$u(x_1,\dots,x_d, t)$ is an element of the Sobolev 
space \eqref{Sobolev}. With the goal of solving 
high-dimensional nonlinear evolution equations (PDEs), 
in this Section we develop a recursive bi-orthogonal 
decomposition theory for \eqref{time_varying_fun}.
To decompose the function $u(x_1,\ldots,x_d,t)$ 
with recursive bi-orthogonal methods we have at least two different options. For instance, if $u$ is square integrable in $t$, 
then we may treat $t$ the same as a spatial variable $x_j$. 
In this way, we essentially include the time variable in the inner 
products \eqref{iprod1}-\eqref{iprod2}, and decompose 
the time-dependent field \eqref{time_varying_fun} 
using the methods of Section \ref{sec:recursive}. 

An alternative approach relies on introducing time 
dependence in all modes appearing in the tensor 
expansion of $u(x_1,\ldots,x_d,t)$, and then deriving 
problem-dependent evolution equations for each mode. 
To this end, one can enforce, e.g., a dynamic 
orthogonality (DO) condition or a bi-orthogonality (BO) 
condition \cite{bo1,bo2,do,do/bo_equiv,robust_do/bo}
at each level of the binary tree representing the tensor 
expansion. This generalizes the DO/BO framework 
originally proposed by Sapsis and Lermusiaux 
\cite{do}, and Cheng {\em et. al.} \cite{bo1}, which holds 
for binary trees with only one level, to binary trees with an 
arbitrary number of levels. The DO/BO method was 
originally proposed to solve 
initial/boundary value problems for nonlinear PDEs 
with parametric uncertainty modeled by the random vector 
$\bm \xi(\omega)$. The key idea was to 
introduce time-redundancy in a Karhunen-Lo\`eve-type 
expansion of the solution for the purpose of efficiently 
representing the time evolution of the stochastic modes 
and the corresponding space-time modes.  As we mentioned 
above,  the classical DO/BO expansion corresponds to  
a binary tree with only one level, where the random field 
$u(\bm x,\bm \xi,t)$ is decomposed as 
\begin{equation}
u(\bm x,\bm \xi,t) =\sum_{k=1}^\infty 
\lambda_k(t)\widehat{u}_k(\bm x,t)\Gamma_k(\bm \xi,t).
\end{equation}
In this Section we generalize this idea, and apply 
it recursively to multivariate functions not necessarily 
dependent on random parameters, until we obtain an expansion in terms of 
one-dimensional functions (whenever possible). 
To illustrate the hierarchical DO/BO method, we consider 
the TT series expansion \eqref{series_tt}, 
and introduce redundant time-dependence in all 
eigenvalues and eigenmodes. This yields the following 
representation of \eqref{time_varying_fun}
\begin{equation}
\label{series_tt_time}
u(x_1,\ldots,x_d,t) = \sum_{i_1 = 1}^{\infty} \cdots 
\sum_{i_{d-1} = 1}^{\infty} \lambda_{i_1}(t) \cdots 
\lambda_{i_1 \cdots i_{d-1}}(t) \psi_{i_1}^{(1)}(t) 
\psi_{i_1 i_2}^{(2)}(t) \cdots \psi_{i_1 \cdots i_{d-1}}^{( d-1)}(t) 
\psi_{i_1 \cdots i_{d-1}}^{(d)}(t).
\end{equation}
Other tensor expansions corresponding to different binary trees, 
e.g., the HT expansion sketched in Figure \ref{fig:trees}, 
can be generalized in a similar way.  
Hereafter, we derive the DO and BO evolution equations 
for the time-dependent modes in 
\eqref{series_tt_time}, and show that these two 
ways of propagating $u(x_1,\ldots,x_d,t)$ forward 
in time on a (smooth) low-dimensional manifold of constant 
rank \cite{h_tucker_geom} are equivalent in the sense 
that the finite dimensional function spaces 
containing the DO and BO components 
are essentially the same.  

In the following Sections every function 
is time dependent, so $t$ is omitted from 
the function arguments.  Superscripts indicate spatial 
dependencies as in Section \ref{sec:rec_bi_orth}, 
so spatial arguments of functions are also omitted 
when there is no ambiguity.  
Angled brackets $\langle \cdot , \cdot \rangle$ here 
denote the $L^2=H^{(0)}$ inner product over all 
spatial components for which the two 
arguments are defined.

\subsection{DO-TT propagator}
\label{sec:DO-TT}

To derive the Dynamically-Orthogonal Tensor-Train (DO-TT) propagator, let us first consider the level-1 expansion 
\begin{equation}
u(x_1,\ldots,x_d,t) = \sum_{i_1 = 1}^{r_1} \psi_{i_1}^{(1)}(t)\Psi_{i_1}^{(2,\ldots,d)}(t) \ ,
\label{level1TT}
\end{equation}
where 
\begin{equation}
\Psi_{i_1}^{(2,\ldots,d)}(t)= \lambda_{i_1}(t)\psi_{i_1}^{(2,\ldots,d)}(t).
\end{equation}
By differentiating \eqref{level1TT}  with respect to time we obtain
\begin{equation}
\label{begin_der_0}
 \frac{\partial u}{\partial t} =  \sum_{i_1 = 1}^{r_1} \frac{\partial \psi_{i_1}^{(1)}}{\partial t} \Psi_{i_1}^{(2,\ldots, d)} + \psi_{i_1}^{(1)} \frac{\partial \Psi_{i_1}^{(2, \ldots, d)}}{\partial t}.
\end{equation}
Clearly, if $u(x_1,\ldots,x_d,t)$ is given then 
$\partial u/\partial t$ is known, and therefore the left hand 
side of \eqref{begin_der_0} is fully determined. On the other 
hand, in the context of nonlinear evolution equations, 
$\partial u/\partial t$ is represented by the right 
hand side of the PDE.
Applying $\langle \cdot , \psi_{k_1}^{(1)} \rangle$ 
to \eqref{begin_der_0} and utlizing the  
DO conditions (see \cite{do})
\begin{align}
\label{dynamic_orthogonality0}
\langle \frac{\partial \psi_{i_1}^{(1)}}{\partial t}, 
\psi_{k_1}^{(1)} \rangle = 0\qquad \text{for all} \quad 
i_1,k_1=1,2,\ldots
\end{align}
yields
\begin{equation}
\label{phi_ev_0}
    \frac{\partial \Psi_{k_1}^{(2, \ldots, d)}}{\partial t} = 
    \underbrace{\langle \frac{\partial u}{\partial t}, \psi_{k_1}^{(1)} \rangle}_{N_{k_1}^{(2,\ldots,d)}(t)}.
\end{equation}
The evolution equations for the modes 
$\psi_{k_1}^{(1)}$ can be obtained by applying 
$\langle \cdot, \Psi_{k_1}^{(2, \ldots d)} \rangle$ 
to \eqref{begin_der_0} and \eqref{phi_ev_0}, and the 
using simple algebra. This yields 
\begin{equation}
 \sum_{i_1 = 1}^{r_1} \frac{\partial \psi_{i_1}^{(1)}}{\partial t} 
 \underbrace{\langle \Psi_{i_1}^{(2, \ldots, d)}, 
 \Psi_{k_1}^{(2,\ldots,d)} \rangle}_{C_{i_1k_1}(t)}
= \underbrace{\langle \frac{\partial u}{\partial t}, 
\Psi_{k_1}^{(2, \ldots, d)} \rangle - \sum_{i_1 = 1}^{r_1} 
\psi_{i_1}^{(1)} \langle \frac{\partial u}{\partial t} , 
\Psi_{k_1}^{(2,\ldots, d)}\psi_{i_1}^{(1)} \rangle}_{M_{k_1}^{(1)}(t)}.
\label{psi_ev_0}
\end{equation}
Equations \eqref{phi_ev_0} and \eqref{psi_ev_0} 
can be conveniently expressed in a matrix-vector form as
\begin{equation}
\begin{aligned}
\bm C(t) \frac{\partial \bm \psi^{(1)}}{\partial t}  =\bm M^{(1)}(t), \qquad 
\frac{\partial \bm \Psi^{(2,\ldots,d)}}{\partial t}  = \bm N^{(2,\ldots,d)}(t).
\end{aligned}
\label{DO1}
\end{equation}
where
\begin{equation}
\bm \psi^{(1)}  =\left[
\begin{array}{c}
\psi^{(1)}_1\\
\vdots\\
 \psi^{(1)}_{r_1}
\end{array}
\right],\qquad
\bm \Psi^{(2,\ldots, d)} =
\left[
\begin{array}{c}
\Psi^{(2,\ldots,d)}_1\\
 \vdots\\ 
\Psi^{(2,\ldots,d)}_{r_1}
\end{array}
\right],
\end{equation}
\begin{equation}
\bm M^{(1)}(t)=
\left[
\begin{array}{c}
M_1^{(1)}(t)\\
\vdots\\
M_{r_1}^{(1)}(t)
\end{array}
\right],\qquad 
\bm N^{(2,\ldots,d)}(t)= 
\left[
\begin{array}{c}
N^{(2,\ldots,d)}_1(t)\\
\vdots\\
N^{(2,\ldots,d)}_{r_1}(t)
\end{array}
\right].
\end{equation}
%
%\begin{equation}
%\bm \psi^{(1)} = \left[\psi^{(1)}_1, \ldots, \psi^{(1)}_{r_1}\right]^T, \qquad
%\bm \Psi^{(2,\ldots, d)} = \left[\Psi^{(2,\ldots,d)}_1, \ldots, 
%\Psi^{(2,\ldots,d)}_{r_1}\right]^T,
%\end{equation}
%\begin{equation}
%\begin{aligned}
%\bm N^{(2,\ldots,d)}(t)= \left[N^{(2,\ldots,d)}_1(t), \ldots,  
%N^{(2,\ldots,d)}_{r_1}(t)\right]^T,\qquad 
%\bm M^{(1)}(t)= \left[M_1^{(1)}(t), \ldots,  M_{r_1}^{(1)}(t)\right]^T.
%\end{aligned}
%\end{equation}
Note that $\bm M^{(1)}(t)$ and $\bm C(t)$ depend 
on $\Psi^{(2,\ldots,d)}$ and $\partial u/\partial t$, while 
$\bm N^{(2,\ldots,d)}(t)$ depends on $\bm \psi^{(1)}$ and 
$\partial u/\partial t$. Therefore, given 
$\partial u/\partial t$ we have that the 
system \eqref{DO1} is closed. 
At this point, we move to the next level of the TT binary 
tree and derive DO evolution equations for the level-$2$ modes. 
By following the same steps as in the derivation of 
the level-$1$ system \eqref{DO1}, we obtain
\begin{equation}
    \begin{aligned}
    \frac{\partial \Psi_{k_1 k_2}^{(3, \ldots, d)}}{\partial t}& = 
    \underbrace{\langle N^{(2,\ldots,d)}_{k_1} , 
    \psi_{k_1 k_2}^{(2)} \rangle}_{N_{k_1 k_2}^{(3,\ldots,d)}(t)} \ ,\\
    \sum_{i_2 = 1}^{r_2} \frac{ \partial \psi_{k_1 i_2}^{(2)}}{\partial t} \langle \Psi_{k_1 i_2}^{(3, \ldots, d)} , \Psi_{k_1 k_2}^{(3, \ldots, d)} \rangle &= \underbrace{\langle N^{(2,\ldots,d)}_{k_1} , \Psi_{k_1 k_2}^{(3, \ldots, d)} \rangle - \sum_{i_2 = 1}^{r_2} \psi_{k_1 i_2}^{(2) } 
    \langle N^{(2,\ldots,d)}_{k_1} , \Psi_{k_1 k_2}^{(3, \ldots, d)}\psi_{k_1 i_2}^{(2)} \rangle}_{ M_{k_1 k_2}^{(2)}(t)} \ , 
    \end{aligned}
\end{equation}
where we defined 
\begin{equation}
\Psi_{k_1 k_2}^{(3, \ldots, d)}=\lambda_{k_1k_2}(t)\psi_{k_1 k_2}^{(3, \ldots, d)}(t).
\end{equation}
Proceeding recursively, it is possible to obtain evolution 
equations for each mode in the TT binary tree sketched in Figure \ref{fig:trees}. Specifically, we have 
    \begin{align}
        \frac{\partial \Psi_{k_1 \cdots k_j}^{(j+1, \ldots, d)}}{\partial t} =& \underbrace{\langle N^{(j, \ldots, d)}_{k_1 \cdots k_{j-1}} , \psi_{k_1 \cdots k_j}^{(j)} \rangle}_{N_{k_1 \cdots k_j}^{(j+1, \ldots, d)}(t)} \\
   \sum_{i_j = 1}^{r_j} \frac{ \partial \psi^{(j)}_{k_1 \cdots k_{j-1} i_j}}{\partial t} &\langle \Psi_{k_1 \cdots k_{j-1} i_j}^{(j+1, \ldots, d)} , \Psi_{k_1 \cdots k_j}^{(j+1, \ldots, d)} \rangle  \label{DOTT2}\\
   & = \underbrace{\langle N_{k_1 \cdots k_{j-1}}, \Psi_{k_1 \cdots k_j}^{(j+1, \ldots, d)} \rangle - \sum_{i_j = 1}^{r_j} \psi_{k_1 \cdots k_{j-1} i_j} \langle  N^{(j,\ldots,d)}_{k_1 \cdots k_{j-1}} , \Psi_{k_1 \cdots k_j}^{(j+1, \ldots, d)} \psi_{k_1 \cdots k_{j-1} i_j}^{(j)} \rangle}_{ M_{k_1 \cdots k_{j}}^{(j)}(t)}.\label{DOTT3}
    \end{align}
Equations \eqref{DOTT2}-\eqref{DOTT3} can be conveniently expressed 
in a matrix-vector form as 
\begin{equation}
\begin{aligned}
\bm C_{i_1\cdots i_{j-1}}(t) \frac{\partial \bm \psi_{i_1\cdots i_{j-1}}^{(j)}}{\partial t}  =\bm M^{(j)}_{i_1\cdots i_{j-1}}(t), \qquad 
\frac{\partial \bm \Psi^{(j+1,\ldots,d)}_{i_1\cdots i_j}}{\partial t}  = \bm N^{(j+1,\ldots,d)}_{k_1\cdots k_j}(t).
\end{aligned}
\label{do_ev_eq}
\end{equation}
where, 
\begin{equation}
 \bm \psi_{i_1\cdots i_{j-1}}^{(j)}  =\left[
\begin{array}{c}
 \psi_{i_1\cdots i_{j-1}1}^{(j)}\\
\vdots\\
 \psi_{i_1\cdots i_{j-1}r_j}^{(j)}
\end{array}
\right],\qquad
\bm \Psi^{(j+1,\ldots,d)}_{i_1\cdots i_j}=
\left[
\begin{array}{c}
\Psi^{(j+1,\ldots,d)}_{i_1\cdots i_j 1 }\\
 \vdots\\ 
\Psi^{(j+1,\ldots,d)}_{i_1\cdots i_j r_j }
\end{array}
\right],
\end{equation}
\begin{equation}
\bm  M^{(j)}_{i_1\cdots i_{j-1}}(t)=
\left[
\begin{array}{c}
M^{(j)}_{i_1\cdots i_{j-1}1}(t)\\
\vdots\\
M^{(j)}_{i_1\cdots i_{j-1}r_j}(t)
\end{array}
\right],\qquad 
\bm N^{(j+1,\ldots,d)}_{k_1\cdots k_j}(t)= 
\left[
\begin{array}{c}
N^{(j+1,\ldots,d)}_{k_1\cdots k_j 1}(t)\\
\vdots\\
N^{(j+1,\ldots,d)}_{k_1\cdots k_j r_j}(t)
\end{array}
\right].
\end{equation}

\paragraph{Remark} The ``non-leaf'' modes 
$\psi_{k_1 \cdots k_j}^{(j+1, \ldots, d)}$ can be constructed at 
any time in terms of the ``leaf'' modes as
\begin{equation}
\psi_{k_1 \cdots k_j}^{(j+1, \ldots, d)} = \sum_{i_{j+1} = 1}^{r_{j+1}} 
\cdots \sum_{i_{d-1} = 1}^{r_{d-1}} \psi_{k_1 \cdots k_j i_{j+1}}^{(j + 1)} 
\cdots  \psi_{k_1 \cdots k_j i_{j+1} \cdots i_{d-1}}^{(d-1)} \psi_{k_1 
\cdots k_j i_{j+1} \cdots i_{d-1}}^{(d)}
\end{equation}
To build the time-dependent multivariate function \eqref{series_tt_time} 
it is sufficient to integrate only the evolution equations 
corresponding to the leaf modes.

\subsection{An example: DO-TT decomposition of a time-dependent 3D function}
Let us consider the time-dependent multivariate function
\begin{equation}
\label{3d_ev_example}
u(x_1,x_2,x_3,t) = (t+1) x_2 x_3 + (t^2-10) x_1 x_3 - 
(4 \sin(t) + 3) x_1 x_2 x_3,\qquad (x_1,x_2,x_3)\in \Omega.
\end{equation}
where  $\Omega= [-1,1]^{3}$ is the standard three-dimensional 
cube. We are interested in decomposing 
\eqref{3d_ev_example} with the DO-TT method. 
To this end, we consider the $L^2$ inner product 
and first perform a recursive bi-orthogonal decomposition 
of the initial state $u(x_1,x_2,x_3,0)$ using the methods of Section 
\ref{sec:recursive}.  To this end, we consider $50$ 
Gauss-Legendre quadrature points and set the 
eigenvalue threshold to $\sigma = 10^{-5}$. By following 
the same steps as in the example \ref{sec:ex_1},  
this yields multivariate ranks $r_1=2$ and $r_2=[1\, 1]$. 
This allows us to approximate the initial condition as 
\begin{equation}
u(x_1,x_2,x_3,0) \simeq \psi^{(1)}_1(0)\psi^{(2)}_{11}(0)\psi^{(3)}_{11}(0) + \psi^{(1)}_2(0)\psi^{(2)}_{21}(0)\psi^{(3)}_{21}(0).
\end{equation}
The time derivative of $u$ is easily obtained as
\begin{equation}
\frac{\partial u}{\partial t} = x_2 x_3 + 2t x_1 x_3 - 4 \cos(t) x_1 x_2 x_3.
\label{ic}
\end{equation}
A substitution of \eqref{ic} intro the multi-level DO evolution 
equations \eqref{do_ev_eq} yields 
\begin{align}
\frac{\partial \psi^{(1)}_k}{\partial t} = &\langle x_2 x_3 \psi^{(2,3)}_{k} \rangle + 2tx_1 \langle x_3 \psi_k^{(2,3)} \rangle - 4 \cos(t) x_1 \langle x_2 x_3 \psi^{(2,3)}_k \rangle \nonumber\\
&- \sum_{i = 1}^{r_1} \psi_i^{(1)} \left[ \langle \psi^{(1)}_i \rangle \langle x_2 x_3 \psi^{(2,3)}_k \rangle + 2t \langle x_1 \psi^{(1)}_i \rangle \langle x_3 \psi^{(2,3)}_k \rangle - \right. \nonumber\\
&\left. 4 \cos(t) \langle x_1 \psi^{(1)}_i \rangle \langle x_2 x_3 \psi^{(2,3)}_k \rangle \right] \ ,\label{do_ev_3d1}\\
\frac{\partial \psi^{(2)}_{k1}}{\partial t} \langle \psi^{(3)}_{k1}  \psi^{(3)}_{k1}\rangle = &x_2 \langle x_3 \psi^{(3)}_{k1} \rangle \langle \psi^{(1)}_k \rangle + 2t \langle x_3 \psi^{(3)}_{k1} \rangle \langle x_1 \psi^{(1)}_k \rangle - 4 \cos(t) x_2 \langle x_3 \psi^{(3)}_{k1} \rangle \langle x_1 \psi^{(1)}_k \rangle \nonumber\\
&- \psi^{(2)}_{k1} \left[ \langle x_2 \psi^{(2)}_{k1} \rangle \langle x_3 \psi^{(3)}_{k1} \rangle \langle \psi^{(1)}_{k} \rangle + 2t \langle \psi^{(2)}_{k1} \rangle \langle x_3 \psi^{(3)}_{k1} \rangle \langle x_1 \psi^{(1)}_k \rangle   \right.\nonumber \\
&\left. - 4 \cos(t) \langle x_2 \psi^{(2)}_{k 1} \rangle \langle x_3 \psi^{(3)}_{k1} \rangle \langle x_1 \psi^{(1)}_k \rangle \right] \ , \label{do_ev_3d2}\\
\frac{\partial \psi^{(3)}_{k1}}{\partial t } = &x_3 \langle x_2 \psi^{(2)}_{k 1} \rangle \langle \psi^{(1)}_k \rangle + 2t x_3 \langle \psi^{(2)}_{k1} \rangle \langle x_1 \psi^{(1)}_{k} \rangle - 4 \cos(t) x_3 \langle x_2 \psi^{(2)}_{k1} \rangle \langle x_1 \psi^{(1)}_k \rangle \ , \label{do_ev_3d3}
\end{align}
for $ j = 1,2$.
In a spectral collocation setting, all modes are represented 
by their values at the 50 Gauss-Legendre collocation points. 
Hence, the DO evolution equations reduce to a system 
of ODEs. All integrals in 
\eqref{do_ev_3d1}-\eqref{do_ev_3d3} are computed by using 
the one-dimensional Gauss-Legendre quadrature rule.
In Figure \ref{fig:3d_ev_modes} we plot the temporal evolution of 
all $6$ modes appearing in the 
system \eqref{do_ev_3d1}-\eqref{do_ev_3d3}. 
The time-dependent $L^2(\Omega)$ error between 
the DO-TT expansion 
\begin{equation}
u(t,x_1,x_2,x_3) =\psi^{(1)}_1(t)\psi^{(2)}_{11}(t)
\psi^{(3)}_{11}(t) + \psi^{(1)}_2(t)\psi^{(2)}_{21}(t)
\psi^{(3)}_{21}(t),
\label{DOTT}
\end{equation}
and the function \eqref{3d_ev_example} is plotted 
in Figure \ref{fig:3d_ev_error} versus time.
We notice that \eqref{3d_ev_example} is a separable 
function with rank $r_1 = 2$, $r_2 = [1,1]$ at each 
time $t$. This means that it can be represented exactly 
on a low-dimensional tensor manifold of constant rank at 
each time. This is the reason why the $L^2$ error in Figure 
\ref{fig:3d_ev_error} is of order $10^{-8}$ for all $t\geq 0$. 
In general, the multivariate function $u(x_1,\ldots,x_d,t)$ 
is not separable and may become rougher/wavier 
as time increases. In these cases, the multivariate 
rank $(r_1, \ldots, r_{d-1})$ of the DO-TT expansion 
usually needs to be increased in time to accurately 
represent represent $u(x_1,\ldots,x_d,t)$. Methods for increasing rank 
will be addressed in Section \ref{sec:add_modes}
in the context of adaptive DO-TT approximation of the 
solution to high-dimensional PDEs.

\begin{figure}[t]
\centerline{\footnotesize\hspace{0.4cm}$t=0.0$ \hspace{4cm} $t=2.5$  \hspace{4.0cm} $t=5.0$}
	\centerline{
	\rotatebox{90}{\hspace{1.3cm}\footnotesize}
		\includegraphics[width=0.3\textwidth]{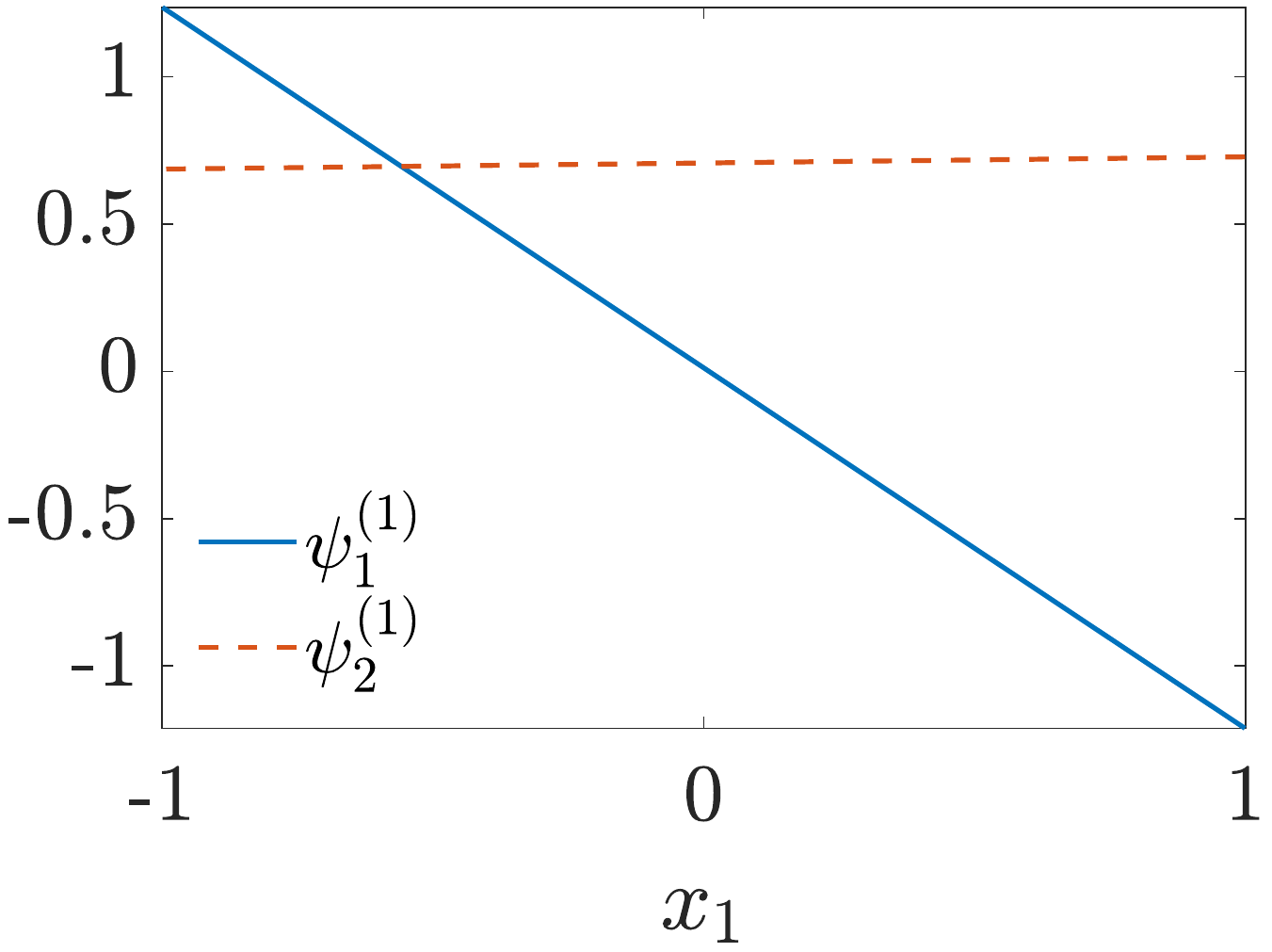}
		\includegraphics[width=0.3\textwidth]{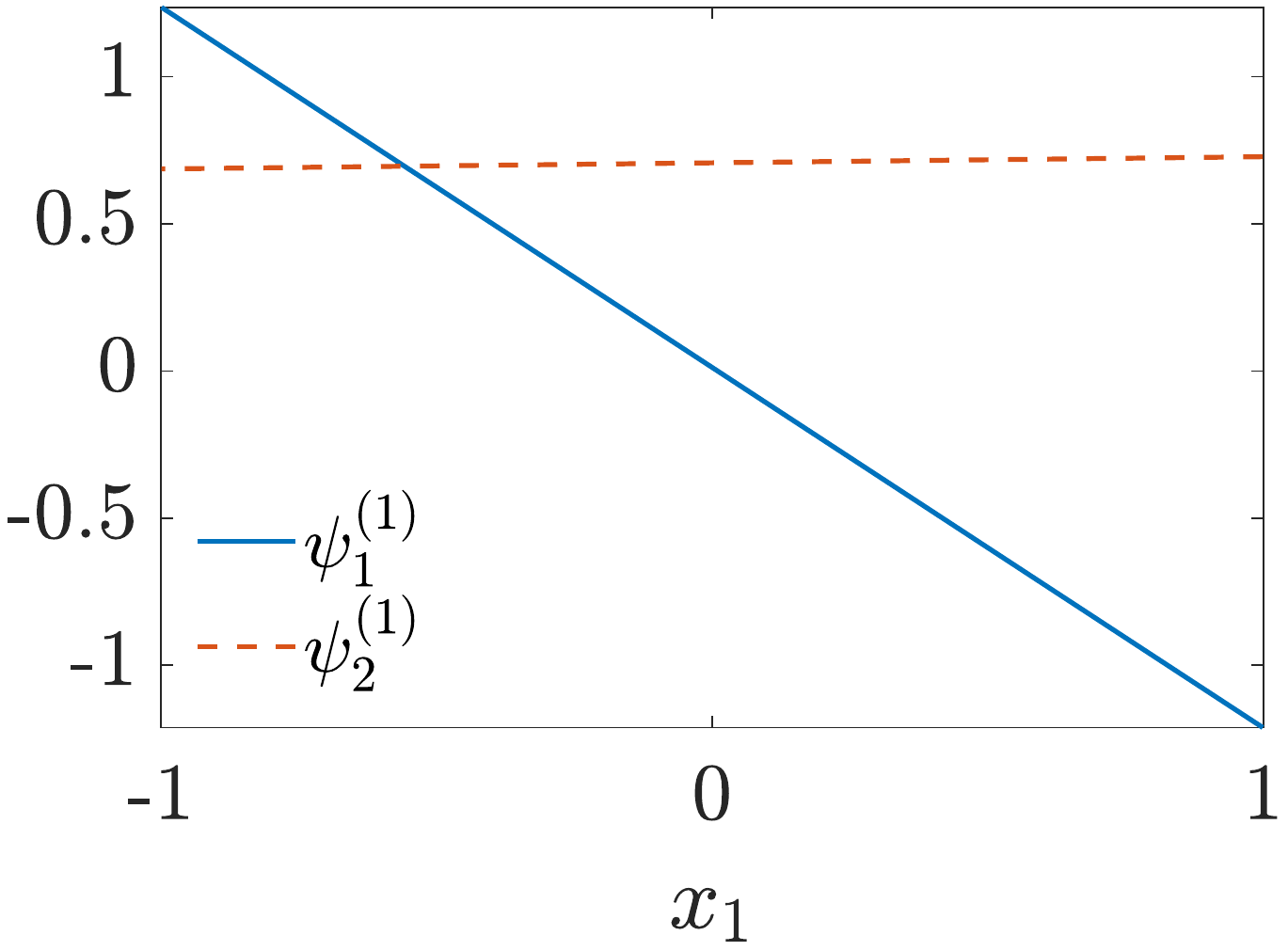}
		\includegraphics[width=0.3\textwidth]{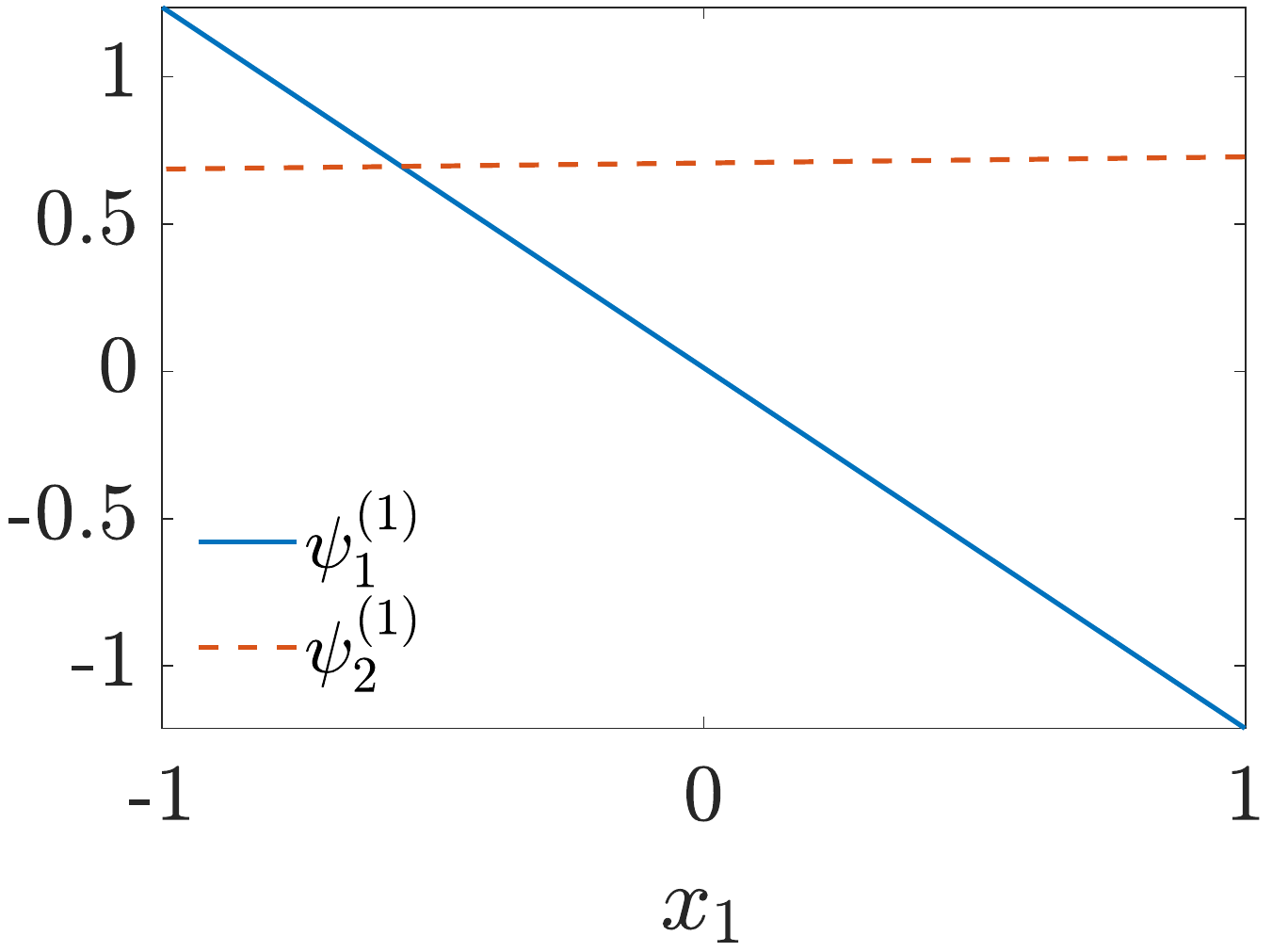}	
}
	\centerline{
	\rotatebox{90}{\hspace{1.3cm} \footnotesize }
		\includegraphics[width=0.3\textwidth]{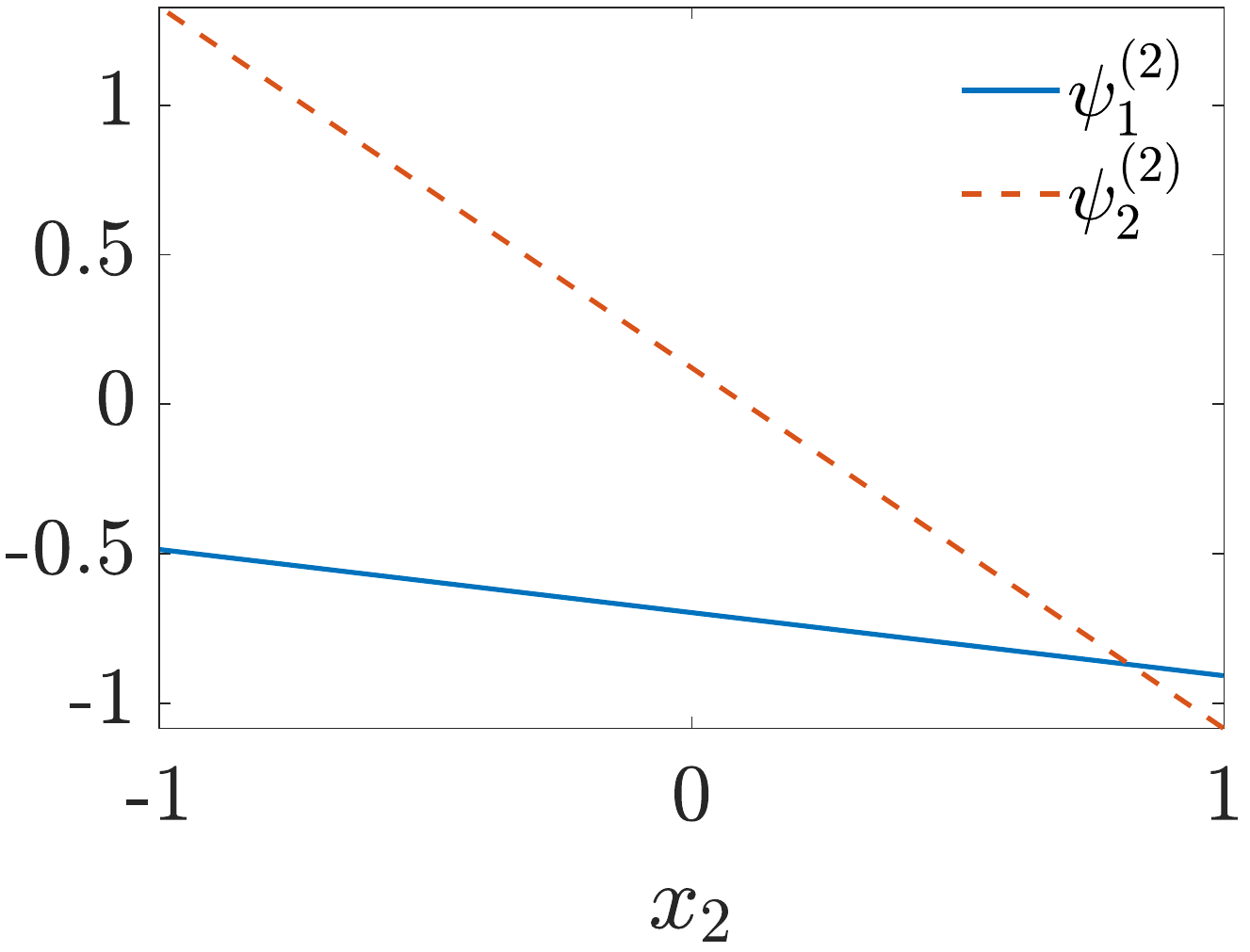}
		\includegraphics[width=0.3\textwidth]{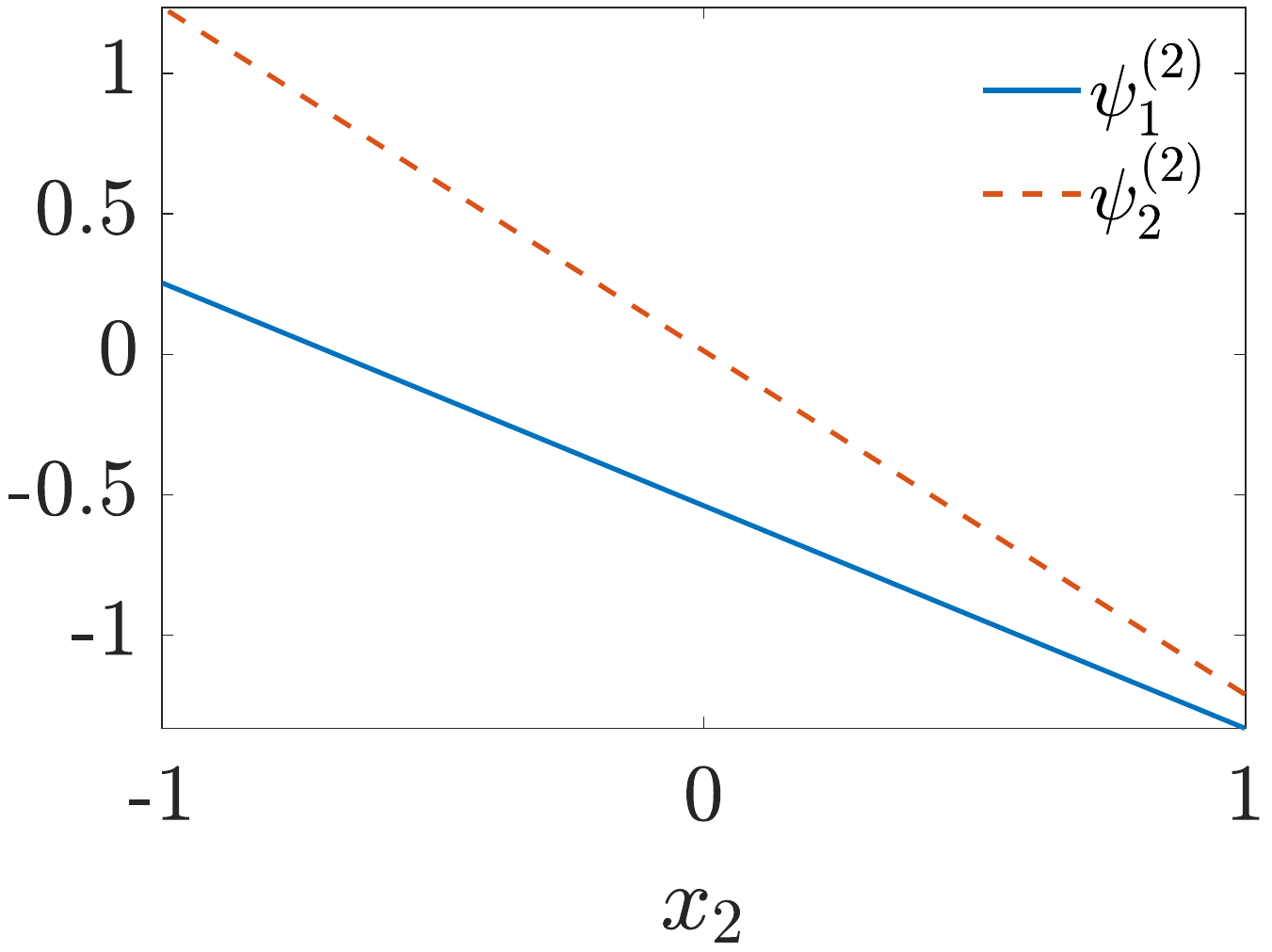}
		\includegraphics[width=0.3\textwidth]{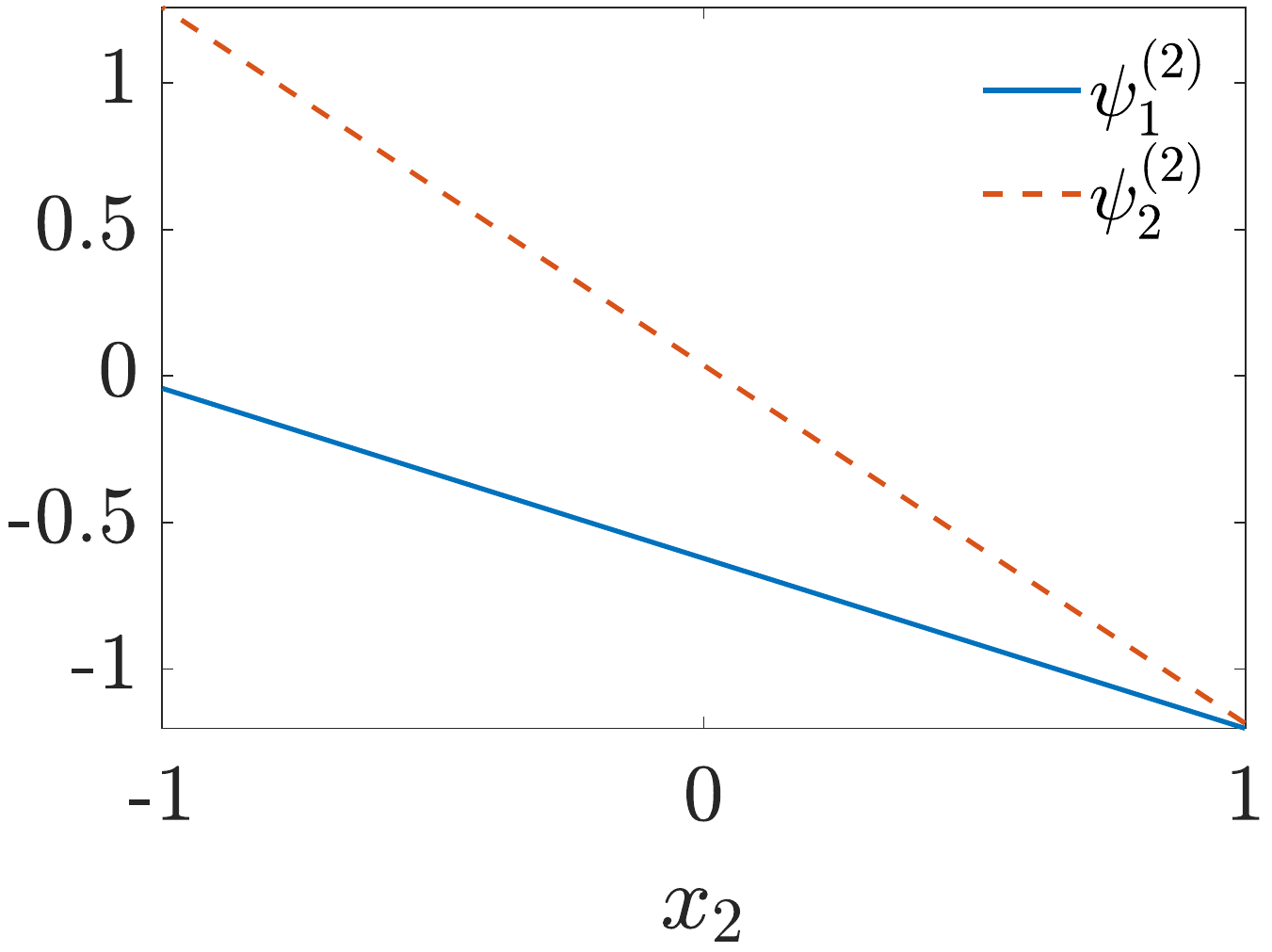}	
}
	\centerline{
	\rotatebox{90}{\hspace{1.3cm} \footnotesize }
		\includegraphics[width=0.3\textwidth]{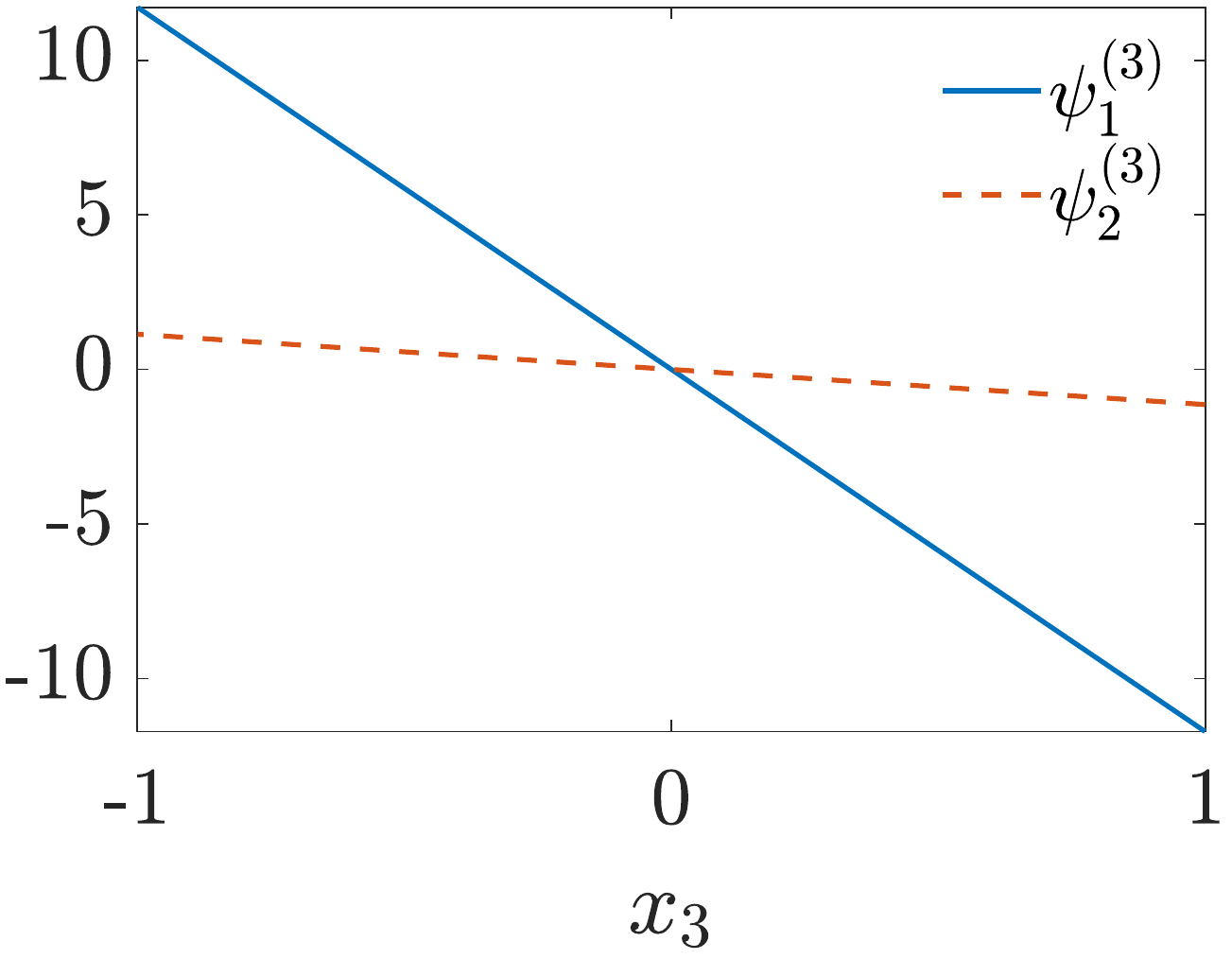}
		\includegraphics[width=0.3\textwidth]{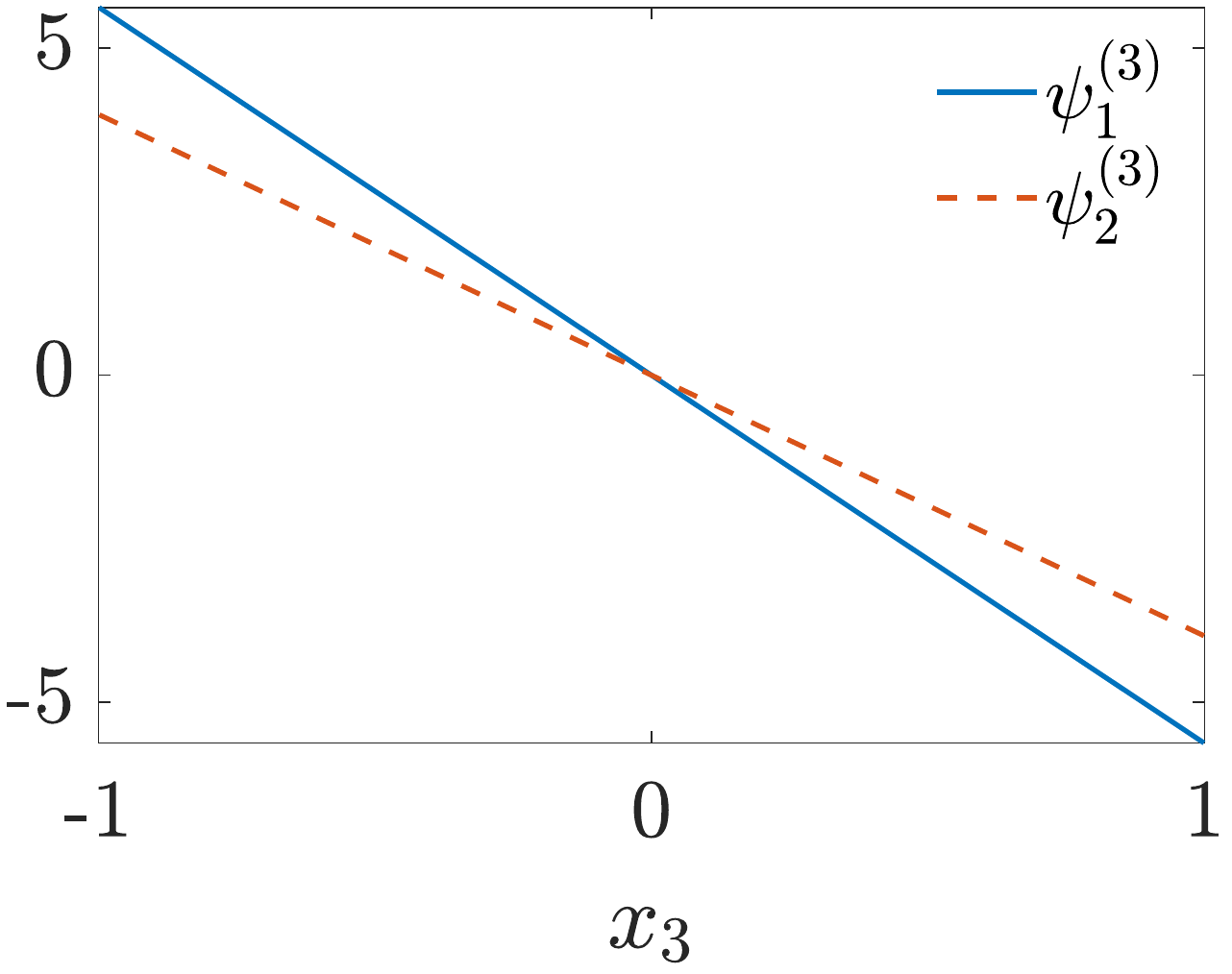}
		\includegraphics[width=0.3\textwidth]{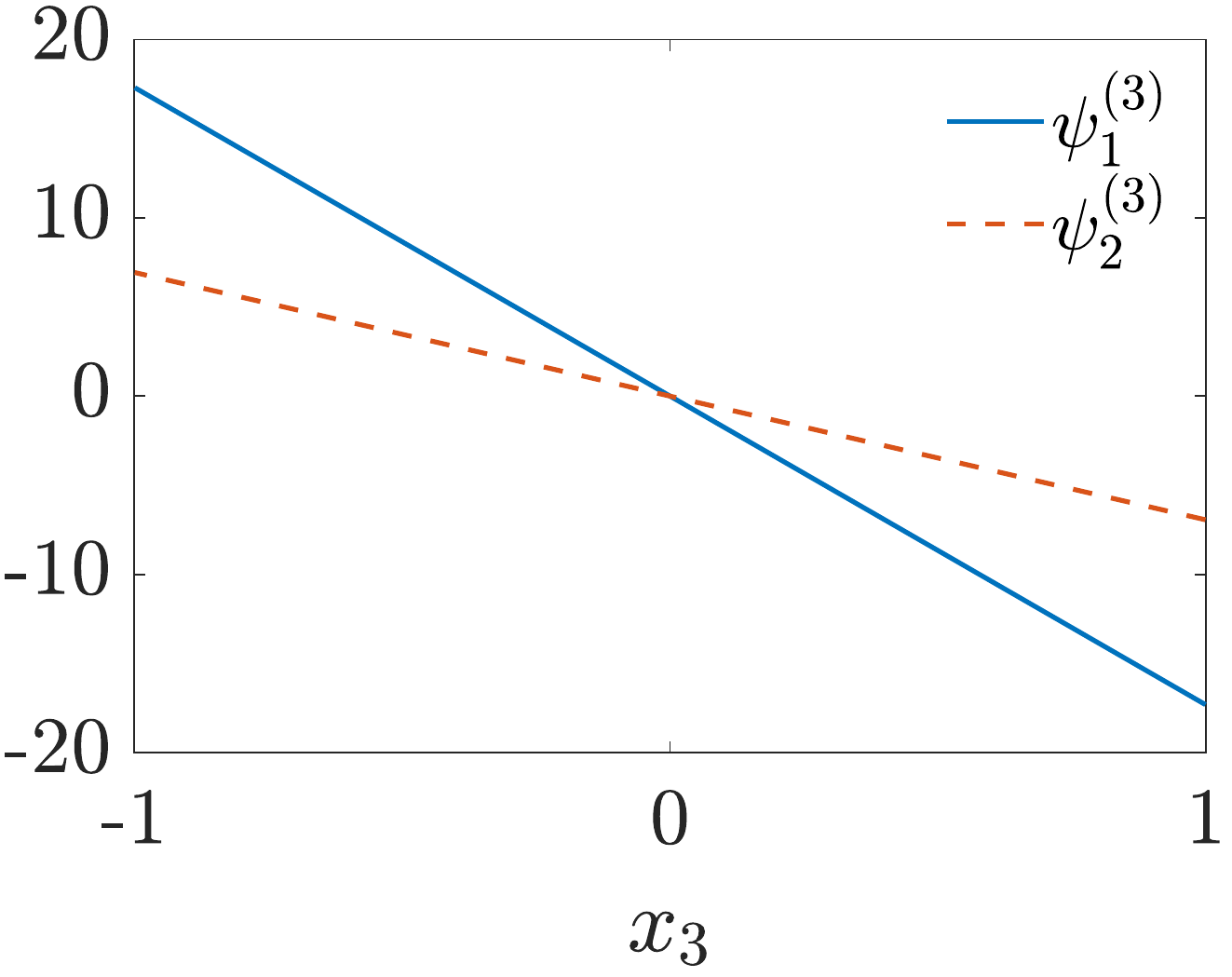}
}
\caption{Temporal evolution of the bi-orthogonal modes in the DO-TT expansion \eqref{DOTT}.}
\label{fig:3d_ev_modes}
\end{figure}

\begin{figure}
\centerline{
\includegraphics[width=0.6\textwidth]{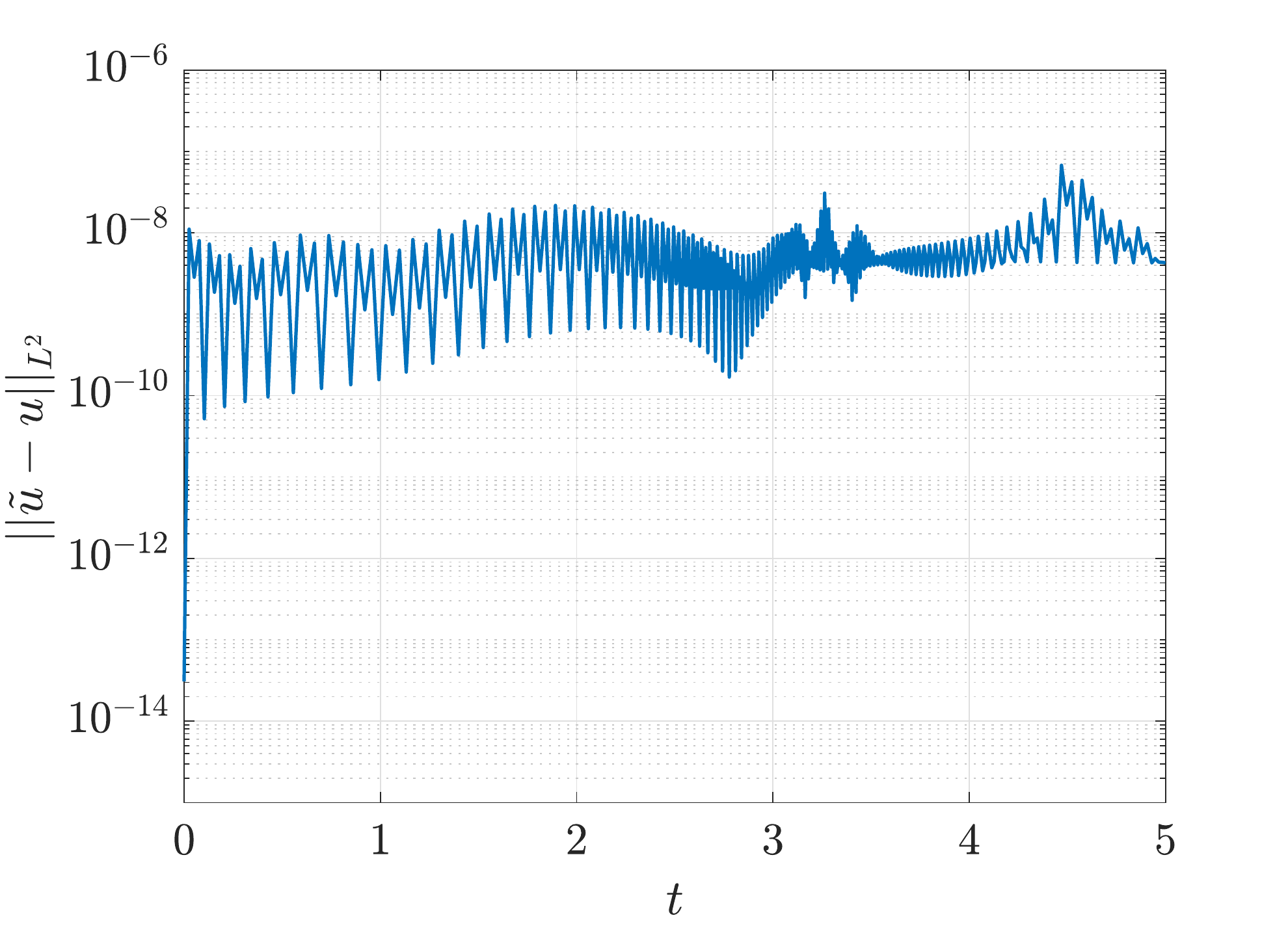}
}
\caption{Time-dependent $L^2(\Omega)$ error between of the DD-TT expansion \eqref{DOTT} and the multivariate function \eqref{3d_ev_example}.}.
\label{fig:3d_ev_error}
\end{figure}

% DA QUI

\subsection{BO-TT propagator}
\label{sec:BO-TT}
Alongside the multi-level DO propagator \eqref{do_ev_eq},
we derive the Bi-Orthogonal Tensor-Train (BO-TT) propagator. 
The key idea is to replace the dynamic orthogonality condition between 
the modes at each level of the binary tree, with a bi-orthogonality 
condition (see \cite{bo1}). To this end, let us consider the following 
sequence of bi-orthogonal decompositions 
\begin{align}
\label{series_tt_bo0}
u(x_1,\ldots,x_d,t)= &\sum_{i_1 = 1}^{\infty}  \varphi_{i_1}^{(1)}(t)
\varphi_{i_1}^{(2,\ldots,d )}(t)\nonumber \\ 
=& \sum_{i_1 = 1}^{\infty} \sum_{i_2 = 1}^{\infty}
 \varphi_{i_1}^{(1)}(t) \varphi_{i_1i_2}^{(2)}(t) 
 \varphi_{i_1i_2}^{(3,\ldots,d)}(t),\nonumber\\
=& \sum_{i_1 = 1}^{\infty} \sum_{i_2 = 1}^{\infty}
 \sum_{i_3 = 1}^{\infty}
 \varphi_{i_1}^{(1)}(t) \varphi_{i_1i_2}^{(2)}(t) 
 \varphi_{i_1i_2i_3}^{(3)}(t) \varphi_{i_1i_2i_3}^{(4,\ldots,d)}(t)\\
 \cdots &\nonumber
\end{align}
where the modes $\varphi_{i_1\cdots i_{j-1}k}^{(j)}(t)$ and 
$\varphi_{i_1\cdots i_{j-1}k}^{(j+1,\ldots,d)}(t)$ satisfy the 
orthogonality conditions
\begin{align}
\label{bo_cond_1}
    \langle \varphi_{i_1 \dots i_{j-1} k}^{(j)}(t), \varphi_{i_1 \dots i_{j-1} p}^{(j)}(t) \rangle &= \lambda_k(t) \delta_{kp}, \\
    \label{bo_cond_2}
    \langle \varphi_{i_1 \dots i_{j-1} k}^{(j+1,\ldots,d)}(t), \varphi_{i_1 \dots i_{j-1} p}^{(j+1,\ldots,d)}(t) \rangle &= \delta_{kp},
\end{align}
for all $t\geq 0$. In contrast with the multi-level DO modes, 
here the one-dimensional BO modes carry the eigenvalues.
By extending the derivation given in \cite{bo1, bo2} to the case 
of multi-level binary trees, it is straightforward to obtain the 
following BO-TT evolution equations
\begin{equation}
\label{bo_evolution}
\begin{aligned}
&\frac{\partial \bm \varphi^{(j)}_{i_1 \cdots i_{j-1}}}{\partial t} = 
\bm M_{i_1 \cdots i_{j-1}}(t) \bm\varphi^{(j)}_{i_1 \cdots i_{j-1}}  + 
\bm p^{(j)}_{i_1 \cdots i_{j-1}}(t) \ , \\
&\bm \Lambda_{i_1 \cdots i_{j-1}}(t)\frac{\partial 
\bm \varphi_{i_1 \cdots i_{j-1}}^{(j+1,\ldots, d)}}{\partial t} = 
- \bm S_{i_1 \cdots i_{j-1}}(t)
\bm \varphi_{i_1 \cdots i_{j-1}}^{(j+1,\ldots, d)}  + 
\bm h^{(j+1,\ldots,d)}_{i_1 \cdots i_{j-1}}(t) \ ,
\end{aligned}
\end{equation}
where 
\begin{equation}
\bm \varphi_{i_1\cdots i_{j-1}}^{(j)} = \left[
\begin{array}{c}
 \varphi_{i_1\cdots i_{j-1}1}^{(j)}\\
\vdots\\
 \varphi_{i_1\cdots i_{j-1}r_j}^{(j)}
\end{array}
\right] \ ,\qquad
 \bm \varphi_{i_1\cdots i_{j-1}}^{(j+1,\ldots d)} = \left[
\begin{array}{c}
 \varphi_{i_1\cdots i_{j-1}1}^{(j+1,\ldots d)}\\
\vdots\\
 \varphi_{i_1\cdots i_{j-1}r_j}^{(j+1,\ldots d)}
\end{array}
\right] \ ,
\end{equation}
and
\begin{equation}
%\label{mats_vecs}
\bm h^{(j+1,\ldots,d)}_{i_1 \cdots i_{j-1}} = \langle N_{i_1 \cdots i_{j-1}}^{(j,\ldots,d)}, 
\bm \varphi^{(j)}_{i_1 \cdots i_{j-1}} \rangle \ , \qquad 
\bm p^{(j)}_{i_1 \cdots i_{j-1}} = \langle N_{i_1 \cdots i_{j-1}}^{(j,\ldots,d)}, 
\bm \varphi^{(j+1,\ldots, d)}_{i_1 \cdots i_{j-1}} \rangle \ , 
\end{equation}  
\begin{equation}
\bm \Lambda_{i_1 \cdots i_{j-1}}(t) = \langle \bm \varphi_{i_1 \cdots i_{j-1}}^{(j)} , \bm \varphi_{i_1 \cdots i_{j-1}}^{(j)} \rangle \ , \qquad
\bm M_{i_1 \cdots i_{j-1}}(t) = \langle \bm \varphi_{i_1 \cdots i_{j-1}}^{(j+1,\ldots, d)}, \frac{\partial \bm \varphi_{i_1 \cdots i_{j-1}}^{(j+1,\ldots, d)}}{\partial t} \rangle \ ,
\end{equation}
\begin{equation}
\bm S_{i_1 \cdots i_{j-1}}(t) = \langle \bm \varphi_{i_1 \cdots i_{j-1}}^{(j)}, \frac{\partial \bm \varphi_{i_1 \cdots i_{j-1}}^{(j)}}{\partial t} \rangle \ ,
\end{equation}
for $j = 2, \ldots, d-2$. The matrices $\bm S_{i_1 \cdots i_j}$ are skew-symmetric, as it can be verified by differentiating \eqref{bo_cond_1} with respect to time.

\subsection{Equivalence between DO-TT and BO-TT expansions}
\label{sec:equivalenceDO-BO-TT}

The DO-TT and the BO-TT series expansions we 
discussed in Section \ref{sec:DO-TT} 
and Section \ref{sec:BO-TT} are equivalent 
in the sense that they approximate a time-dependent 
multivariate function using components from the same 
finite-dimensional function space. In this Section we prove such 
equivalence. The approach we take is to prove 
the equivalence for one
level of the TT binary tree, following similar steps taken 
in \cite{do/bo_equiv}, and then proceed inductively 
to show the equivalence throughout the whole TT binary 
tree. Clearly, switching from DO to BO at any level of the 
binary tree affects all expansions in the child nodes. 
Let $\bm \varphi^{(j)}(t) = [\varphi^{(j)}_1(t), \ldots, 
\varphi^{(j)}_{r_j}(t)]$ and $\bm \varphi^{(j+1, \ldots, d)}(t) 
= [\varphi^{(j+1,\ldots,d)}_1(t), \ldots,
\varphi^{(j+1,\ldots,d)}_{r_j}(t)]$ (row vectors) 
be the BO modes at the $j$th-level of the TT binary 
tree. Consider the transformation
\begin{equation}
\label{transformation}
\bm \psi^{(j)}(t) = \bm\varphi^{(j)}(t) \bm 
\Lambda_j^{-\frac{1}{2}}(t) \bm P_j(t), \qquad 
\bm \psi^{(j+1,\ldots, d)}(t) = \bm \varphi^{(j+1, \ldots, d)}(t) 
\bm \Lambda_j^{\frac{1}{2}}(t) \bm P_j(t),
\end{equation}
where $[\bm \Lambda_j(t)]_{ik} = \langle \varphi^{(j)}_i, 
\varphi^{(j)}_k\rangle$, and $\bm P_j(t)$ 
satisfies the matrix differential equation\footnote{$\bm P_j(t)$
is a time-dependent $r_j\times r_j$ matrix with real coefficients.}
\begin{equation}
\label{matrix_differential_equation}
\begin{cases}
\displaystyle \frac{d \bm P_j}{d t} = -\bm \Lambda_j^{-\frac{1}{2}} 
\bm \Sigma_j \bm \Lambda_j^{-\frac{1}{2}}\bm P_j \vse\\
 \bm P_j(0) = \bm I
\end{cases}
\end{equation}
with 
\begin{equation}
[\bm \Sigma_j(t)]_{ik} = 
\begin{cases} 
[\bm S_j(t)]_{ik} & \quad i \neq j  \vse\\
0 & \quad i = j
\end{cases} \ , \qquad 
[\bm S_j(t)]_{ik} = \langle \varphi_i^{(j)}, \frac{\partial \varphi^{(j)}_k}{\partial t} \rangle.
\label{skewsym}
\end{equation}
\begin{lemma}
\label{lemma:orthogonal}
The matrix $\bm P_j(t)$ defined by the initial value problem 
\eqref{matrix_differential_equation} is orthogonal 
for all $t \geq 0$.
\end{lemma}
\begin{proof}
The matrix $\bm F_j(t) = - \bm \Lambda_j^{-\frac{1}{2}}(t) 
\bm \Sigma_j(t) \bm \Lambda_j^{-\frac{1}{2}}(t)$ is skew-symmetric 
for all $t\geq 0$, since $\bm \Sigma_j$ is skew-symmetric (see Eq. \eqref{skewsym}).  
Therefore, we have 
\begin{align*}
\frac{d }{d t} (\bm P_j^T \bm P_j) &= 
\frac{d \bm P_j^T}{d t} \bm P_j + \bm P_j^T \frac{d \bm P_j}{d t} \\
&= (\bm F_j\bm P_j)^T \bm P_j + \bm P_j^T \bm F_j\bm P_j \\
&= \bm P_j^T (\bm F_j^T + \bm F_j )\bm P_j \\
&= \bm 0. 
\end{align*}
This implies that $\bm P_j(t)^T \bm P_j(t)=\bm P_j(0)^T \bm P_j(0)=\bm I$, i.e., 
$\bm P_j(t)$ is orthogonal for all $t\geq 0$. We notice that the same 
conclusion holds if we replace $\bm P_j(0)$ with any orthogonal matrix, 
not just the identity. 

\end{proof}

\noindent
Hereafter we prove that the linear transformation 
\eqref{transformation}-\eqref{matrix_differential_equation} 
defines the mapping between the DO and BO modes at the $j^{\text{th}}$-level 
of the TT binary tree. 

\begin{theorem}
\label{thm:equiv_thm}
The linear transformation defined by 
\eqref{transformation}-\eqref{matrix_differential_equation} 
is invertible and it defines a new set of modes 
$\{\psi^{(j)}_1(t),\ldots,\psi^{(j)}_{r_j}(t)\}$ such that,
for all $t\geq 0$,  
\begin{itemize}
\item[(i)] $\{\psi^{(j)}_1,\ldots,\psi^{(j)}_{r_j}\}$ is orthonormal,  
\item[(ii)] $\displaystyle \sum_{k=1}^{r_j}\psi^{(j)}_k(t)
\psi^{(j+1,\ldots,d)}_k(t) = \sum_{k=1}^{r_j} \varphi^{(j)}_k(t)\varphi^{(j+1,\ldots,d)}_k(t)$,
\item[(iii)]  $\{\psi^{(j)}_1,\ldots,\psi^{(j)}_{r_j}\}$ satisfies 
the DO condition $\displaystyle \langle \frac{\partial \psi^{(j)}_i}{\partial t}  \psi^{(j)}_k \rangle = 0\qquad \forall i,k=1,\ldots, r_j$.
\end{itemize}
\end{theorem}

\begin{proof}
The transformations defined in 
\eqref{transformation} are invertible by 
Lemma \ref{lemma:orthogonal}.  To prove $(i)$, we notice that
\begin{align*}
\bm \Lambda_j &= \langle [\bm \varphi^{(j)}]^T \bm \varphi^{(j)} \rangle \\
&= \langle [\bm \psi^{(j)} \bm P_j^T \bm\Lambda_j^{\frac{1}{2}}]^T \bm\psi^{(j)} \bm P_j^T \bm \Lambda_j^{\frac{1}{2}} \rangle \\
&= \langle \bm \Lambda_j^{\frac{1}{2}} \bm P_j [\bm \psi^{(j)}]^T \bm \psi^{(j)} \bm P_j^T \bm\Lambda_j^{\frac{1}{2}} \rangle \\
&=   \bm \Lambda_j^{\frac{1}{2}} \bm P_j \langle [\bm \psi^{(j)}]^T \bm \psi^{(j)} \rangle \bm P_j^T  \bm \Lambda_j^{\frac{1}{2}}.
\end{align*}
Multiply by $\bm P_j^T \bm \Lambda_j^{-\frac{1}{2}}$ and 
$\bm \Lambda_j^{-\frac{1}{2}} \bm P_j$ to the left and the 
right hand sides, respectively. This yields, 
\begin{align*}
\langle [\bm \psi^{(j)}]^T \bm \psi^{(j)} \rangle &= \bm P_j^T 
\bm \Lambda_j^{-\frac{1}{2}} \bm \Lambda_j 
\bm \Lambda_j^{-\frac{1}{2}} \bm P_j \\
&= \bm I
\end{align*}
which proves $(i)$.  To prove $(ii)$, it is sufficient to apply the
transformation \eqref{transformation}. In fact,  
\begin{align*}
\sum_{k=1}^{r_j}  \varphi^{(j)}_k \varphi^{(j+1,\ldots,d)}_k
&= \bm \varphi^{(j)}[\bm\varphi^{(j+1,\ldots,d)}]^T \\
&= \bm \psi^{(j)} \bm P_j^T \bm \Lambda_j^{\frac{1}{2}} 
\bm \Lambda_j^{-\frac{1}{2}} \bm P_j \bm \psi^{(j+1,\ldots,d)} \\
&= \bm \psi^{(j)} [\bm \psi^{(j+1,\ldots,d)}]^T\\
&=\sum_{k=1}^{r_j}  \psi^{(j)}_k \psi^{(j+1,\ldots,d)}_k.
\end{align*}
To prove $(iii)$,  we first differentiate $\bm \varphi^{(j)} = 
\bm \psi^{(j)} \bm P_j^T \bm \Lambda_j^{\frac{1}{2}}$ with 
respect to time to obtain\footnote{This equality follows 
from \eqref{matrix_differential_equation} and the identity 
$$\bm S_j = \bm \Sigma_j + \frac{1}{2} \frac{\partial \bm \Lambda_j}{\partial t}.$$}
\begin{align*}
\frac{\partial \bm \varphi^{(j)}}{\partial t} = 
\frac{\partial \bm \psi^{(j)}}{\partial t} \bm P_j^T 
\bm \Lambda_j^{\frac{1}{2}} + \bm \psi^{(j)} \bm P_j^T 
\bm \Lambda_j^{-\frac{1}{2}}[\bm S_j - 2 \bm \Sigma_j]^T.
\end{align*}
At this point, we have 
\begin{align*}
S &= \langle [\bm \varphi^{(j)}]^T \frac{\partial \bm \varphi^{(j)}}{\partial t} \rangle\\
&= \langle \bm \Lambda_j^{\frac{1}{2}} \bm P_j [\bm \psi^{(j)}]^T 
\left(\frac{\partial \bm \psi^{(j)}}{\partial t} \bm P_j^T 
\bm \Lambda_j^{\frac{1}{2}} + \bm \psi^{(j)} \bm P_j^T 
\bm \Lambda_j^{-\frac{1}{2}}[\bm S_j - 2 \bm \Sigma_j]^T\right) \rangle\\
&= \bm \Lambda_j^{\frac{1}{2}} \bm P_j \langle [\bm \psi^{(j)}]^T 
\frac{\partial \bm \psi^{(j)}}{\partial t} \rangle \bm P_j^T \bm 
\Lambda^{\frac{1}{2}} + [\bm S_j - 2 \bm \Sigma_j]^T,
\end{align*}
i.e., 
\begin{align*}
\frac{1}{2} \bm \Lambda_j^{\frac{1}{2}} \bm P_j \langle [\bm \psi^{(j)}]^T 
\frac{\partial \bm \psi^{(j)}}{\partial t} \rangle \bm P_j^T 
\bm \Lambda_j^{\frac{1}{2}} &= \frac{ \bm S_j - \bm S_j^T}{2} - \bm \Sigma_j  \\
&= \bm 0
\end{align*}
because $\bm \Sigma_j$ is the skew-symmetric part of $\bm S_j$.  
We know that $\bm P_j$ and $\bm \Lambda_j$ 
are nonzero, and therefore it must be the case that
\begin{equation}
\langle \frac{\partial [\bm \psi^{(j)}]^T}{\partial t}\bm \psi^{(j)} \rangle = \bm 0.
\end{equation}
i.e., the modes $\psi^{(j)}_k$ are dynamically orthogonal. 

\end{proof}

With the equivalence between the BO and DO modes established at the 
$j^{\text{th}}$ level of the TT binary tree, we can discuss the effects 
of this transformation on the remaining parts of the tree.  
An immediate consequence is that every mode belonging to levels 
above the $j^{\text{th}}$ one remains unchanged.  However, all modes below 
the $j^{\text{th}}$-level need to be recomputed.  Moreover, in the case 
of the BO-TT representation,  one has to make sure 
that there are no eigenvalue crossings in any of 
the bi-orthogonal decompositions \cite{do/bo_equiv}.

\section{Dynamically orthogonal tensor methods for high-dimensional nonlinear PDEs}
\label{sec:PDEs}
In this Section we develop dynamically orthogonal tensor 
methods to compute to solution of initial/boundary value 
problems involving nonlinear PDEs of the form 
\begin{equation}
 \label{pde_intro}
 \begin{cases}
\displaystyle\frac{\partial {u(\bm x,t)}}{\partial t} = 
G(u) \qquad \bm x \in\Omega, \quad t\geq 0 \vse\\
{u}(\bm x,0) = u_0(\bm x)  \qquad \bm x \in\Omega \vse\\
Bu(\bm x,t)  = h(t,\bm x)  \qquad \bm x \in \partial \Omega 
\end{cases}
\end{equation}
where $G$ is a nonlinear operator, $B$ is a linear 
boundary operator, $\Omega$ is a bounded 
subset of $\mathbb{R}^d$ which can be represented 
as a Cartesian products of $d$ one-dimensional 
domains\footnote{The numerical results we present 
in this paper are for spatial domains $\Omega$ that can be represented 
as Cartesian products of one-dimensional domains.
Obviously, this is not always the case. Handling 
high-dimensional functions and PDEs in 
complex geometries is not a trivial task. 
For instance, a four-dimensional sphere may be discretized by a set of 
four-dimensional cubes, i.e., tesseracts. Each tesseract consists 
of eight cubical cells, 24 faces, 32 edges 
and 16 vertices. Connecting such tesseracts in 
a finite-element fashion is not straightforward. Similarly, 
mapping high-dimensional complex domains into 
separable domains (whenever possible) is 
not straightforward.}, 
while $u_0(\bm x)$ and $h(\bm x,t)$ are, respectively, the initial 
condition and the boundary condition. 
To compute the solution of \eqref{pde_intro} 
we substitute any of the tensor series expansion 
we discussed in Section \ref{sec:time_evolution}, e.g.,  
\eqref{series_tt_time}, into \eqref{pde_intro} and 
derive a coupled system of nonlinear evolution equations 
for the one-dimensional modes $\psi^{(j)}_{i_1\cdots i_j}(t)$. 
The derivation of such evolution equations is identical 
to the derivation given in Section \ref{sec:time_evolution}, 
with $\partial u/\partial t$ replaced by $G(u)$.  
Specifically, the DO-TT system 
\eqref{DOTT2}-\eqref{DOTT3} takes the form 
\begin{equation}
\begin{aligned}
\frac{\partial \Psi_{k_1 \cdots k_j}^{(j+1, \ldots, d)}}{\partial t} =&
N_{k_1 \cdots k_j}^{(j+1, \ldots, d)} \ , \\
   \sum_{i_j = 1}^{r_j} \frac{ \partial \psi^{(j)}_{k_1 \cdots k_{j-1} i_j}}{\partial t} &\langle \Psi_{k_1 \cdots k_{j-1} i_j}^{(j+1, \ldots, d)} , \Psi_{k_1 \cdots k_j}^{(j+1, \ldots, d)} \rangle\\
   & =\langle N_{k_1 \cdots k_{j-1}}, \Psi_{k_1 \cdots k_j}^{(j+1, \ldots, d)} \rangle - \sum_{i_j = 1}^{r_j} \psi_{k_1 \cdots k_{j-1} i_j} \langle  N^{(j,\ldots,d)}_{k_1 \cdots k_{j-1}} , \Psi_{k_1 \cdots k_j}^{(j+1, \ldots, d)} \psi_{k_1 \cdots k_{j-1} i_j}^{(j)} \rangle \ ,
\end{aligned}
\label{do_PDE}
\end{equation}
where 
\begin{equation}
N_{k_1}^{(2,\ldots,d)}=\langle G(u),\psi^{(1)}_{k_1}\rangle \ ,\  \ldots  \ ,
N_{k_1 \cdots k_j}^{(j+1, \ldots, d)}=\langle N_{k_1 \cdots k_{j-1}}^{(j, \ldots, d)}\psi^{(j)}_{k_1\cdots k_j}\rangle,\qquad j=2,3,\ldots .
\label{Niprod}
\end{equation}
If $G(u)$ is separable, i.e., if it can be written as 
\begin{equation}
\label{sep_op_0}
G= \sum_{i = 1}^{r_G} G_1^{(i)} \otimes \cdots \otimes G_d^{(i)} \ ,
\end{equation}
where $G_j^{(i)}$ ($i=1,\ldots, r_G$) are nonlinear operators 
acting only on functions of $x_j$, then after replacing $u$ with 
the series expansion \eqref{series_tt_time} all inner products in \eqref{do_PDE}-\eqref{Niprod} 
are essentially one-dimensional integrals. 
A simple example of a separable nonlinear 
operator is $G(u)=-u\cdot \nabla u+\nabla^2 u$, 
where $\nabla$ and $\nabla^2$ are, respectively,  
$d$-dimensional gradient and Laplace operators. Other 
examples of separable linear operators will be given in 
Section \ref{sec:numerics}.
Projecting \eqref{pde_intro} recursively onto the 
tensor-train modes, e.g., as in Section \ref{sec:DO-TT}, 
allows us to compute the solution on a tensor manifold 
with constant rank. This is achieved by solving 
a coupled system of one-dimensional nonlinear PDEs, e.g., 
the system \eqref{do_PDE}. However, the solution 
of \eqref{pde_intro} may not have an accurate representation 
on a tensor manifold with constant rank for all times. Hence, 
we may need to add or remove modes 
adaptively as time integration proceeds. 

\paragraph{Remark}
Classical numerical tensor methods 
to solve high-dimensional PDEs with explicit time stepping 
schemes require rank-reduction to project the solution back 
into a tensor manifold \cite{h_tucker_geom} 
with specified rank (see \cite{venturi2018numerical} \S 5.5). 
This can be achieved, e.g., 
by a sequence of singular value decompositions 
\cite{hsvd_tensors_grasedyk,Kressner2014}, or by 
optimization \cite{Kolda,DaSilva2015,parr_tensor,Silva,Rohwedder,Karlsson}.  
Rank reduction can be computationally intensive, 
especially if performed at each time step. 
Numerical tensor methods with implicit time stepping suffer 
from similar issues. In particular, the nonlinear system that 
yields the solution at the next time step needs to be 
solved on a tensor manifold with constant rank, e.g., 
by using Riemannian optimization algorithms 
\cite{DaSilva2015,h_tucker_geom,Etter}.
The dynamically orthogonal tensor method we propose 
operates in a different way. In particular, 
the hard-to-compute nonlinear projection 
\cite{Lubich2018,hierar} that maps the solution 
of high-dimensional PDEs onto a tensor manifold with 
constant rank \cite{h_tucker_geom} here is 
represented explicitly by the hierarchical DO/BO 
propagator, i.e., by a system of coupled one-dimensional 
nonlinear PDEs. 
In other words, there is no need to perform tensor 
rank reduction \cite{hsvd_tensors_grasedyk,Grasedyck2018,Kressner2014}, 
rank-constrained temporal integration \cite{Lubich2018,hierar}, or 
Riemannian optimization \cite{DaSilva2015}, 
when solving high-dimensional PDEs with the dynamically 
orthogonal tensor method we propose.

\subsection{Adaptive addition and removal of modes}
\label{sec:add_modes}

The solution to the PDE \eqref{pde_intro} 
may not be accurately approximated by elements of a 
tensor manifold with fixed rank at all times. 
Therefore, it may be desirable to increase or decrease 
the tensor rank of the solution as time integration 
proceeds. Removing modes is straightforward since 
one can simply truncate the BO-TT or DO-TT decomposition
we discussed in Section \ref{sec:BO-TT} and 
Section \ref{sec:DO-TT}, respectively, at
any level of the binary tree to a decomposition of 
smaller rank. This obviously affects all child 
nodes and corresponding modes in 
the binary tree. 

Adding modes is a more subtle task.  If the hierarchical 
rank of the tensor representing the 
solution is too large at any level of the binary tree, 
then the matrices at the left hand side of DO-TT and BO-TT 
propagators (Eqs. \eqref{do_ev_eq} and \eqref{bo_evolution}), 
i.e., $\bm C_{i_1\cdots i_j}(t)$ and 
$\bm \Lambda_{i_1 \cdots i_j}(t)$, may 
become nearly singular.  
On the other hand, if the hierarchical rank is too small,  the 
tensor approximation may not be accurate. The energy 
by which each mode contributes to the series expansion of 
the solution can be tracked by the 
eigenvalues of $\bm \Lambda_{i_1 \cdots i_j}$ in the BO-TT 
setting,  and by the eigenvalues of $\bm C_{i_1 \cdots i_j}(t)$ 
in the DO-TT setting. Once a new mode is added with 
zero energy, one or more of the matrices $\bm C_{i_1 \cdots i_j}$ 
or $\bm \Lambda_{i_1 \cdots i_j}$ become singular. 
To overcome this issue, and be able to continue integrating 
the system Babaee {\em et al.} \cite{robust_do/bo} 
proposed an energy threshold criterion 
combined with matrix pseudo-inverse. This technique is effective, 
although it does slightly pollute the solution at the time 
of adding the mode. The process of activating a new mode from 
a state with zero energy, can be rigorously addressed by using the theory
of fast-slow systems \cite{fast_slow}. In fact, both DO-TT and BO-TT 
propagators (Eqs. \eqref{do_ev_eq} and \eqref{bo_evolution})
become very stiff systems when adding a new mode, 
which requires appropriate temporal-integrators. 
Hereafter we propose two new algorithms to add modes in 
the DO-TT/BO-TT propagators:   

\paragraph{\bf Algorithm 1}
In this algorithm, we add a pair of modes 
(left and right) satisfying the orthogonality 
conditions of DO or BO (whichever condition is currently being enforced) 
with zero energy.  Evolve the modes which do not require inverting 
the now singular matrix (right hand modes in the DO setting and 
left hand modes in the BO setting) for a short amount of time 
while keeping the other modes constant.  At this point, the energy 
of the new mode comes up to some value $\lambda_{\epsilon}$ 
which is significantly smaller than the energy of the more developed 
modes.  Now when the matrix is inverted to continue mode propagation 
of the modes which were fixed, we obtain a slow-fast 
system \cite{fast_slow}.  The evolution of these modes 
remains slow-fast until the energy amongst all modes 
become more balanced.

\paragraph{\bf Algorithm 2}
In this algorithm, we switch from DO-TT or BO-TT  propagators 
to an explicit time stepping scheme involving numerical tensors 
\cite{venturi2018numerical,khoromskij,tensor_survey} for a 
number of time steps. This naturally increases the rank of the 
solution tensor (see \cite{venturi2018numerical}, \S 5.5) based 
on the structure of the PDE. At this point, we perform an 
orthogonalization of such tensor
\cite{hsvd_tensors_grasedyk,Kressner2014}, 
and restart the DO-TT/BO-TT propagators using such 
new orthogonal set of expansion modes.
In other words, given a DO-TT (or BO-TT) representation of the solution 
of \eqref{pde_intro} at time $t$, we convert such representation
into a numerical tensor in any format
\cite{Kolda,Bachmayr,hsvd_tensors_grasedyk}. 
Then we perform a number of time steps of \eqref{pde_intro} 
with an explicit temporal integration scheme using the numerical 
tensor representation, without performing any 
rank reduction. This naturally sends the solution tensor 
into a tensor manifold with larger multivariate rank. 
At this point we orthogonalize the solution tensor to 
obtain a new set of bi-orthogonal modes, which can be 
further truncated using the thresholding 
technique explained in Section \ref{sec:thresholding},
and then restart the DO-TT (or BO-TT) propagator with the new 
initial condition.\vse\vse

\paragraph{Remark} Algorithm 1 relies on 
enlarging the finite dimensional function 
space which the approximate solution lives in.  
Algorithm 2 is different in that the new bi-orthogonal modes 
obtained from the numerical tensor need not lie in the same 
finite dimensional function space as the previous set of modes.  
In other words, we are re-representing the 
solution in a different finite dimensional function space (tensor 
manifold with larger multivariate rank).  
Moreover, we are not explicitly adding modes with low energy, but 
we are letting the PDE itself increase the solution rank using numerical 
tensor techniques. Algorithm 2 may or may not result in 
a slow-fast system. Numerical examples demonstrating the 
effectiveness of Algorithm 2 will be provided in 
Section \ref{sec:hyperbolic}.

\section{Numerical examples}
\label{sec:numerics}
In this Section,  we demonstrate the accuracy and 
computational efficiency of the dynamically orthogonal 
tensor method we propose in this paper to compute  
the solution of high-dimensional PDEs. 
Specifically, we study the DO-TT 
representation, and apply it to hyperbolic and parabolic 
PDEs in periodic hypercubes with dimension ranging 
from 2 to 50. The reason for the choice of such PDEs 
is that they admit analytical solutions, which we will use
to rigorously assess the accuracy and convergence rate
of the proposed methods. We also demonstrate 
the adpative  algorithm we developed in 
Section \ref{sec:add_modes} to dynamically 
enrich the time-evolving basis at each level 
of the TT binary tree.

\subsection{Hyperbolic PDEs}
\label{sec:hyperbolic}
In this Section we study the DO-TT propagator of several 
linear hyperbolic PDEs in periodic domains. 
\subsubsection{Two-dimensional hyperbolic PDE}
Let us begin with the 
two-dimensional initial/boundary value problem
\begin{equation}
 \label{2dpde}
 \begin{cases}
\displaystyle\frac{\partial u(x_1,x_2,t)}{\partial t} = (\sin(x_1) +
 3\cos(x_2))\frac{\partial u(x_1,x_2,t)}{\partial x_1} +
 \cos(x_2)\frac{\partial u(x_1,x_2,t)}{\partial x_2} \vs\\
u(x_1,x_2,0) = \exp\left[\sin(x_1 + x_2)\right]
\end{cases}
\end{equation}
in the spatial domain $\Omega=[0,2\pi]^2$ with periodic 
boundary conditions. 
\begin{figure}
\centerline{\hspace{0.0cm}\footnotesize (a)  
\hspace{7.5cm} (b)}
\centerline{
\includegraphics[width=0.5\textwidth]{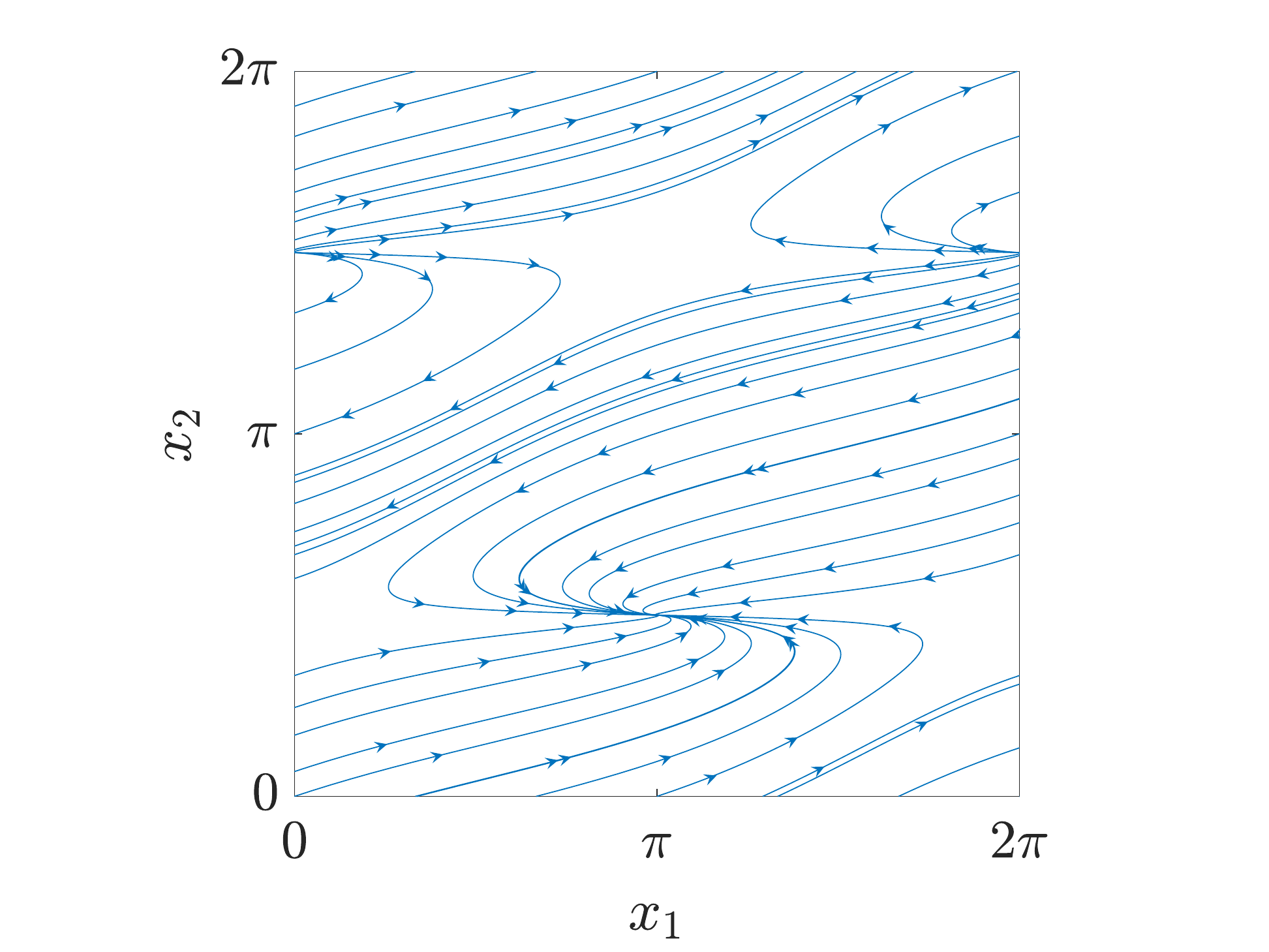}
\includegraphics[width=0.5\textwidth]{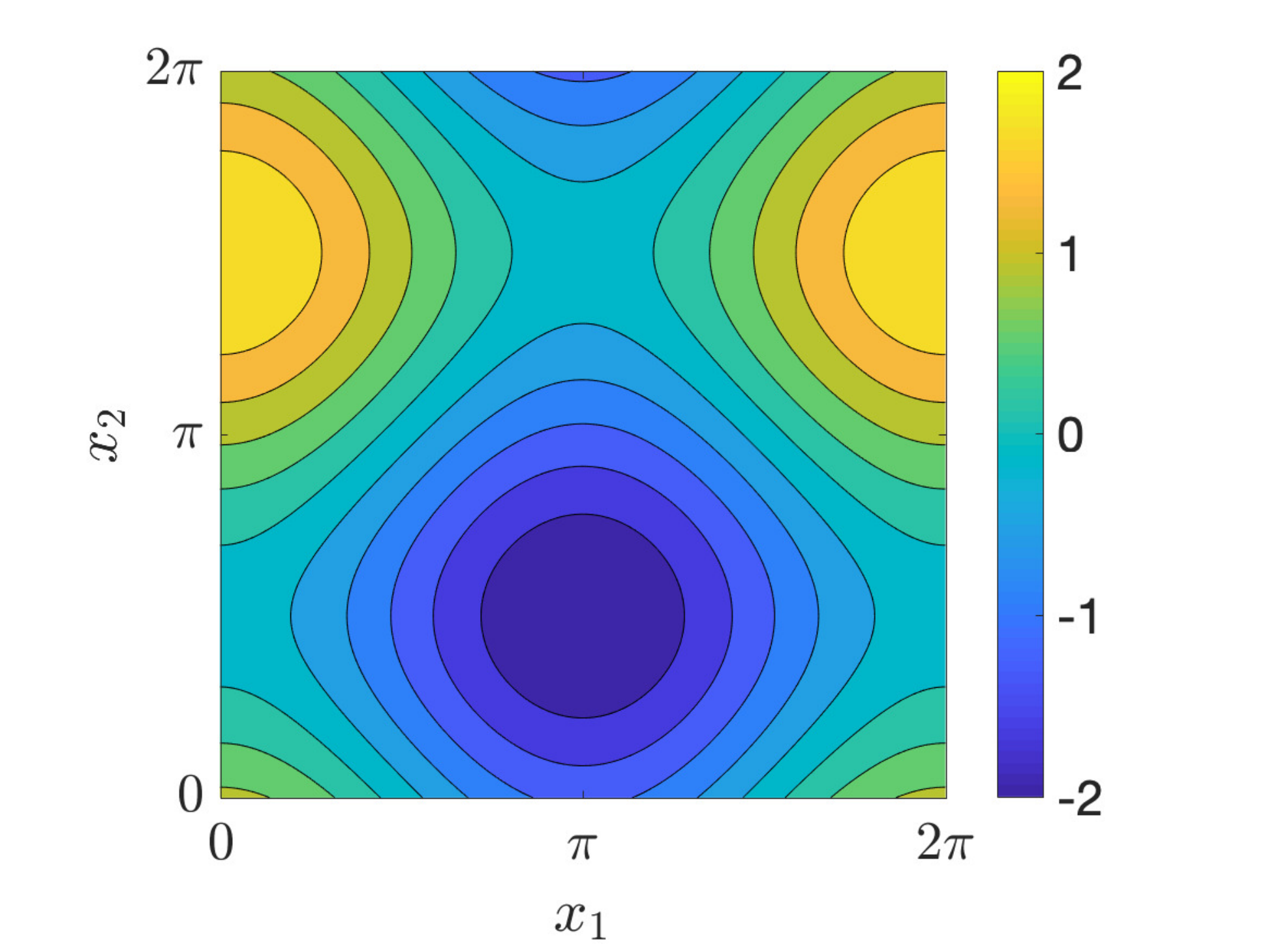}
}
\caption{Two-dimensional advection equation \eqref{2dpde}. 
(a) Characteristic curves associated associated with the PDE 
and (b) and divergence of the velocity field that advects the 
solution (b). Clearly, the flow depicted in (a) is not volume-preserving.}.
\label{fig:phase_plane}
\end{figure}
As is well known, the PDE \eqref{2dpde} can be reduced 
to the trivial ODE $du/dt=0$ along the flow generated 
by the dynamical system (see, e.g.,  \cite{Rhee}) 
\begin{equation}
\label{char_system}
\left\{
\begin{aligned}
\displaystyle\frac{d x_1}{d t} &= \sin(x_1) + 3 \cos(x_2), \vs\\
\displaystyle\frac{d x_2}{d t} &= \cos(x_2).
\end{aligned}\right.
\end{equation}
In Figure \ref{fig:phase_plane} we plot the phase 
portrait we obtained by solving \eqref{char_system} numerically
for different initial conditions $(x_{01},x_{02})\in\Omega$.
With the flow $\{x_1(t,x_{01},x_{02}),x_1(t,x_{01},x_{02})\}$ 
available, we can write the analytical solution to \eqref{2dpde} as
\begin{equation}
u(x_1,x_2,t)=\exp\left[\sin(x_{01}(x_1,x_2,t) + x_{02}(x_1,x_2,t))\right],
\label{2Dsol}
\end{equation}
where $\{x_{01}(x_1,x_2,t), x_{02}(x_1,x_2,t)\}$ denotes 
the inverse flow. The semi-analytical solution \eqref{2Dsol} 
is plotted in the first row of Figure \ref{fig:2d_solution} at 
different times. 
Next, we compute the solution to  \eqref{2dpde} using the 
DO-TT method we proposed in Section \ref{sec:PDEs}. To this 
end, we first perform a bi-orthogonal decomposition of the initial condition 
$u(x_1,x_2,0)$ in \eqref{2dpde} with threshold set to $\sigma = 10^{-13}$ 
(see Section \ref{sec:thresholding}). This yields $17$ modes 
$(\psi^{(1)}_{i_1} , \psi^{(2)}_{i_1})_{i_1 = 1}^{17}$, each of 
which is collocated on an evenly-spaced grid with $257$ nodes 
in $[0,2\pi]$. The DO-TT system associated with the PDE
 \eqref{2dpde} is 
\begin{equation}
\label{do_tt_2d_ex}
\begin{aligned}
\frac{\partial \psi^{(2)}_j}{\partial t} &= \sum_{i_1 = 1}^{r_1} 
\left\{ \langle \sin(x_1) \frac{\partial \psi^{(1)}_{i_1}}{\partial x_1} 
\psi^{(1)}_j \rangle \psi^{(2)}_{i_1} + 
3 \langle \frac{\partial \psi^{(1)}_{i_1}}{\partial x_1} \psi^{(1)}_j \rangle 
\psi^{(2)}_{i_1} \cos(x_2) \right\} + 
\cos(x_2) \frac{\partial \psi^{(2)}_j}{\partial x_2}, \\
\sum_{p=1}^{r_1} \langle \psi^{(2)}_j\psi^{(2)}_p\rangle 
\frac{\partial \psi^{(1)}_p}{\partial t} &= 
\sum_{i_1 = 1}^{r_1} \left\{
\sin(x_1) \frac{\partial \psi^{(1)}_{i_1}}{\partial x_1} 
\langle \psi^{(2)}_{i_1} \psi^{(2)}_j \rangle 
+ 3 \frac{\partial \psi^{(1)}_{i_1}}{\partial x_1} 
\langle \cos(x_2) \psi_{i_1}^{(2)} \psi_j^{(2)} \rangle +  \right.\\
& \psi_{i_1}^{(1)} \langle \cos(x_2) 
\frac{\partial \psi^{(2)}_{i_1}}{\partial x_2} \psi^{(2)}_j \rangle 
- \sum_{p=1}^{r_1} \psi^{(1)}_p \left[ \langle \sin(x_1) 
\frac{\partial \psi^{(1)}_{i_1}}{\partial x_1} \psi^{(1)}_p \rangle 
\langle \psi^{(2)}_{i_1} \psi^{(2)}_j \rangle \right.\\
&\left.\left.+ 3 \langle \frac{\partial \psi^{(1)}_{i_1}}{\partial x_1} 
\psi^{(1)}_p \rangle \langle \cos(x_2) \psi_{i_1}^{(2)} \psi_j^{(2)} \rangle
 + \langle \psi_{i_1}^{(1)} \psi^{(1)}_p \rangle \langle \cos(x_2)
\frac{\partial \psi^{(2)}_{i_1}}{\partial x_2} \psi^{(2)}_j \rangle 
\right] \right\}.
\end{aligned}
\end{equation}
In a Fourier spectral collocation setting \cite{spectral}, the 
partial derivatives $\partial/\partial x_1$,  $\partial/\partial x_2$
and the inner products can be easily represented by one-dimensional 
spectral differentiation matrices, and one-dimensional  Fourier quadrature 
rules \cite{spectral}. This allows us to transforms the PDE system 
\eqref{do_tt_2d_ex} into a system of nonlinear ODEs 
with $2\times r_1\times 256$ equations. It is worthwhile 
emphasizing that the DO-TT system \eqref{do_tt_2d_ex} is 
nonlinear even though the PDE \eqref{2dpde} is linear. 
The nonlinearity is related to the fact 
that we implicitly project the solution back on a tensor manifold 
with constant rank $r_1$ at each time. As we pointed out in 
Section \ref{sec:PDEs}, projecting on a tensor manifold with constant 
rank is a nonlinear operation.
The system \eqref{do_tt_2d_ex} is solved numerically by  
inverting the matrix $C_{jk}(t)=\langle \psi^{(2)}_j\psi^{(2)}_k\rangle$ 
(assuming it is non-singular) at each time step, and an 
explicit RK4 scheme with time step $\Delta t = 10^{-3}$. 
We ran one simulation with constant rank $r_1=17$, 
two adaptive simulations using the pseudo-inverse (PI) 
technique proposed in \cite{robust_do/bo} 
for addition of modes, and one adaptive simulation 
using our Algorithm 2 in Section \ref{sec:add_modes}. 
In Figure  \ref{fig:2d_solution} we 
compare the time snapshots of the constant rank  
DO-TT solution ($r_1=17$) with the semi-analytical 
solution \eqref{2Dsol}.   
\begin{figure}[t]
\hspace{0.5cm}
\centerline{\footnotesize\hspace{-0.5cm}$t = 0.0$ \hspace{4.5cm} $t = 0.5$  \hspace{4.5cm} $t = 1.0$ }

\hspace{0.5cm} 
\centerline{
	\rotatebox{90}{\hspace{0.7cm}\footnotesize Method of characteristics }
		\includegraphics[width=0.34\textwidth]{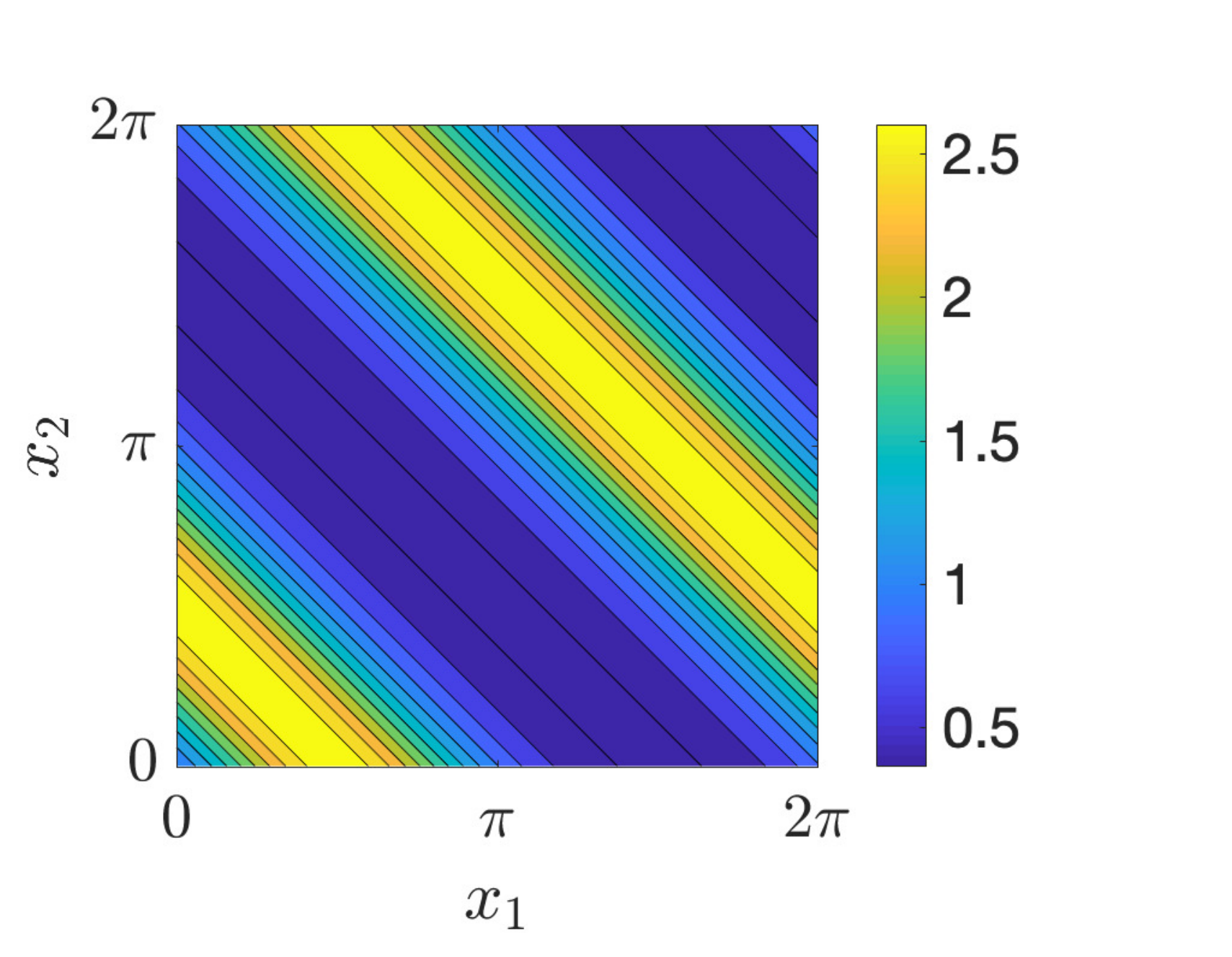}
		\includegraphics[width=0.34\textwidth]{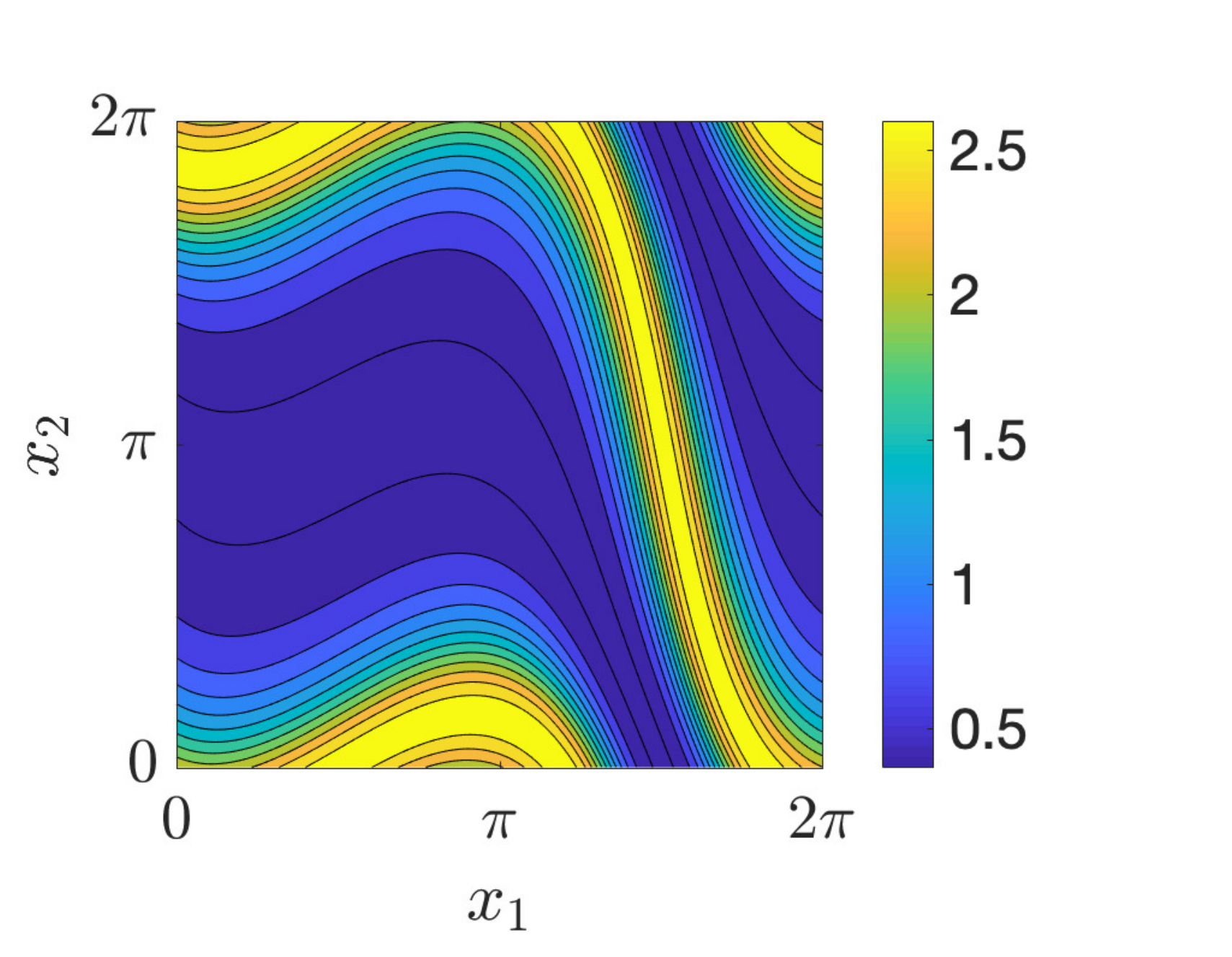}
		\includegraphics[width=0.34\textwidth]{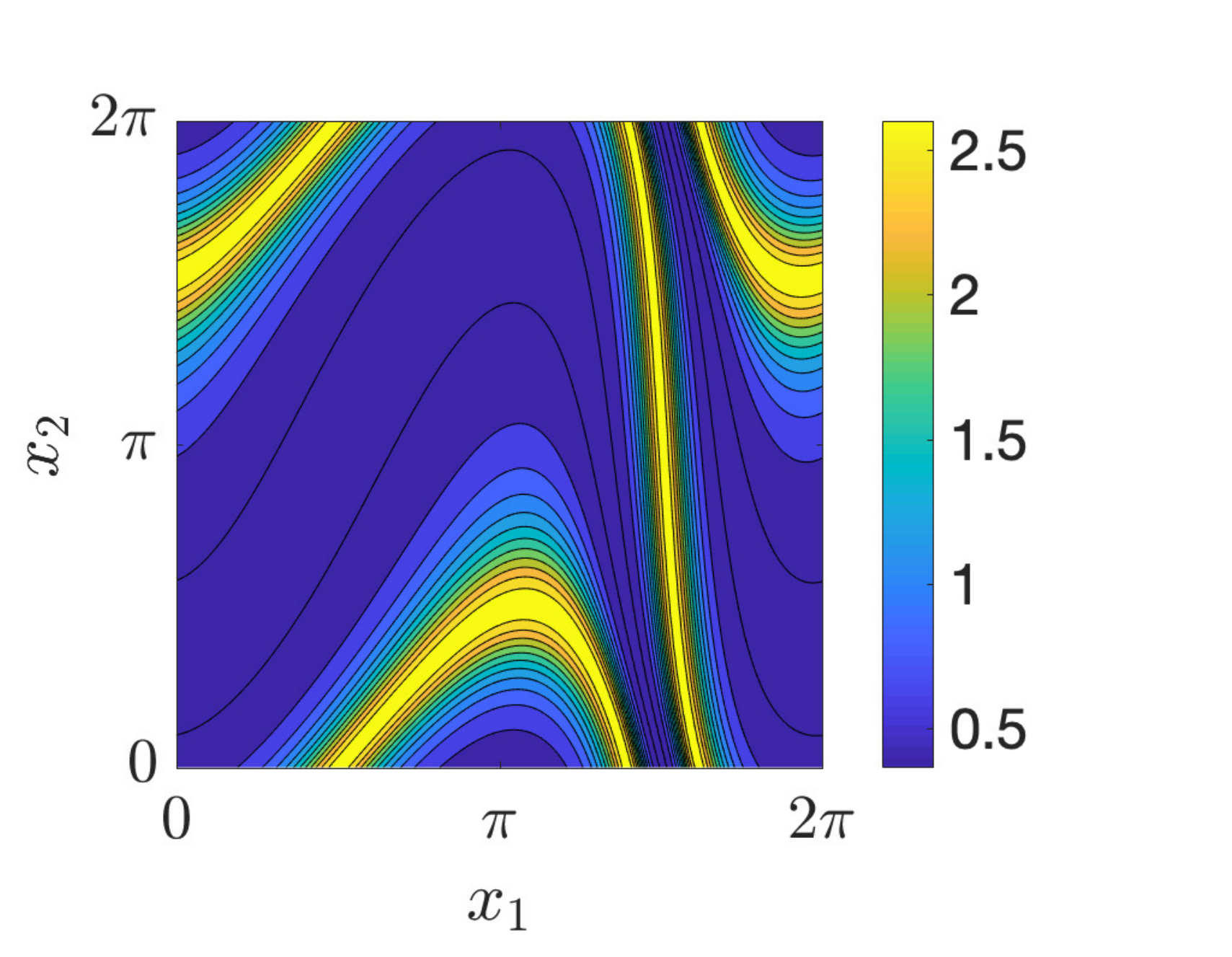}	
}
	
\hspace{0.5cm}	
\centerline{
	\rotatebox{90}{\hspace{1.2cm}  \footnotesize  DO-TT (rank 17)}
       \includegraphics[width=0.34\textwidth]{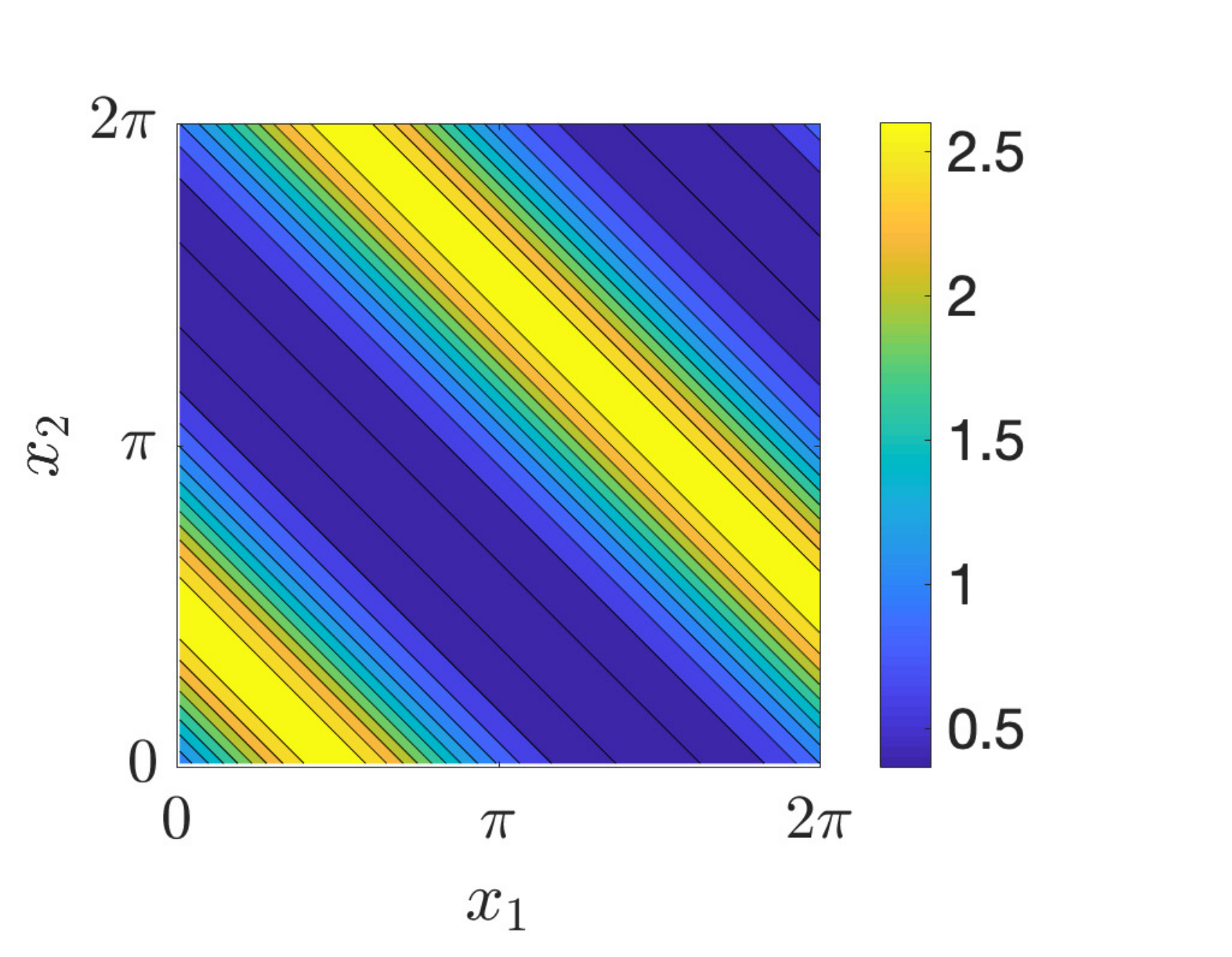}
	    \includegraphics[width=0.34\textwidth]{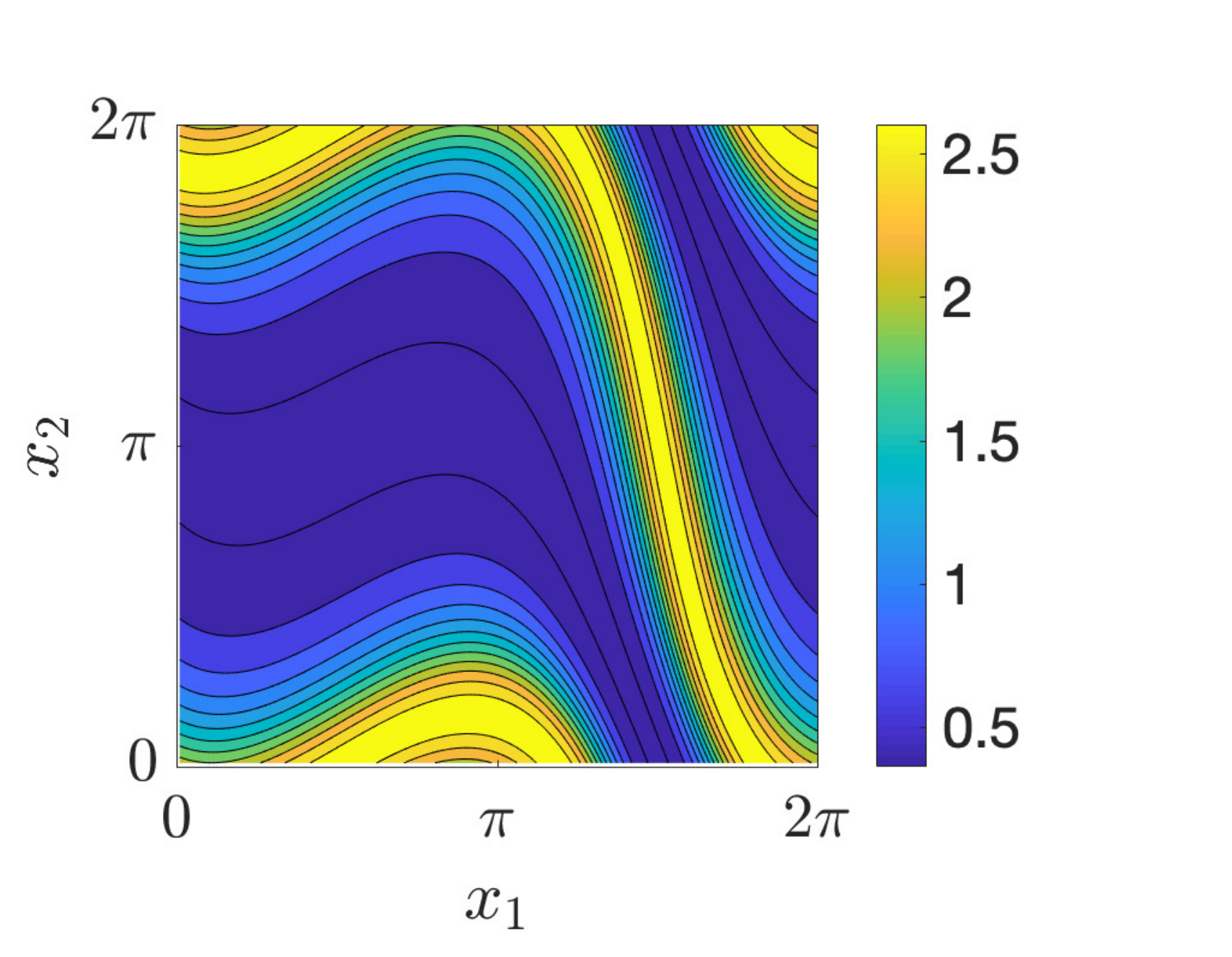}
	    \includegraphics[width=0.34\textwidth]{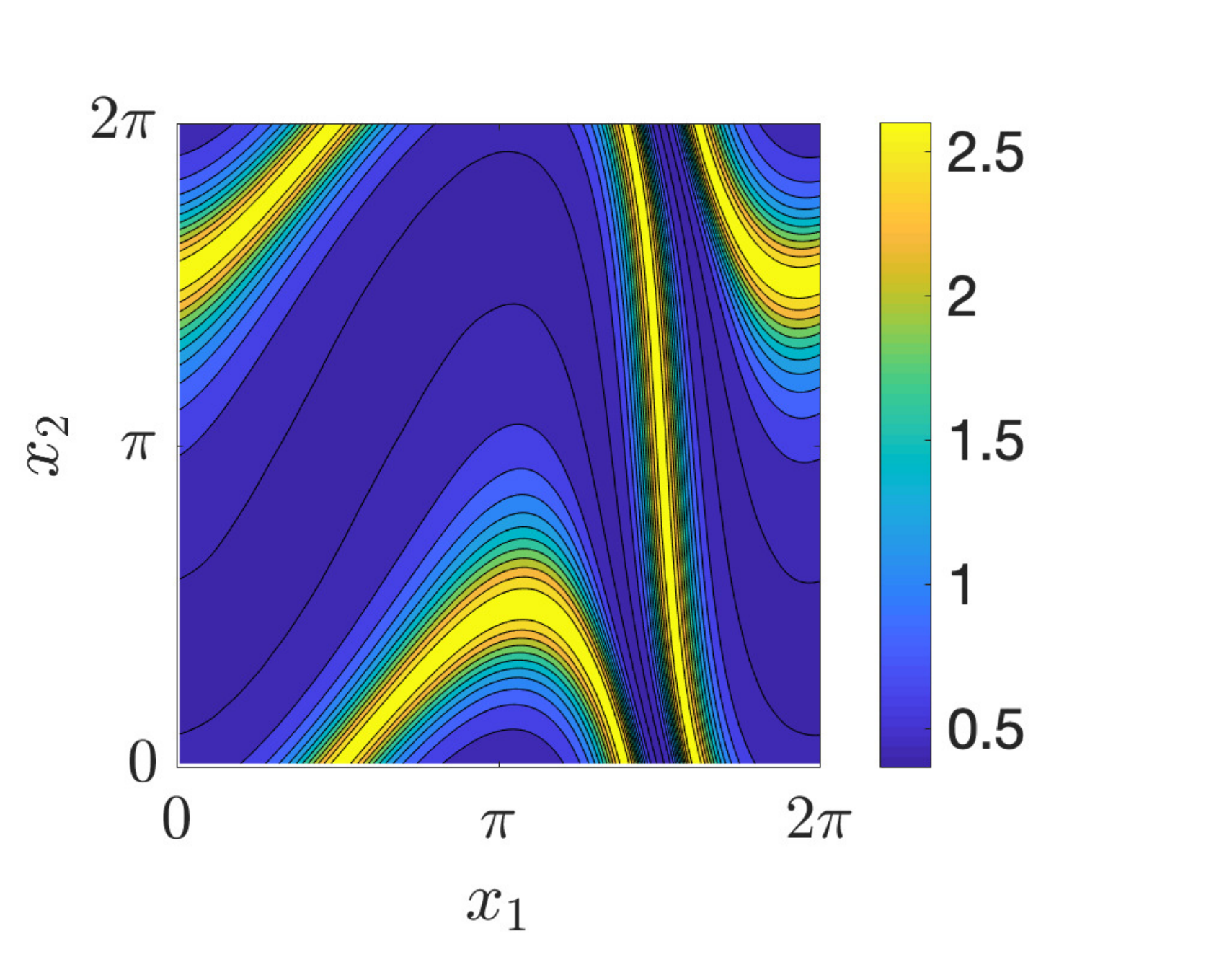}
}

\hspace{0.5cm}	
\centerline{
	\rotatebox{90}{\hspace{1.2cm}  \footnotesize Pointwise error}
       \includegraphics[width=0.34\textwidth]{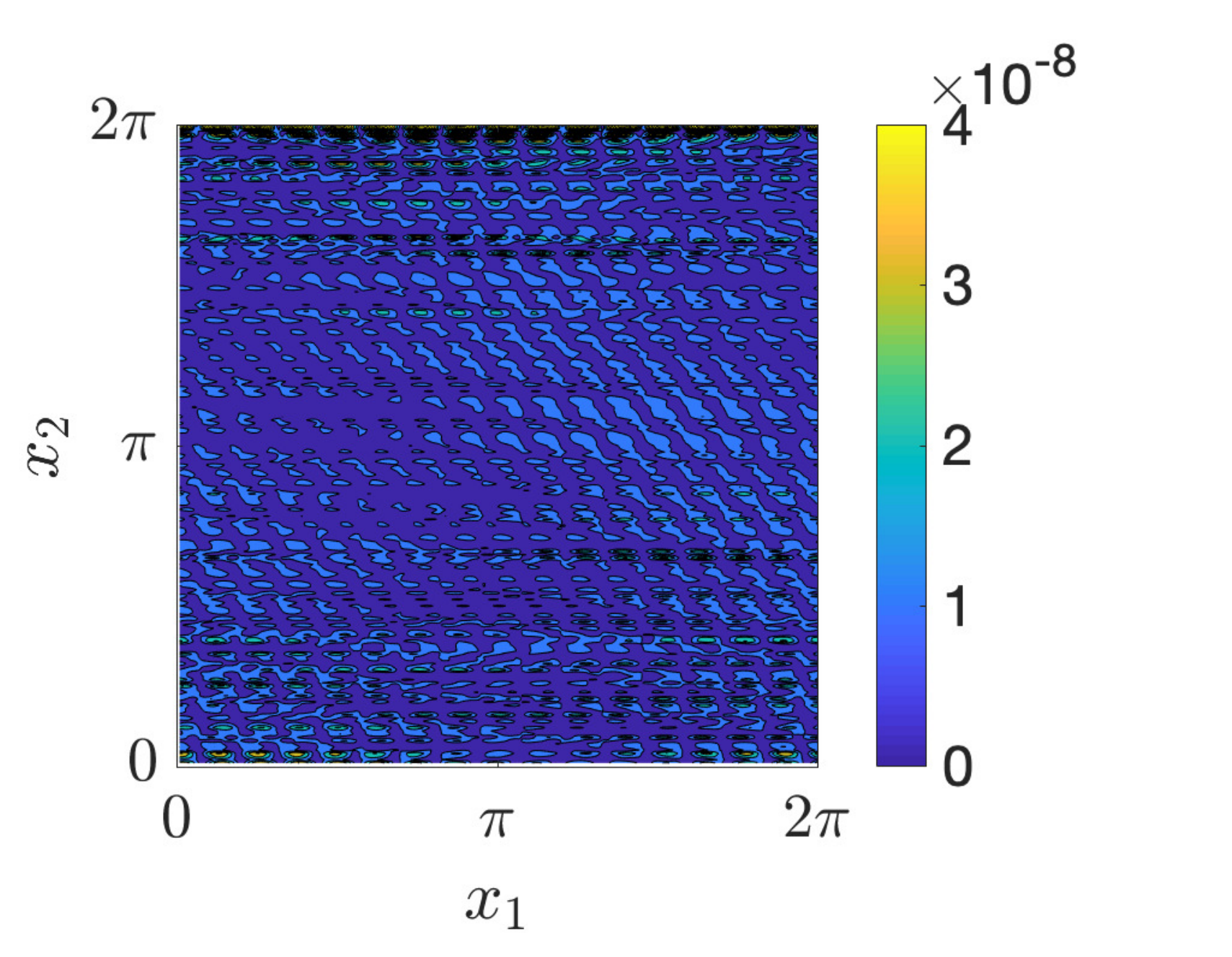}
	    \includegraphics[width=0.34\textwidth]{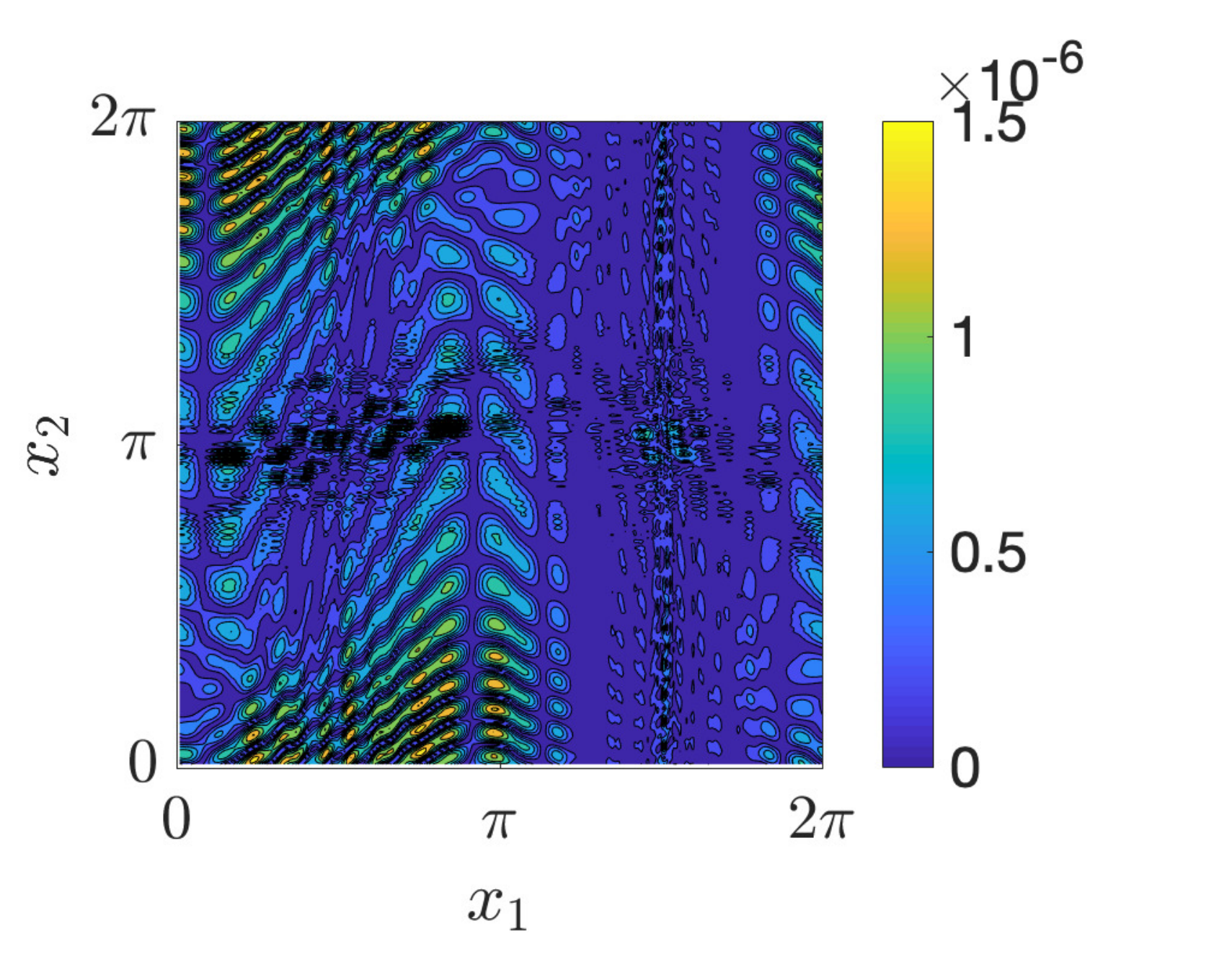}
	    \includegraphics[width=0.34\textwidth]{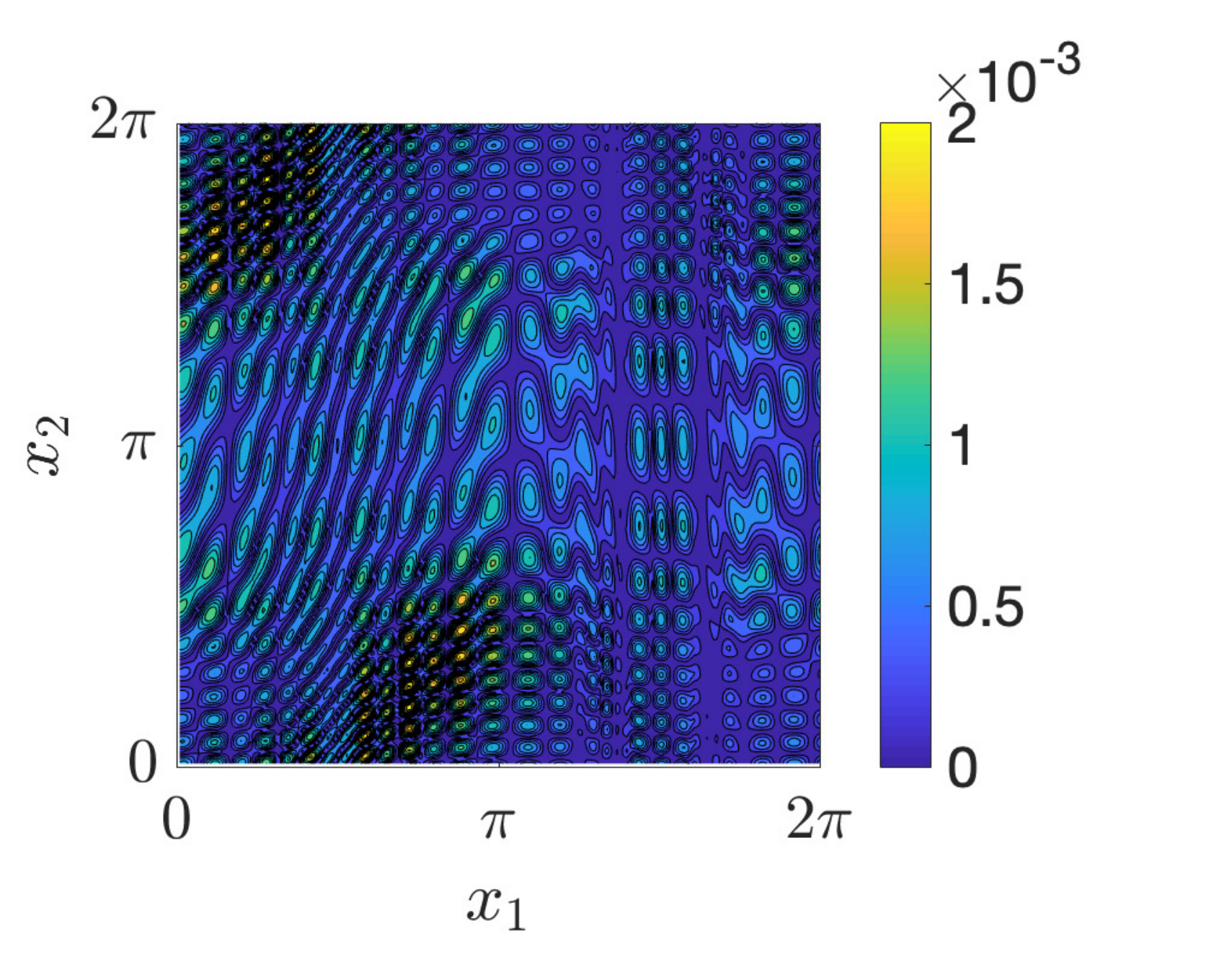}
}
\caption{Time snapshots of the solution to the PDE \eqref{2dpde} obtained using Method of Characteristics and the proposed Dynamically Orthogonal Tensor Train (DO-TT) method with constant rank ($r_1=17$). We also plot the maximum pointwise error between the two solutions.}
\label{fig:2d_solution}
\end{figure}
In Figure \ref{fig:2d_error} we plot the time-dependent 
$L^2(\Omega)$ errors between various DO-TT simulations and 
the semi-analytical solution \eqref{2Dsol}.  It is seen that 
the simulation with constant rank has an error slope that 
increases substantially around $t = 0.5$. This suggests that 
$17$ DO-TT modes are no longer sufficient 
to represent the solution \eqref{2Dsol} for $t>0.5$. 
This can also be also seen from the fact that the DO-TT spectrum 
tends to flatten out in time (Figure \ref{fig:2d_mode_ev}), 
suggesting that each of the $17$ modes is picking up more 
and more energy. To overcome this problem, and therefore 
control the error growth in time, we implemented the 
adaptive algorithm for mode addition/removal proposed 
in \cite{robust_do/bo}, and compared it with our Algorithm 2 
in Section \ref{sec:add_modes} (Algorithm 2). 
As it is seen in Figure \ref{fig:discontinuous_adaptive_modes}, 
such algorithm can indeed control the temporal 
growth of the DO-TT spectrum, 
hence the overall error (see Figure \ref{fig:2d_error}).
Each time modes are added using our Algorithm 2 
(Section \ref{sec:add_modes}), there is a re-orthogonalization 
process which can yield a discontinuity in the temporal evolution 
of each mode (see Figure \ref{fig:discontinuous_adaptive_modes}). 
In practice, the DO-TT system is re-started from a new initial 
condition after such re-orthogonalization. This does is not 
affect the solution, nor creates any temporal discontinuity 
or error jump (see Figure \ref{fig:2d_error}). 
\begin{figure}[t]
\centerline{\footnotesize\hspace{0.4cm}$t=0.0$ \hspace{3.9cm} $t=0.5$  \hspace{4.1cm} $t = 1.0$}
	\centerline{
	\rotatebox{90}{\hspace{1.3cm}  \footnotesize}
		\includegraphics[width=0.3\textwidth]{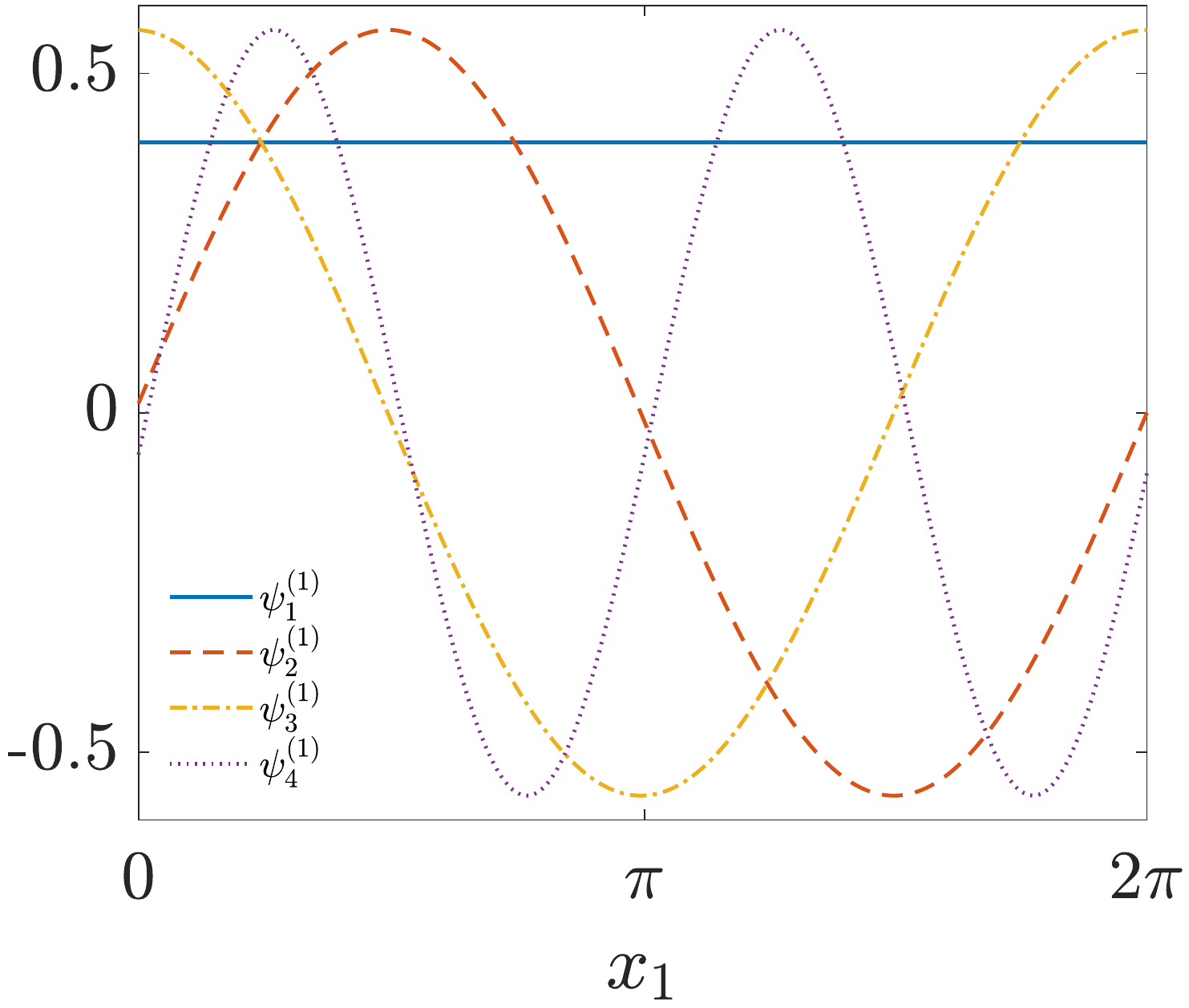}
		\includegraphics[width=0.3\textwidth]{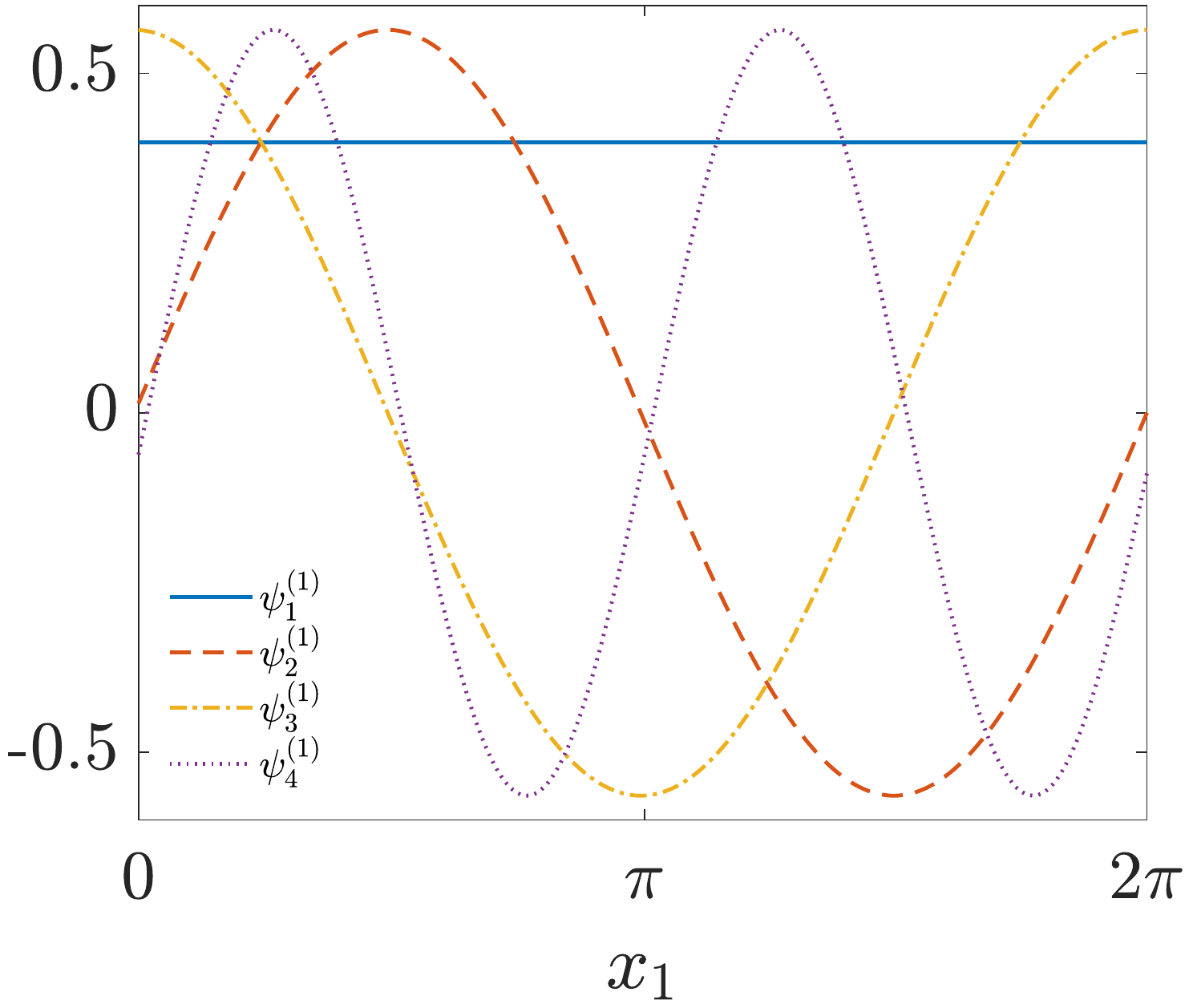}
		\includegraphics[width=0.3\textwidth]{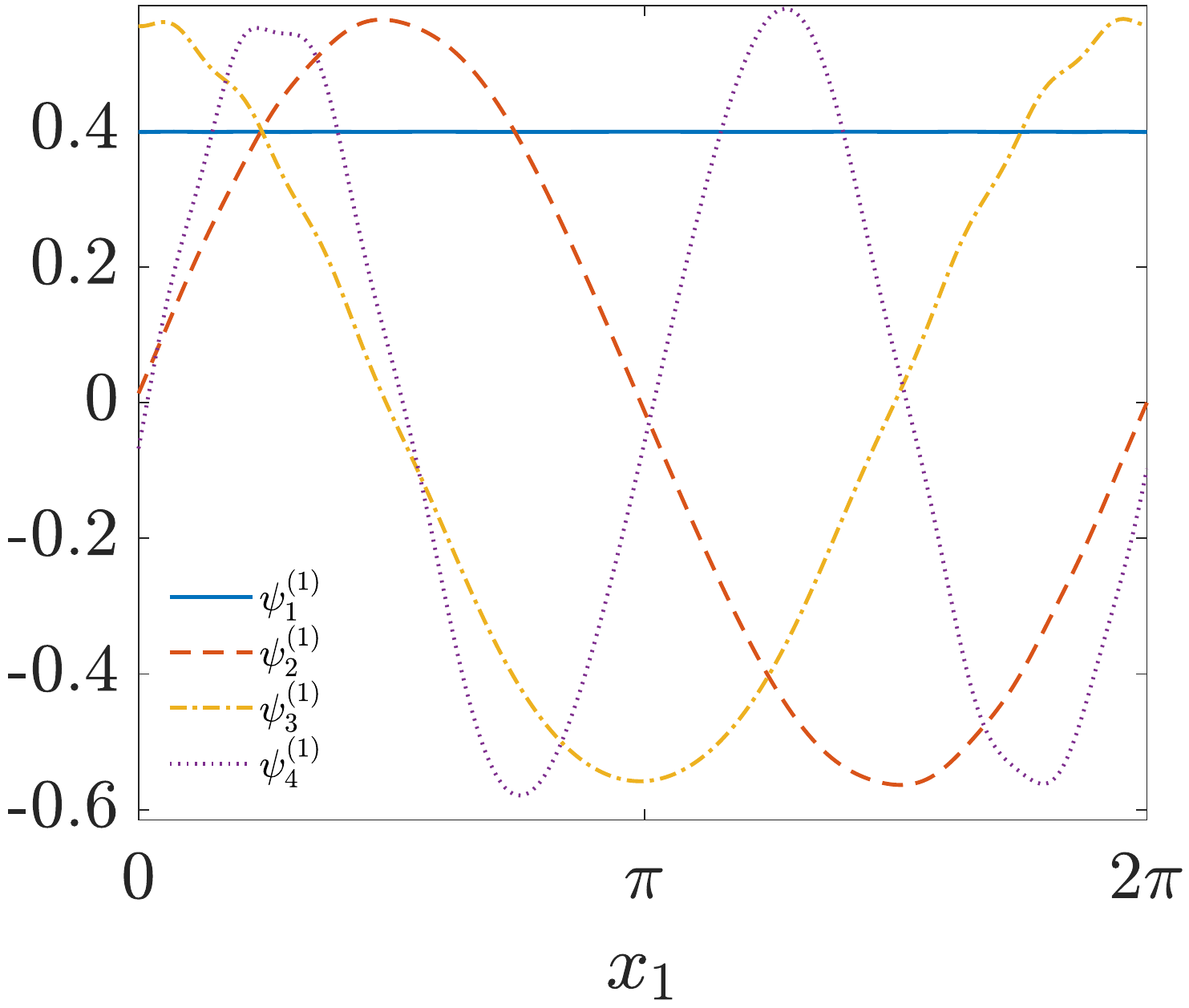}	
}
	\centerline{
	\rotatebox{90}{\hspace{1.3cm}  \footnotesize }
		\includegraphics[width=0.3\textwidth]{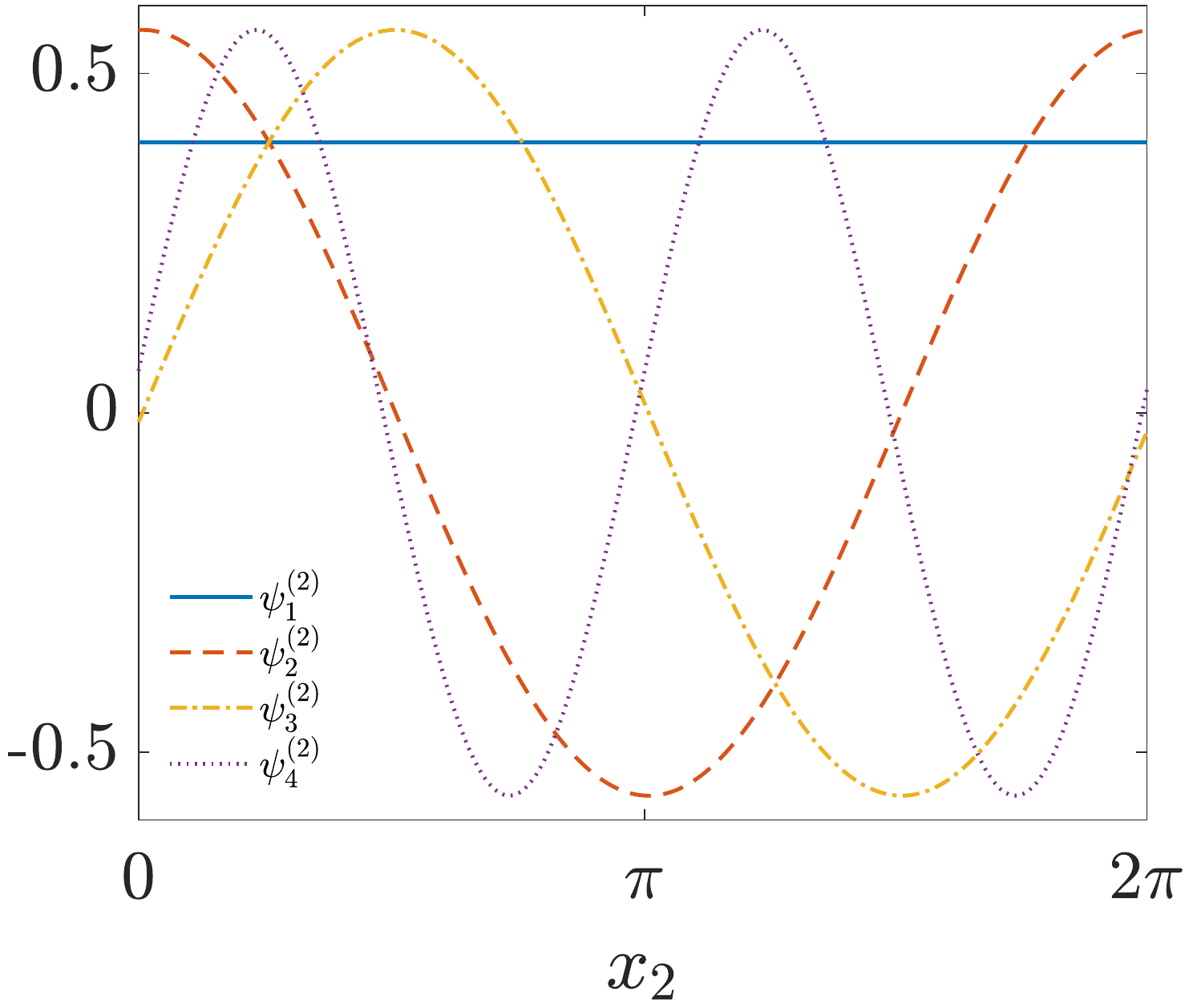}
	    \includegraphics[width=0.3\textwidth]{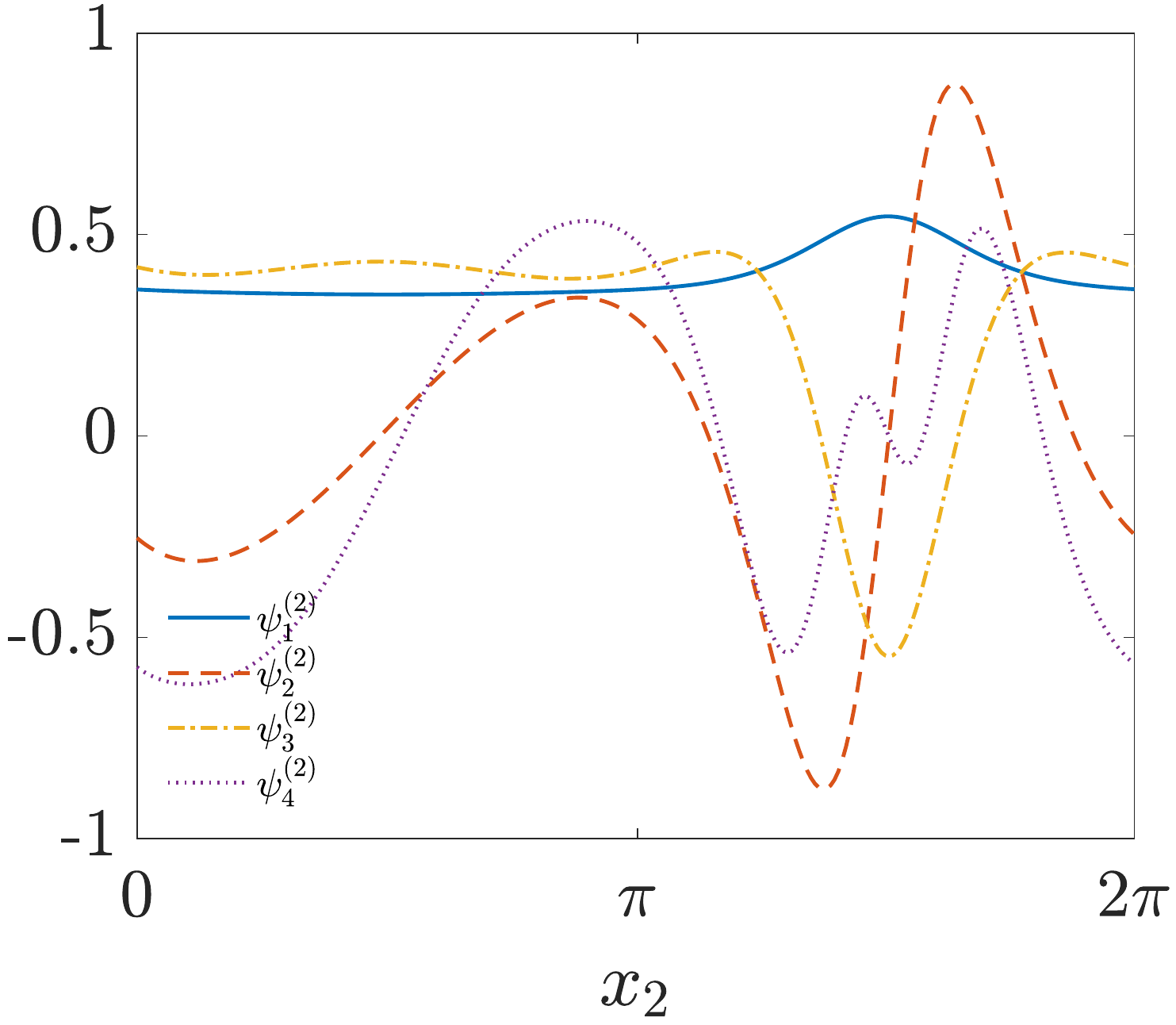}
	    \includegraphics[width=0.3\textwidth]{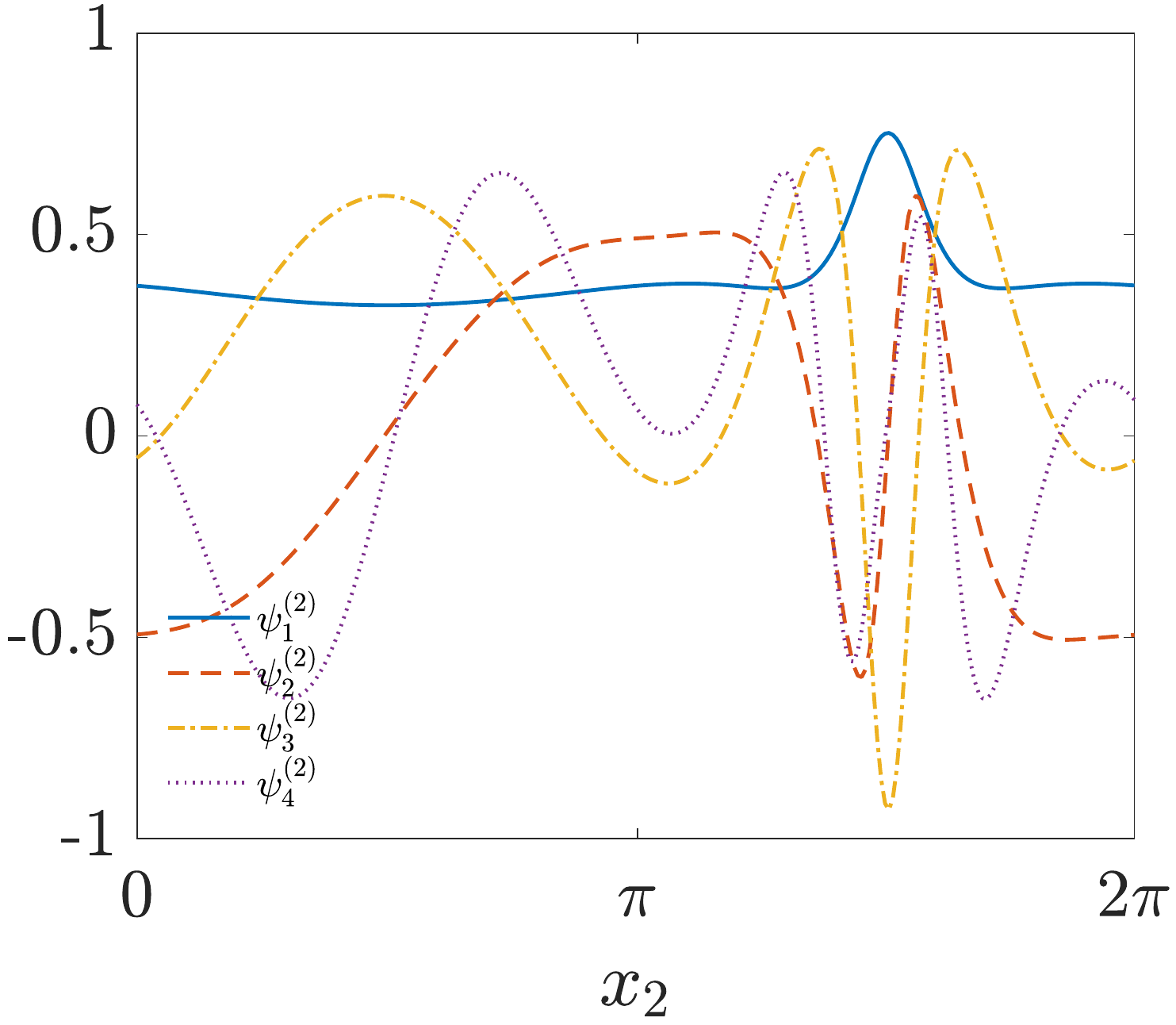}
}
	\centerline{
	\rotatebox{90}{\hspace{2.1cm} \footnotesize $\lambda_{i_1}$ }
		\includegraphics[width=0.3\textwidth]{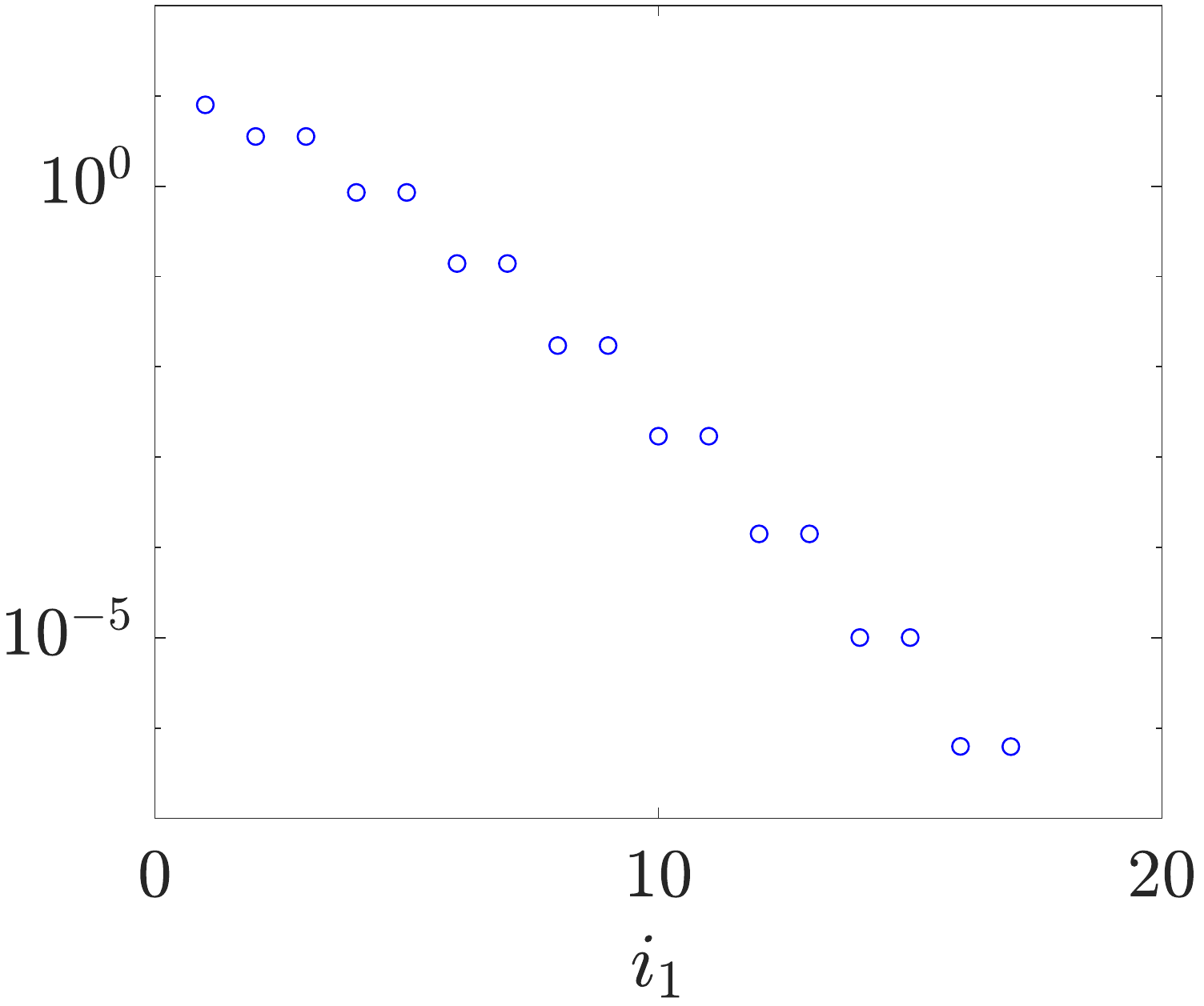}	
	    \includegraphics[width=0.3\textwidth]{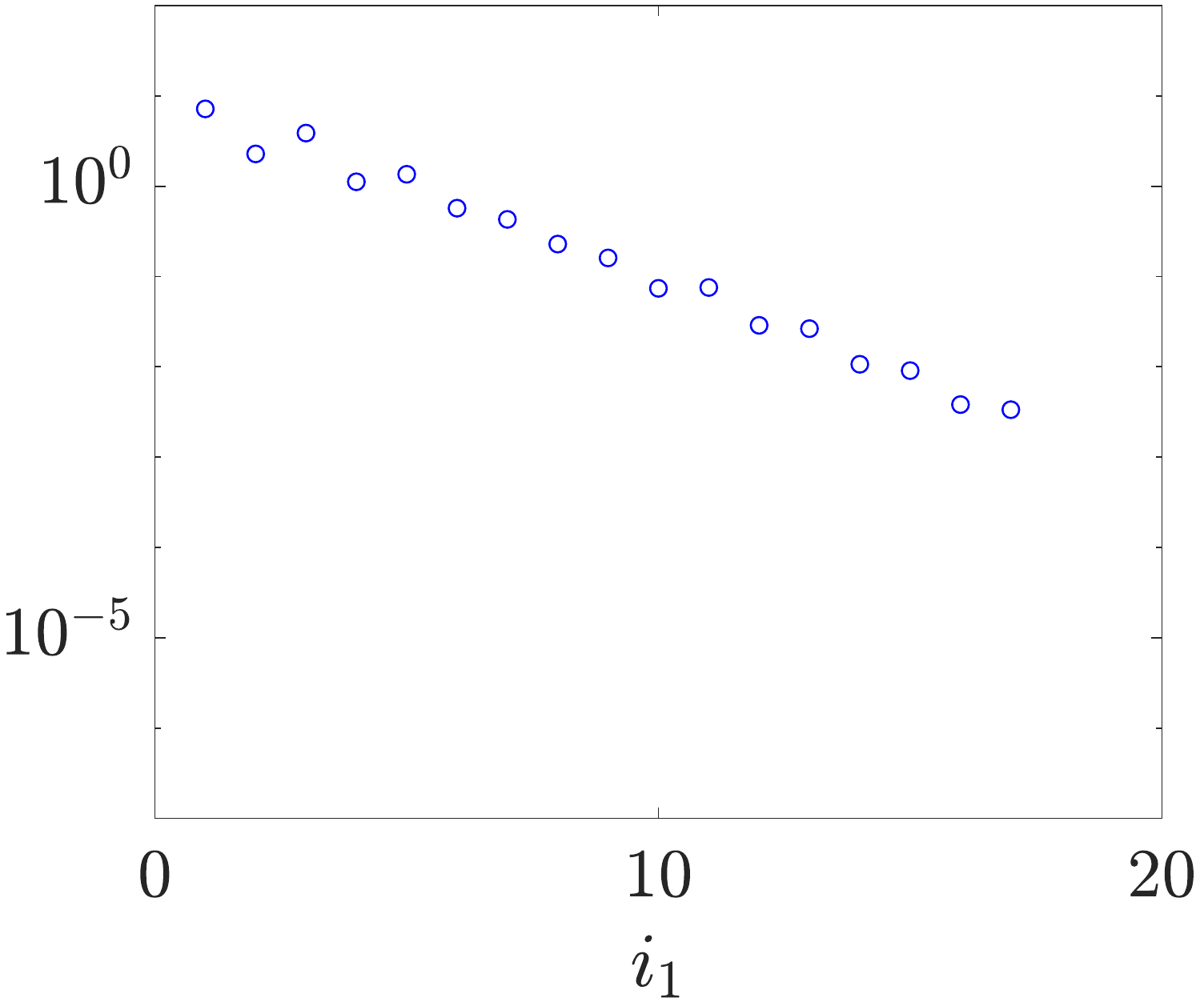}
	    \includegraphics[width=0.3\textwidth]{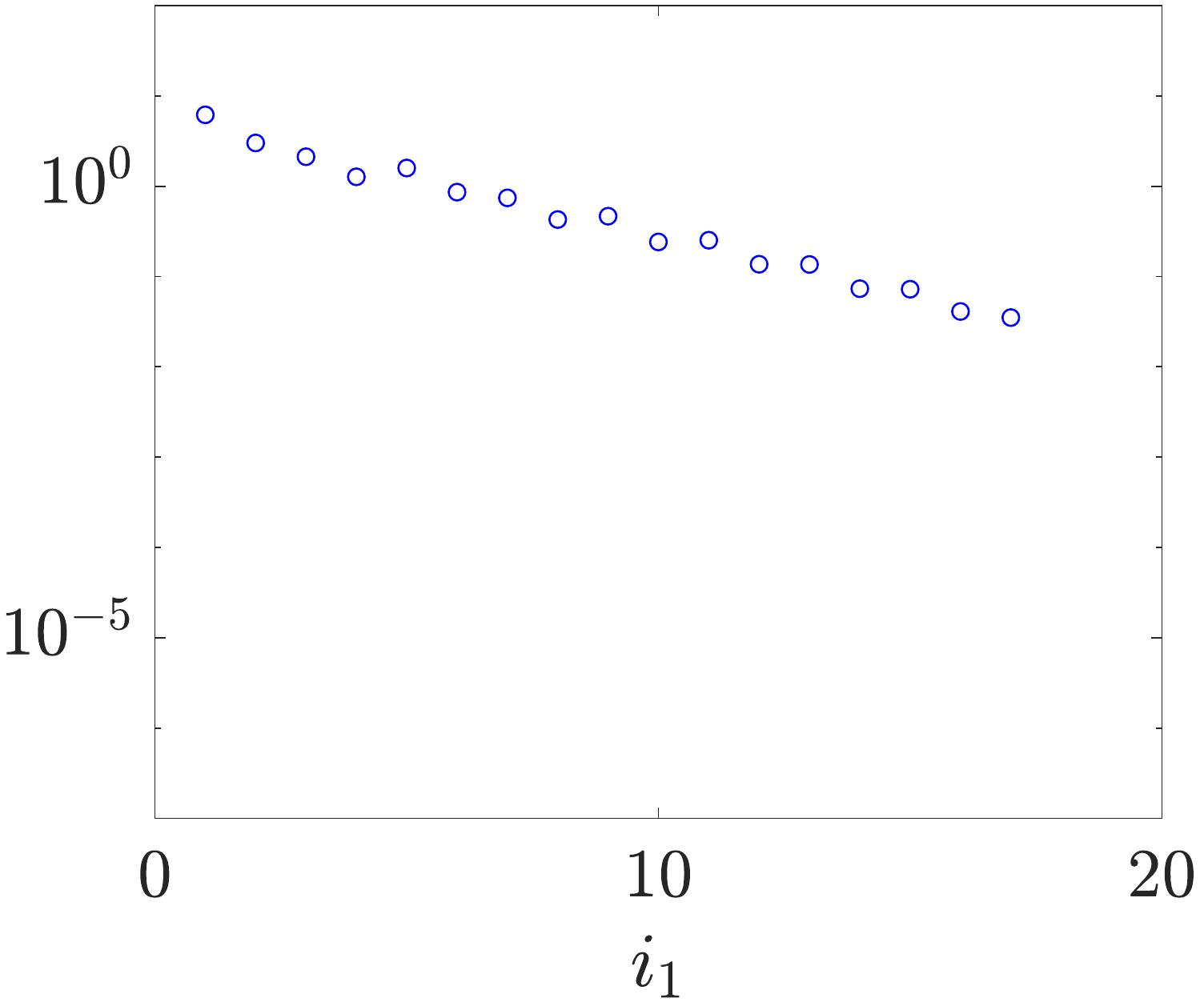}
}
\caption{First four modes of the constant-rank ($r_1=17$) 
DO-TT representation of  the solution to the PDE \eqref{2dpde}, 
and corresponding spectrum. Shown are results 
at times $t = 0.0$, $t = 0.5$ and $t = 1.0$.}
\label{fig:2d_mode_ev}
\end{figure}

\begin{figure}[t]
\centerline{\footnotesize\hspace{0.4cm}$t=0.0$ \hspace{3.9cm} $t=0.5$  \hspace{4.1cm} $t = 1.0$}
	\centerline{
	\rotatebox{90}{\hspace{1.3cm}  \footnotesize}
		\includegraphics[width=0.3\textwidth]{2d_lhs_modes-eps-converted-to.pdf}
		\includegraphics[width=0.3\textwidth]{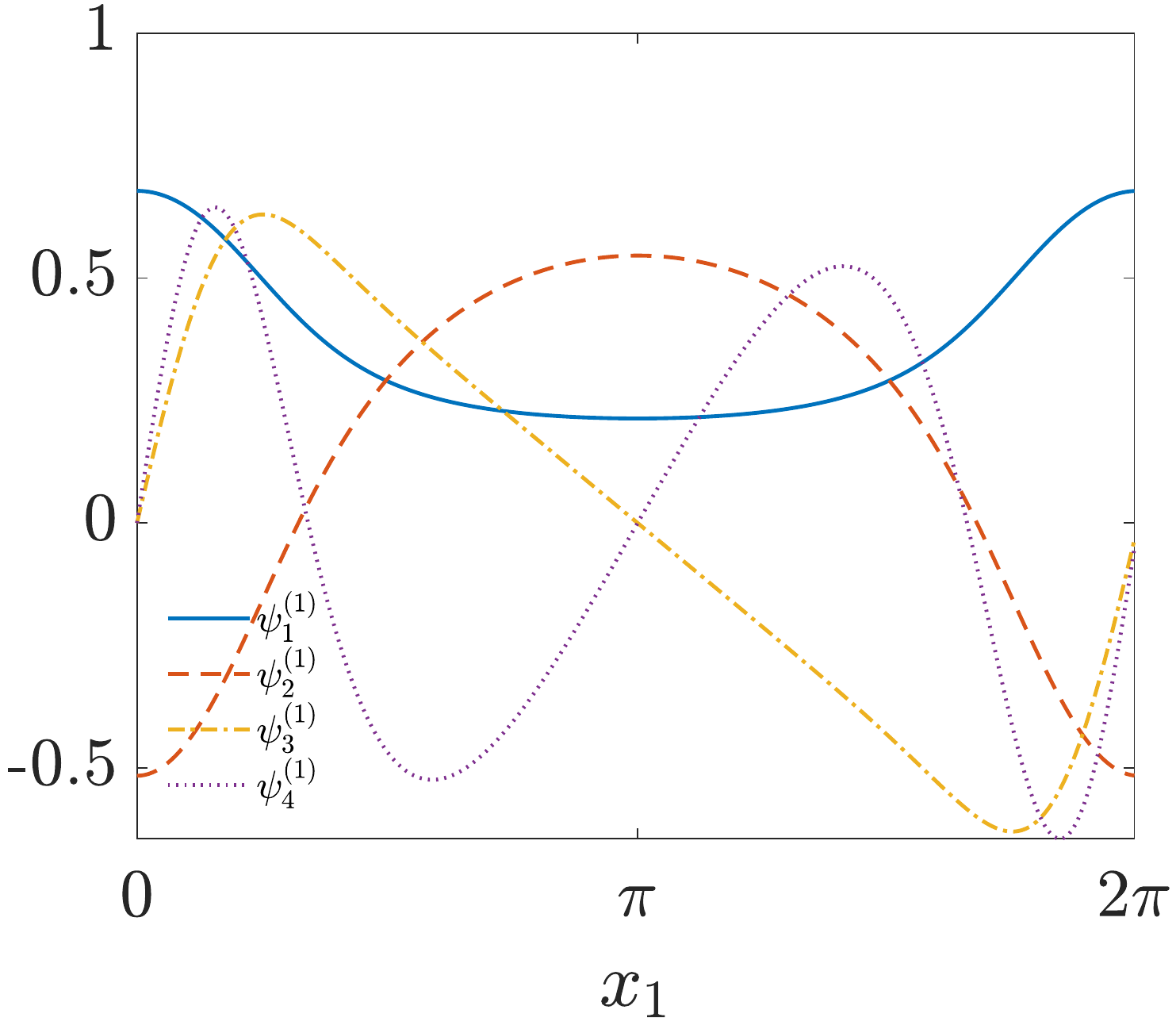}
		\includegraphics[width=0.3\textwidth]{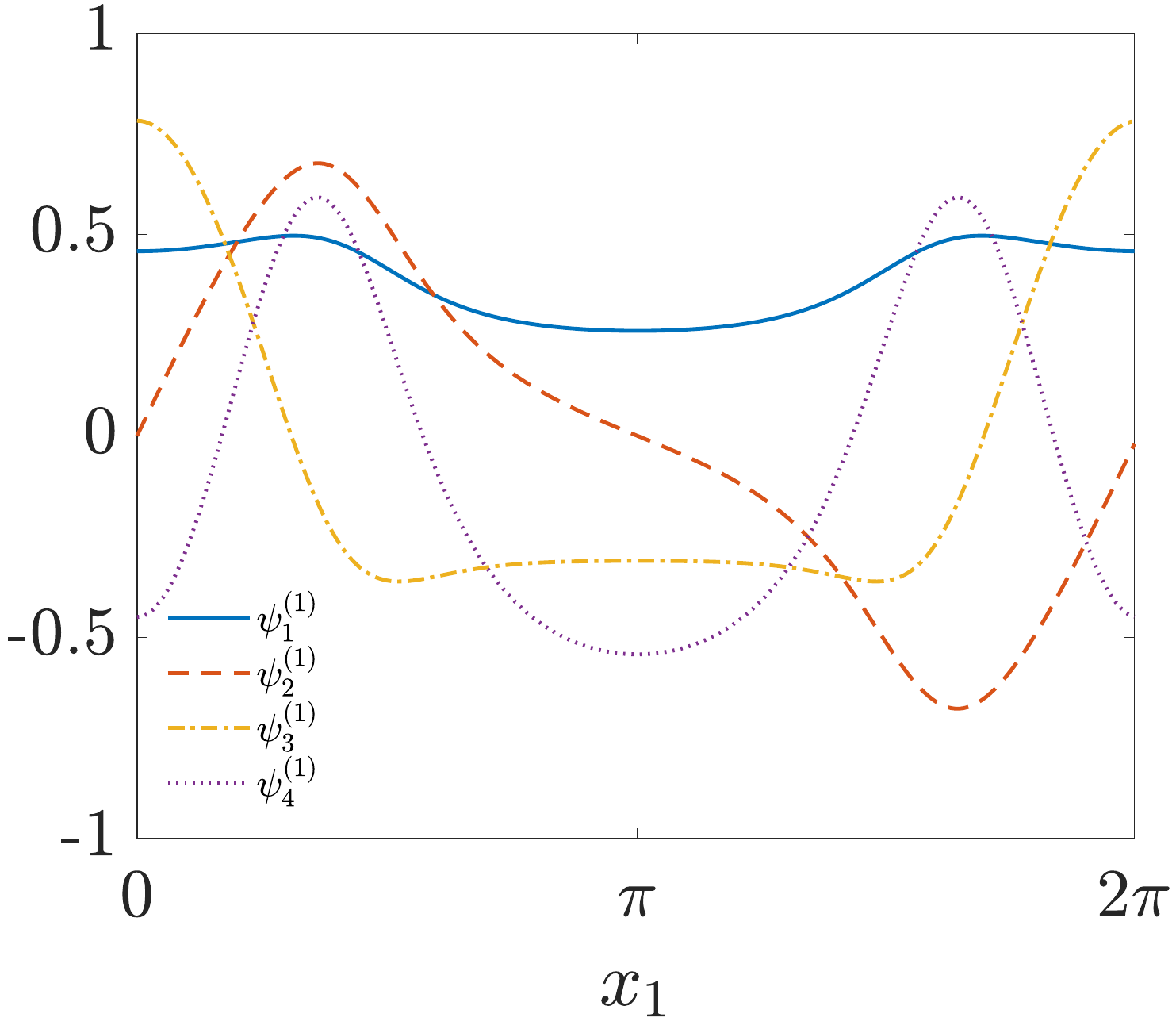}	
}
	\centerline{
	\rotatebox{90}{\hspace{1.3cm}  \footnotesize }
		\includegraphics[width=0.3\textwidth]{2d_rhs_modes-eps-converted-to.pdf}
	    \includegraphics[width=0.3\textwidth]{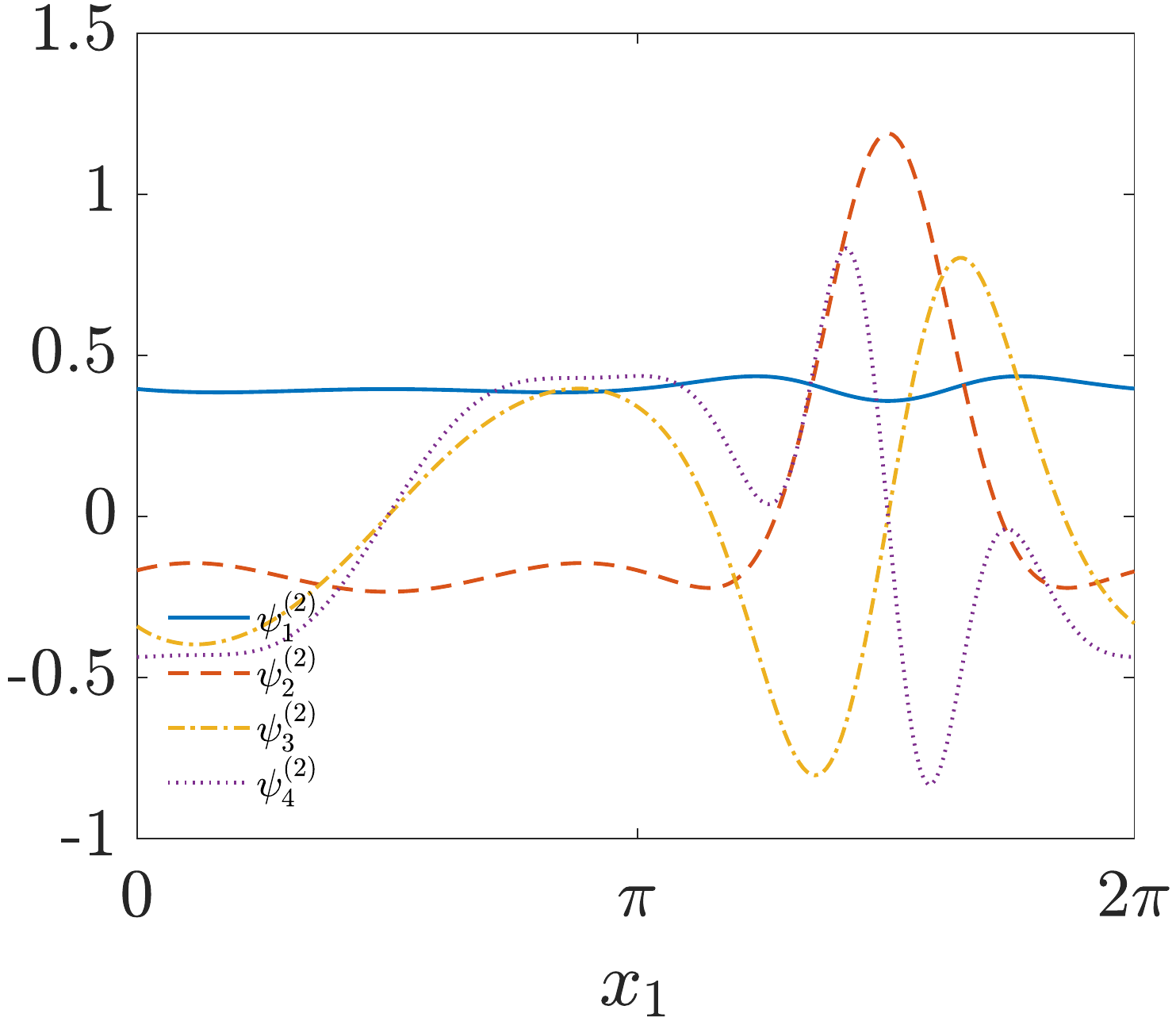}
	    \includegraphics[width=0.3\textwidth]{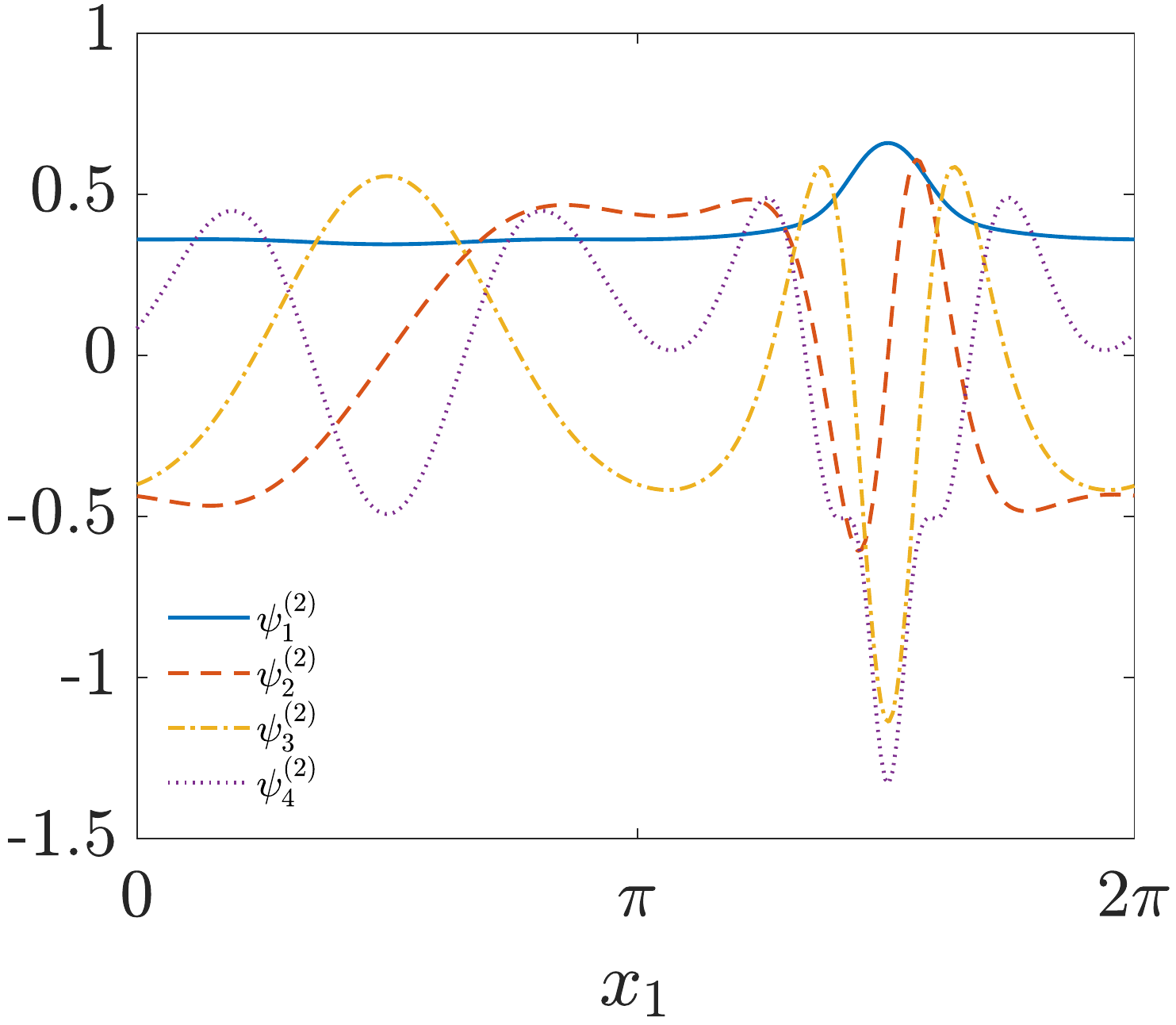}
}
	\centerline{
	\rotatebox{90}{\hspace{2.1cm} \footnotesize $\lambda_{i_1}$ }
		\includegraphics[width=0.3\textwidth]{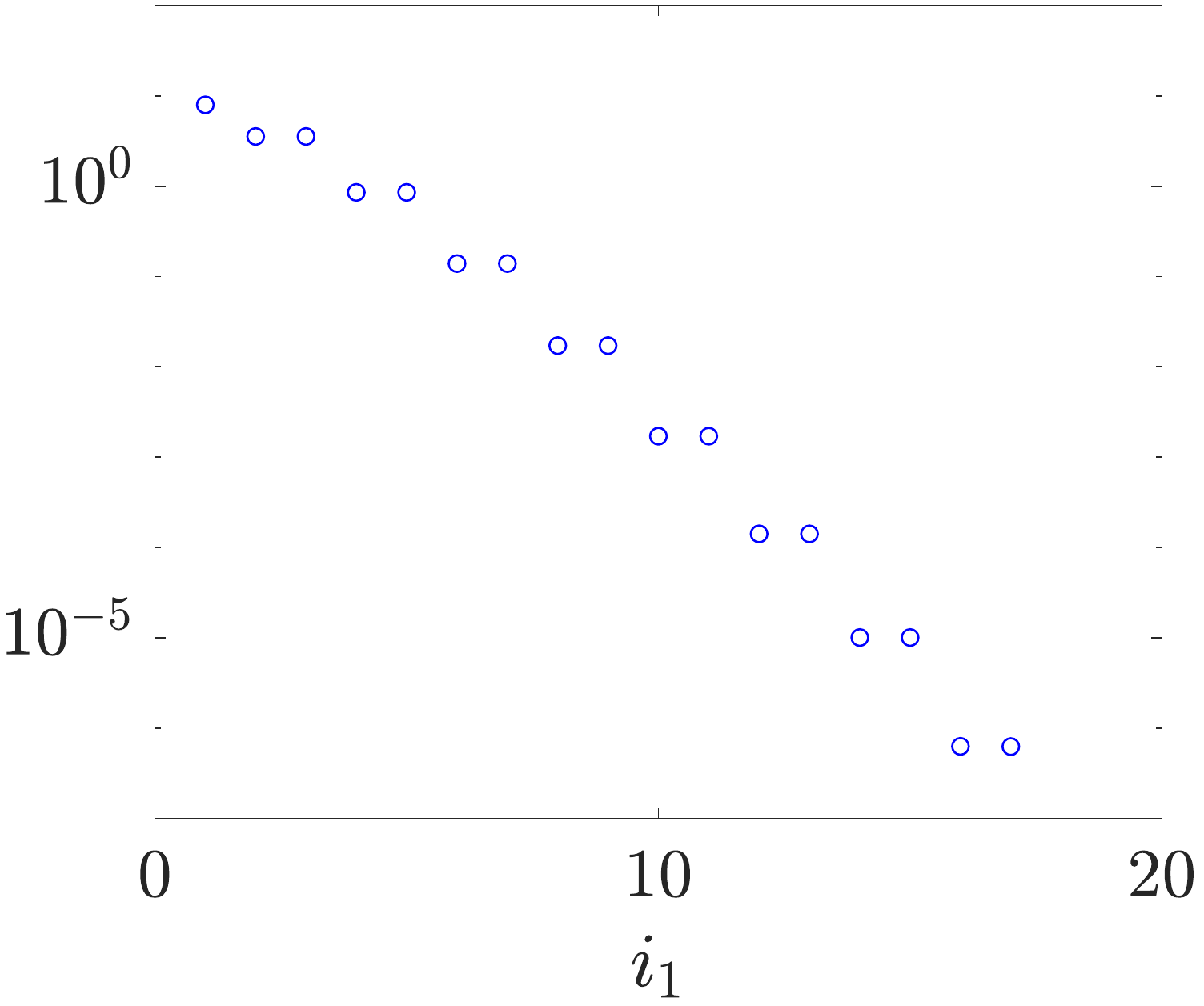}	
	    \includegraphics[width=0.3\textwidth]{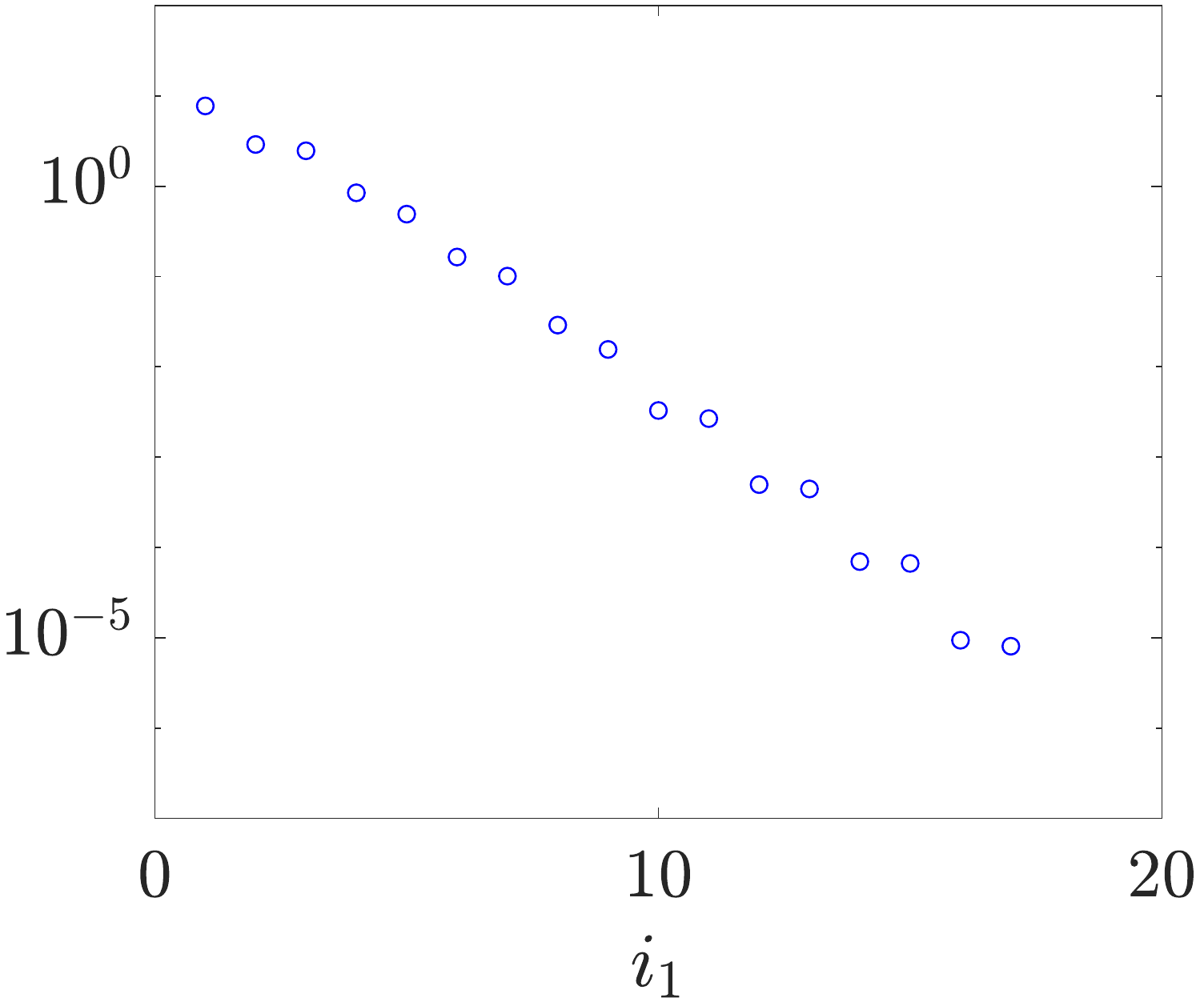}
	    \includegraphics[width=0.3\textwidth]{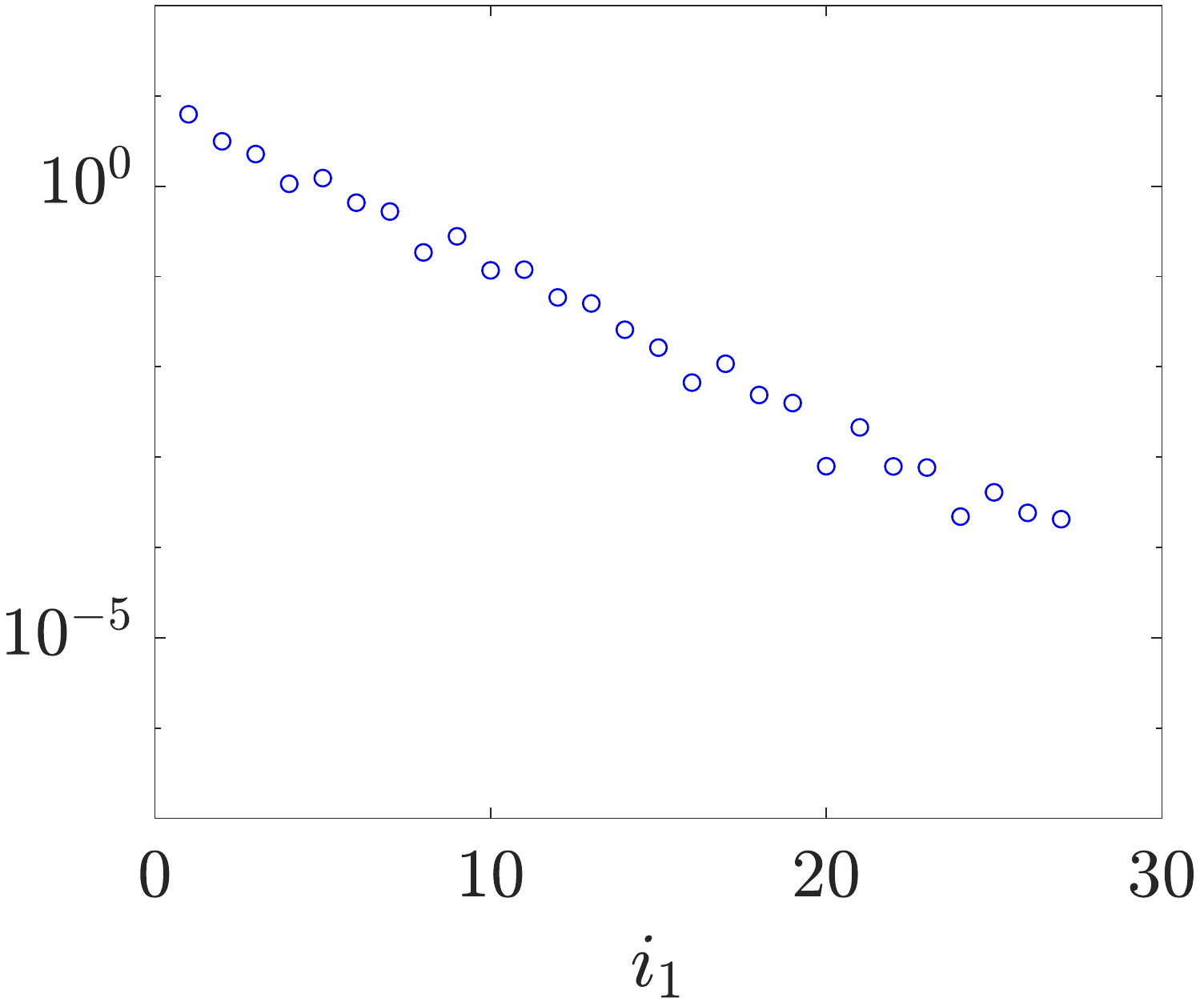}
}
\caption{First four modes of the adaptive-rank DO-TT representation 
of the solution to \eqref{2dpde}, and corresponding spectrum. 
The adaptive expansion is computed by using our Algorithm 2  
in Section \ref{sec:add_modes}. Shown are results at 
times $t = 0.0$, $t = 0.5$ and $t = 1.0$. It is seen that
the proposed algorithm can control the spectral 
decay at $t=1$ (compare the spectra in the last row of this Figure 
with the spectra in Figure \ref{fig:2d_mode_ev}).}
\label{fig:2d_mode_ev_ht_adaptive}
\end{figure}

\begin{figure}[t]
\centerline{\footnotesize\hspace{0.2cm}$\psi^{(2)}_2$  \hspace{7.cm} $\psi^{(2)}_3$ }
	\centerline{
	\rotatebox{90}{\hspace{1.3cm}  \footnotesize}
		\includegraphics[width=0.4\textwidth]{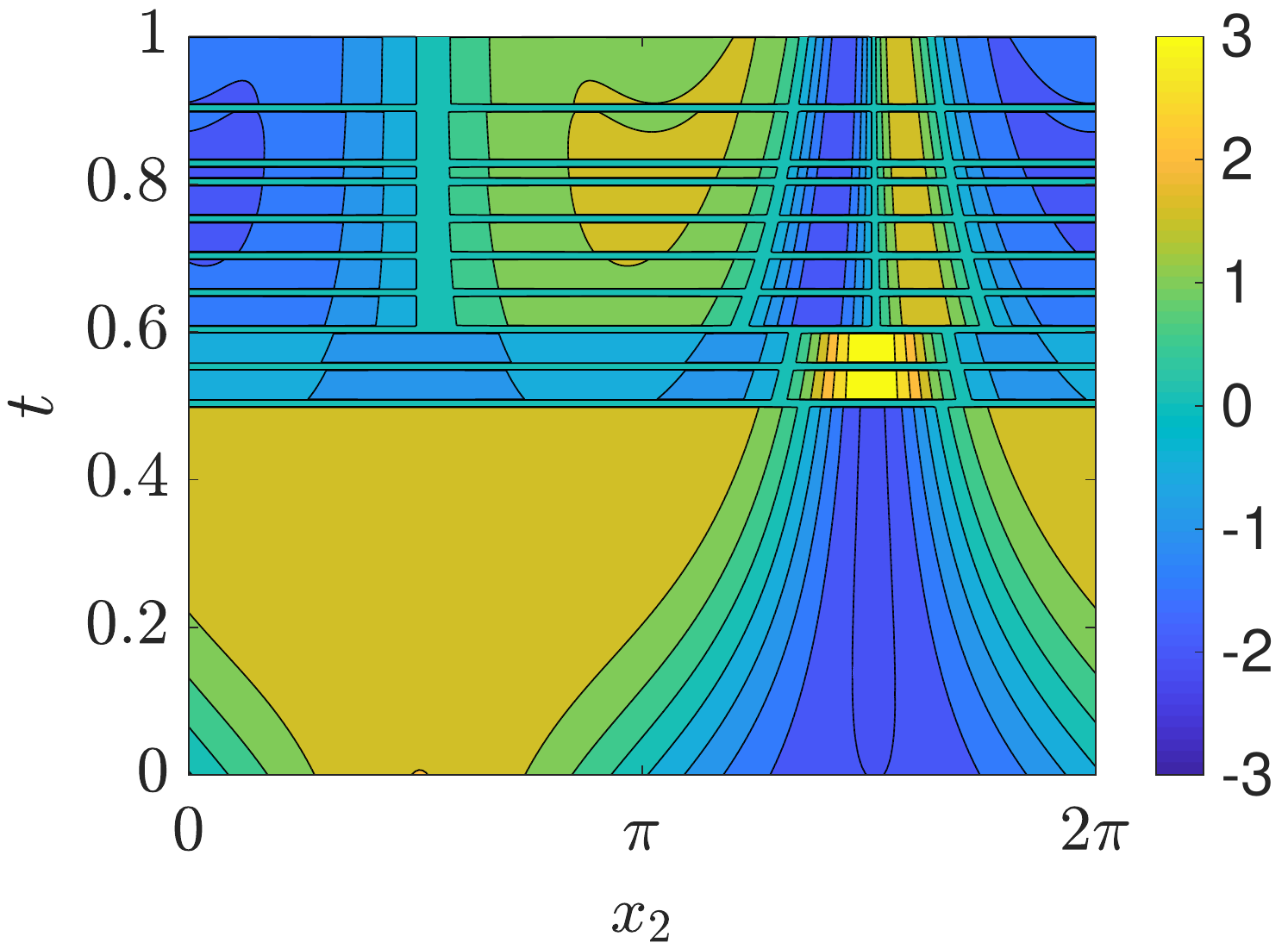}\hspace{1cm}
		\includegraphics[width=0.4\textwidth]{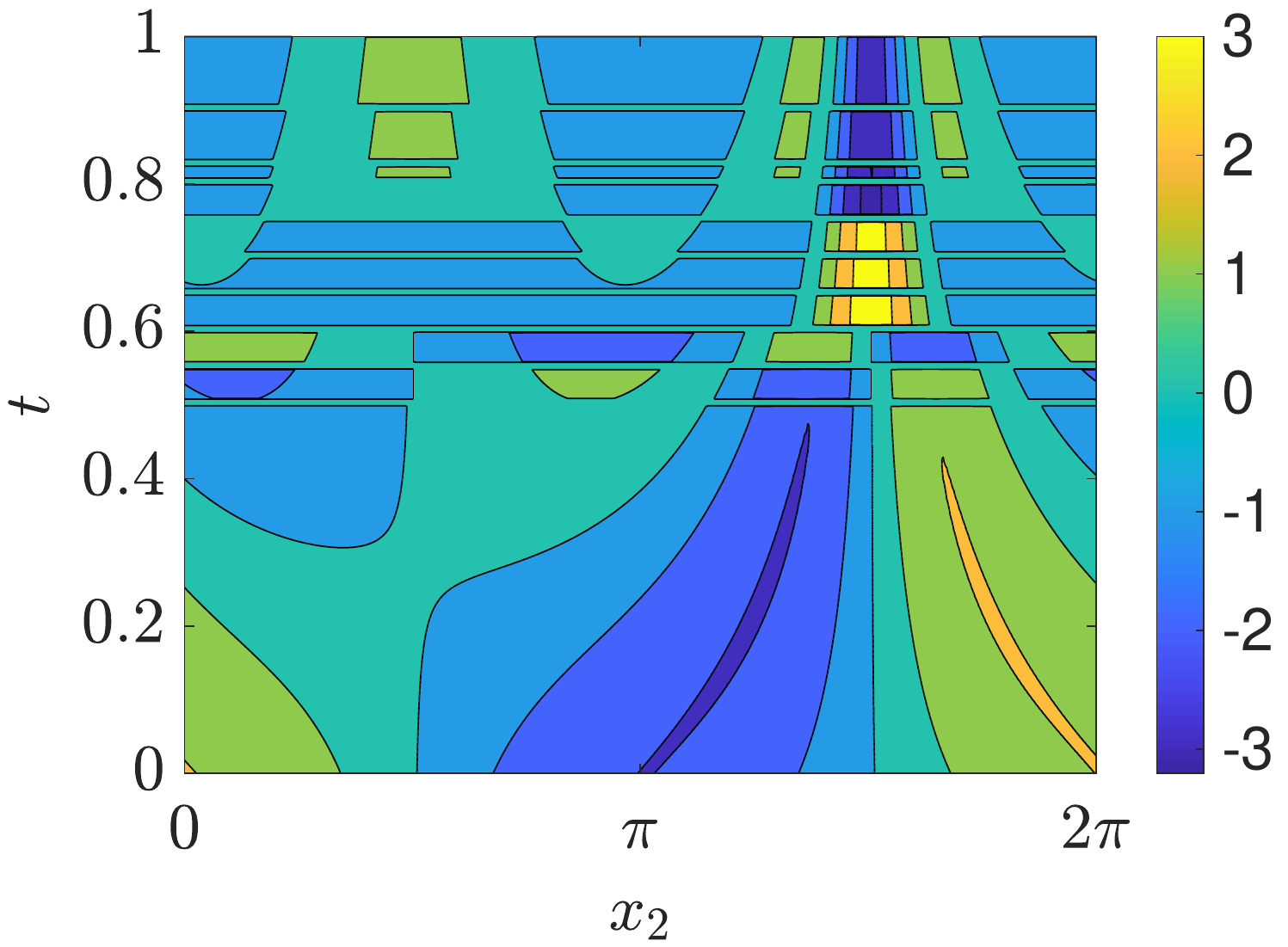}		
}
\caption{Adaptive DO-TT simulation of the PDE \eqref{2dpde}. 
Time evolution of the modes 
$\psi_2^{(2)}$ and  $\psi_3^{(2)}$ we obtained by using 
Algorithm 2 in Section \ref{sec:add_modes}. In particular, 
here we add one mode at time 
$t = 0.5$, $t = 0.55$, $t = 0.6$, $t = 0.65$, $t = 0.7$, 
$t = 0.75$, $t = 0.8$, $t = 0.825$, $t = 0.85$, $t = 0.9$ 
(10 modes total). When a new mode is added, there is a 
re-orthogonalization process that can yield a discontinuity 
in the temporal evolution of each mode. In practice, the DO-TT 
system is re-started from a new initial 
condition after such re-orthogonalization takes place. 
This does not create any temporal discontinuity in the 
solution, nor any error jump (see Figure \ref{fig:2d_error}).}
\label{fig:discontinuous_adaptive_modes}
\end{figure}

\begin{figure}
\centerline{\hspace{0.3cm}\hspace{8cm}}
\centerline{
\includegraphics[width=0.5\textwidth]{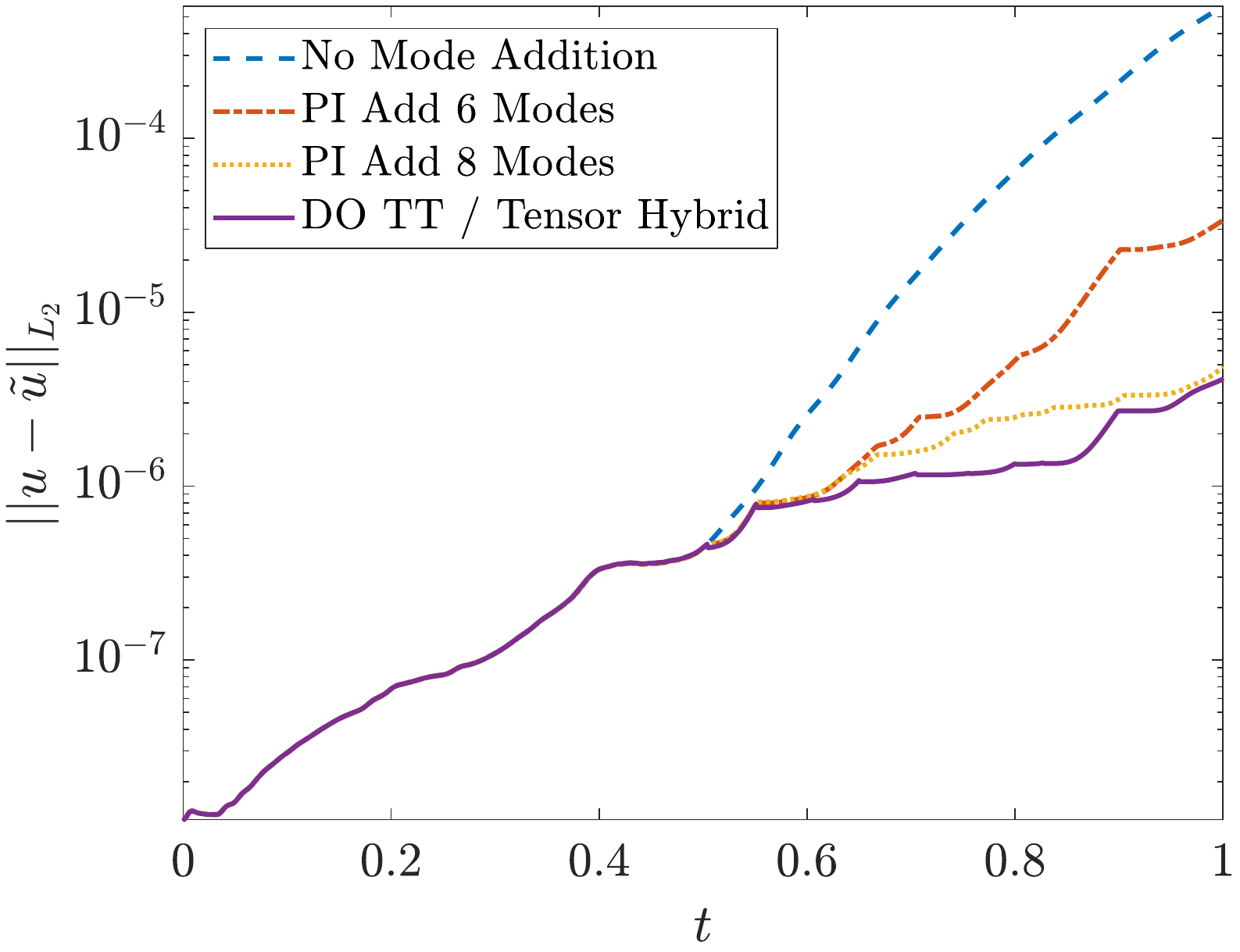}
}
\caption{Two-dimensional hyperbolic PDE \eqref{2dpde}. Time-dependent 
$L^2(\Omega)$ errors between various DO-TT simulations and 
the semi-analytical solution \eqref{2Dsol}.  In particular, we plot the constant rank solution obtained by setting $r_1=17$ in \eqref{do_tt_2d_ex} (dashed line); the DO-TT solution we obtained by adding one extra mode at $t = 0.5$, $t = 0.55$, $t = 0.6$, $t = 0.65$, $t = 0.7$, $t = 0.75$ (6 extra modes total) using the pseudo-inverse (PI) algorithm proposed in \cite{robust_do/bo}; the DO-TT solution  we obtained by adding one extra mode at $t = 0.5$, $t = 0.55$, $t = 0.6$, $t = 0.65$, $t = 0.7$, $t = 0.75$, $t = 0.8$, $t = 0.825$, $t = 0.85$, $t = 0.9$ using the pseudo-inverse algorithm (dashed line) and our Algorithm 2 in Section \ref{sec:add_modes}. }
\label{fig:2d_error}
\end{figure}

\subsubsection{Four-dimensional hyperbolic PDE}
\label{sec:4dhyperbolic}
Let us consider the following four-dimensional linear 
hyperbolic PDE 
\begin{equation}
\label{4d_pde}
 \begin{cases}
\displaystyle\frac{\partial {u(\bm x,t)}}{\partial t} =  
\sum_{i,j = 1}^4 c_{ij} f_j(x_j) \frac{\partial u(\bm x,t)}{\partial x_i} 
\vse\\
\displaystyle 
u(\bm x,0) = \exp\left[-\frac{1}{10} \sin(x_1 + x_2 + x_3 + x_4)\right] 
\end{cases}
\end{equation}
in the spatial  domain $\Omega=[0,2\pi]^4$, 
with periodic boundary conditions. 
In equation \eqref{4d_pde} $c_{ij}$ are real numbers and $f_j(x_j)$ are 
real-valued functions.  The system of DO-TT evolution 
equations \eqref{do_PDE} can be explicitly written for the linear 
PDE \eqref{4d_pde}. Such system is rather complicated, and 
therefore not presented here. For numerical 
demonstration, we set the coefficient 
matrix $c_{ij}$ as 
$$\bm c = \begin{bmatrix}
0 & 0.5 & 0 & 0\\
0 & 0 & - 0.3 & 0 \\
0 & 0 & 0 & -1 \\
0.5 & 0 & 0 & 0 
\end{bmatrix}
$$
and consider the following functions
\begin{align*}
f_1(x_1) = \sin(x_1), \qquad 
f_2(x_2) = \cos(2x_2), \qquad 
f_3(x_3) = \sin(3x_3), \qquad 
f_4(x_4) = \cos(4x_4).
\end{align*}
This yields non-trivial dynamics ($c_{ij}$ is not diagonal). 
We solve the DO-TT system \eqref{do_PDE} numerically 
using a Fourier spectral collocation method with $20$ Fourier points 
in each variable $x_i$, and RK4 time integration with $\Delta t = 10^{-3}$. 
To this end, we first decompose the four-dimensional 
initial condition in \eqref{4d_pde} with the bi-orthogonal method 
we discussed in Section \ref{sec:recursive}. 
Specifically, we consider the space $H^{0}(\Omega)=L_2(\Omega)$ 
and set the eigenvalue threshold to $\sigma = 10^{-10}$. This 
yields the following hierarchical ranks
\begin{align}
r_1 = 9,
\qquad 
r_2 = \begin{bmatrix} 1 & 2 & 2 & 2 & 2 & 2 & 2 & 2 &2 \end{bmatrix}, 
\qquad 
r_3  = \begin{bmatrix} 1 & 2 & 2 & 2 & 2 & 2 & 2 & 2 & 2 \\
0 & 2 & 2 & 2 & 2 & 2 & 2 & 2 & 2
\end{bmatrix}^T,
\label{ranks}
\end{align}
which we keep constant throughout the simulation. In other words, 
here we do not perform adaptive addition/removal of DO-TT 
modes as we did in the previous PDE example 
(Eqs. \eqref{2dpde} and \eqref{do_tt_2d_ex}). 
In Figure \ref{fig:4d_modes} we plot the time 
evolution of a few representative DO-TT modes. 
\begin{figure}[t]
\centerline{\footnotesize\hspace{0.2cm}$\psi_{21}^{(2)}$ \hspace{4.8cm} $\psi^{(4)}_{211}$  \hspace{4.8cm} $\psi^{(4)}_{311}$}
	\centerline{
	\rotatebox{90}{\hspace{1.3cm}  \footnotesize}
		\includegraphics[width=0.33\textwidth]{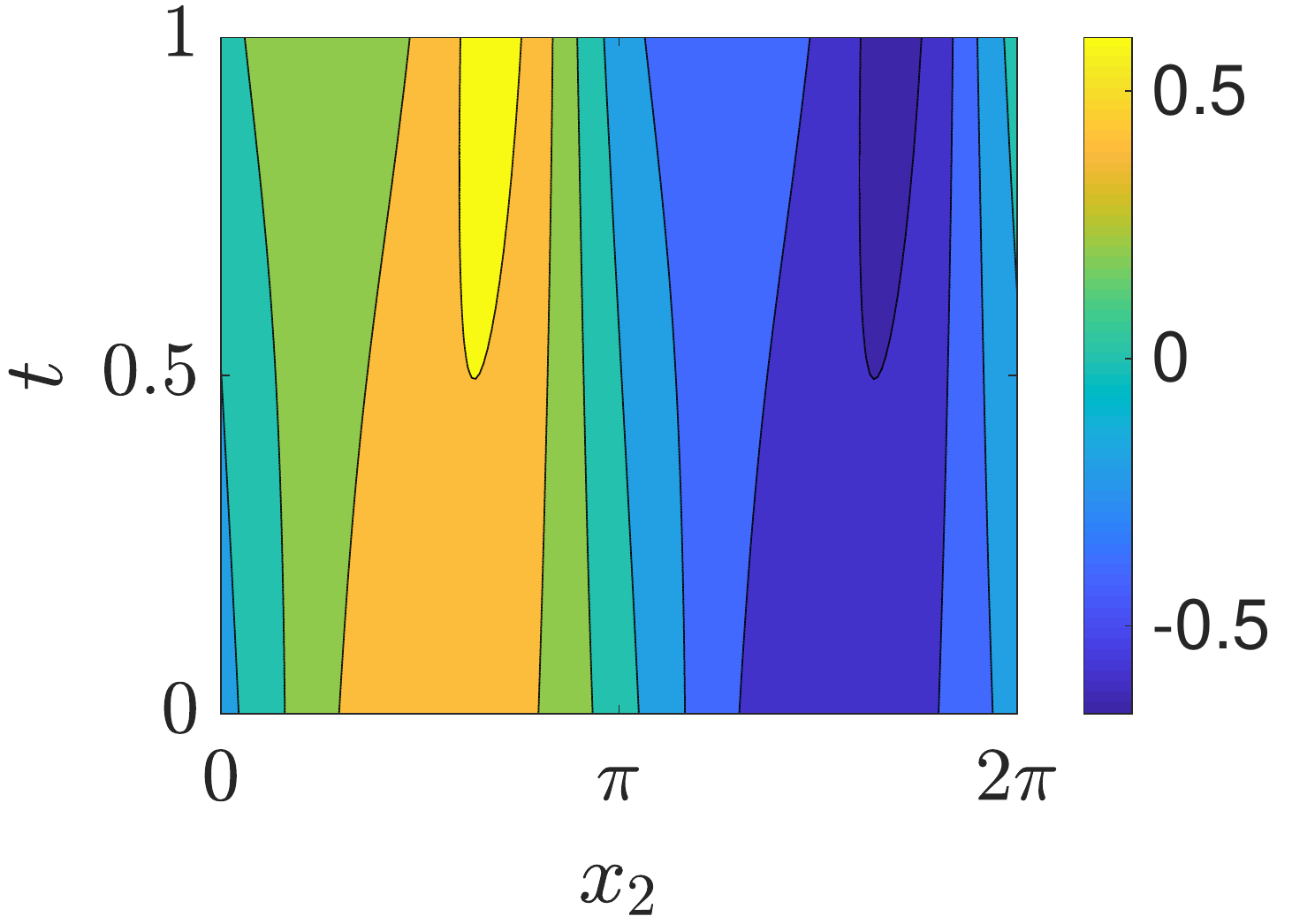}
		\includegraphics[width=0.33\textwidth]{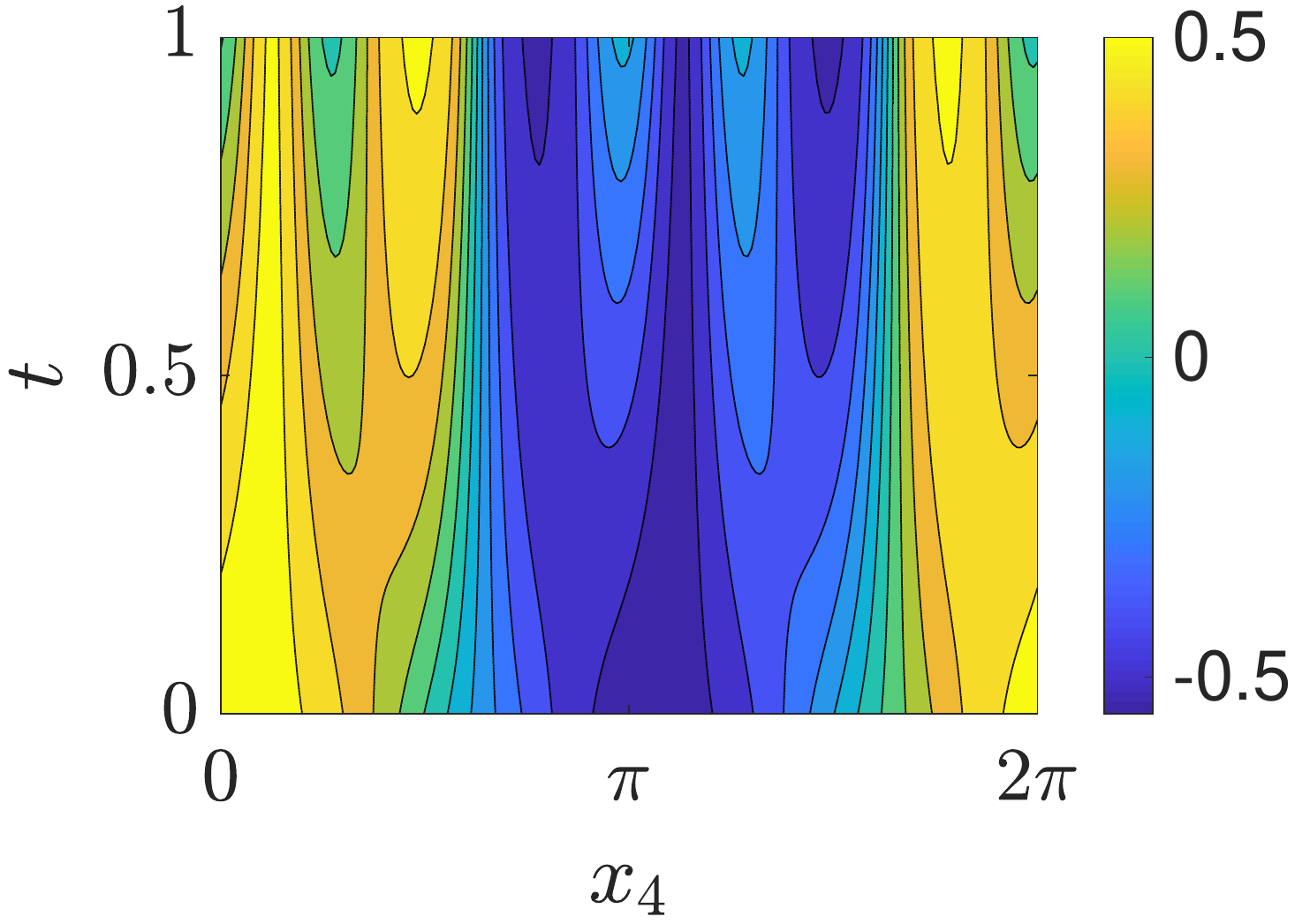}
		\includegraphics[width=0.33\textwidth]{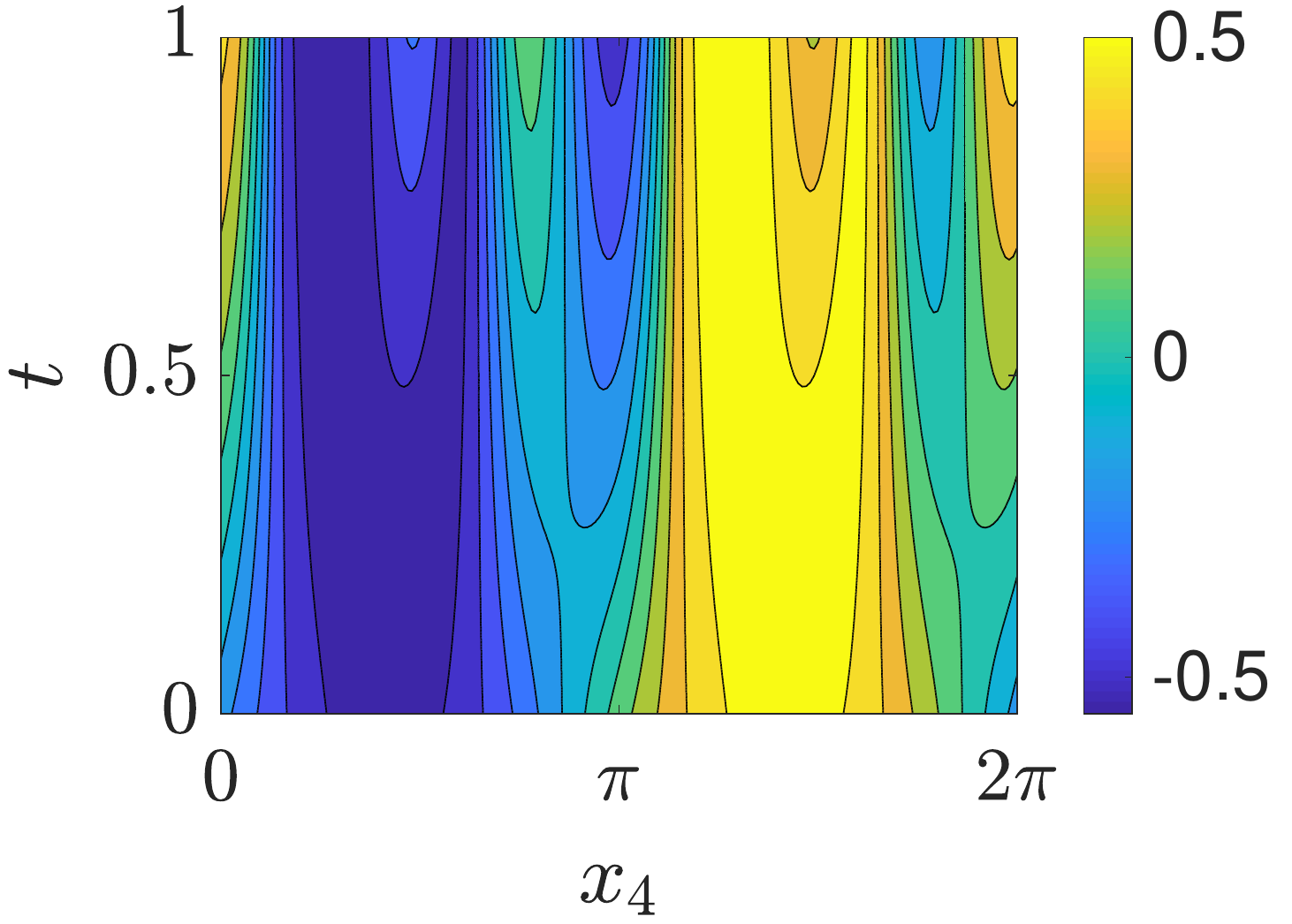}	
}
\caption{Time evolution of a few representative DO-TT modes generated by the four-dimensional hyperbolic initial/boundary value problem \eqref{4d_pde}.}
\label{fig:4d_modes}
\end{figure}
It is seen that as time evolves the spatial frequency of 
such modes increases which suggests that the
hyperbolic dynamics activate higher spatial 
frequencies in the spectral representation of the PDE 
solution. This is shown in Figure \ref{fig:4d_solution}, 
where we plot one section of the solution to the initial value problem 
\eqref{4d_pde} we obtained with the method of 
characteristics (benchmark solution), 
and the solution we obtained with the DO-TT propagator. 
It is seen that even with the very few hierarchical ranks 
summarized in \eqref{ranks} we were able to resolve the 4D solution to 
a reasonable accuracy (see Figure \ref{fig:4d_error}(a)).
\begin{figure}[t]
\hspace{0.5cm}
\centerline{\footnotesize\hspace{-0.5cm}$t = 0.0$ \hspace{4.5cm} $t = 0.5$  \hspace{4.5cm} $t = 1.0$ }

\hspace{0.5cm} 
\centerline{
	\rotatebox{90}{\hspace{0.7cm}\footnotesize Method of characteristics }
	    \includegraphics[width=0.34\textwidth]{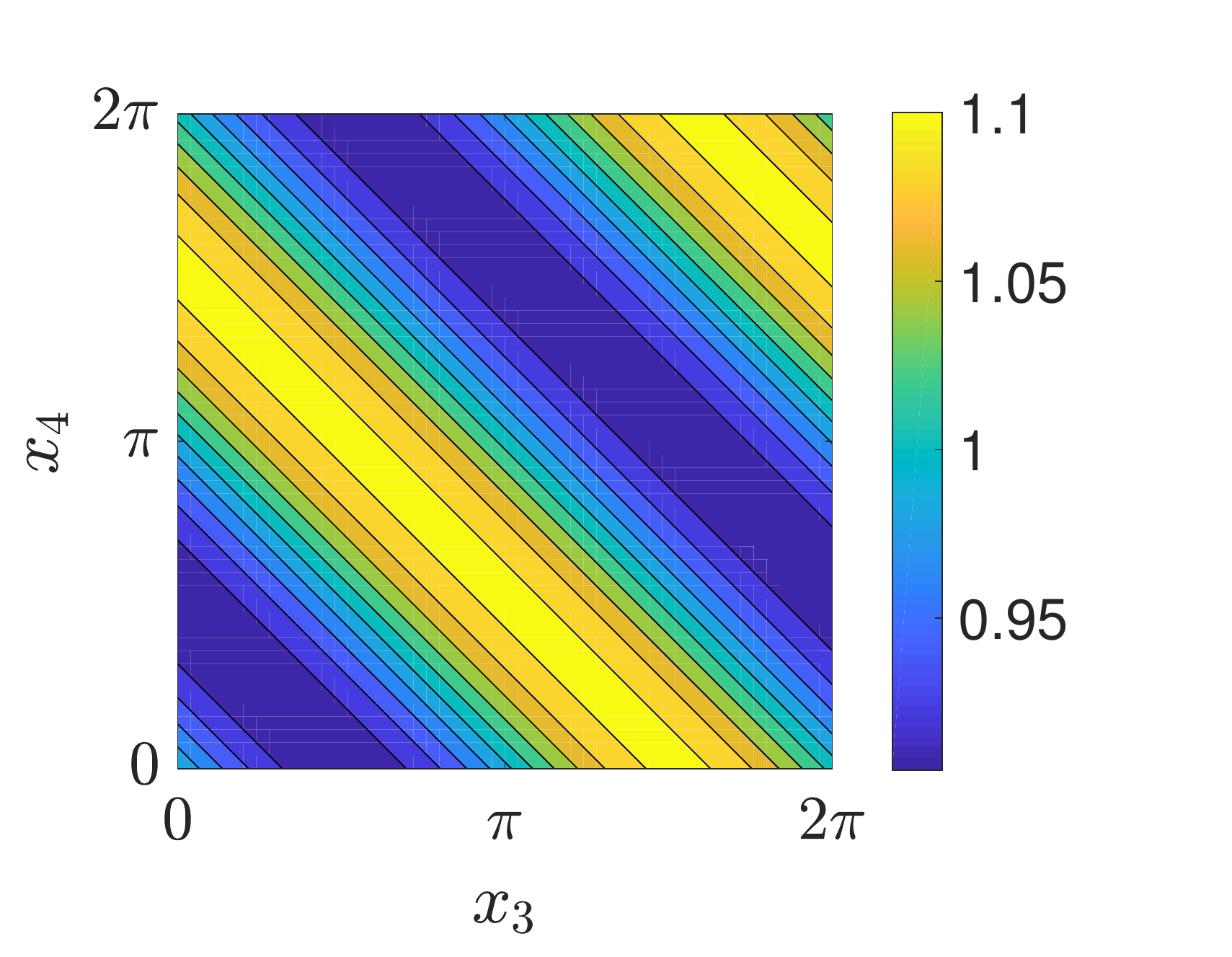}
		\includegraphics[width=0.34\textwidth]{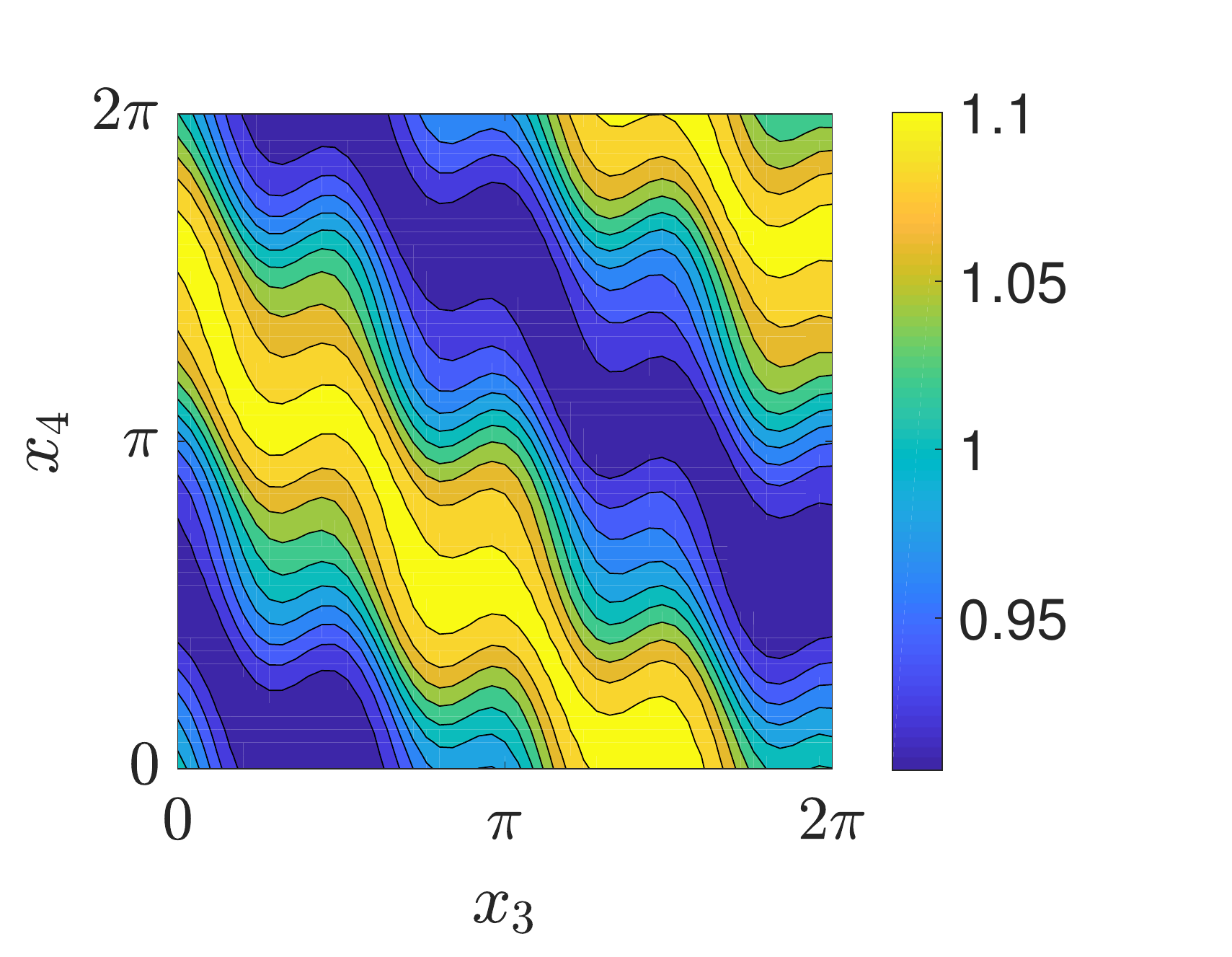}
		\includegraphics[width=0.34\textwidth]{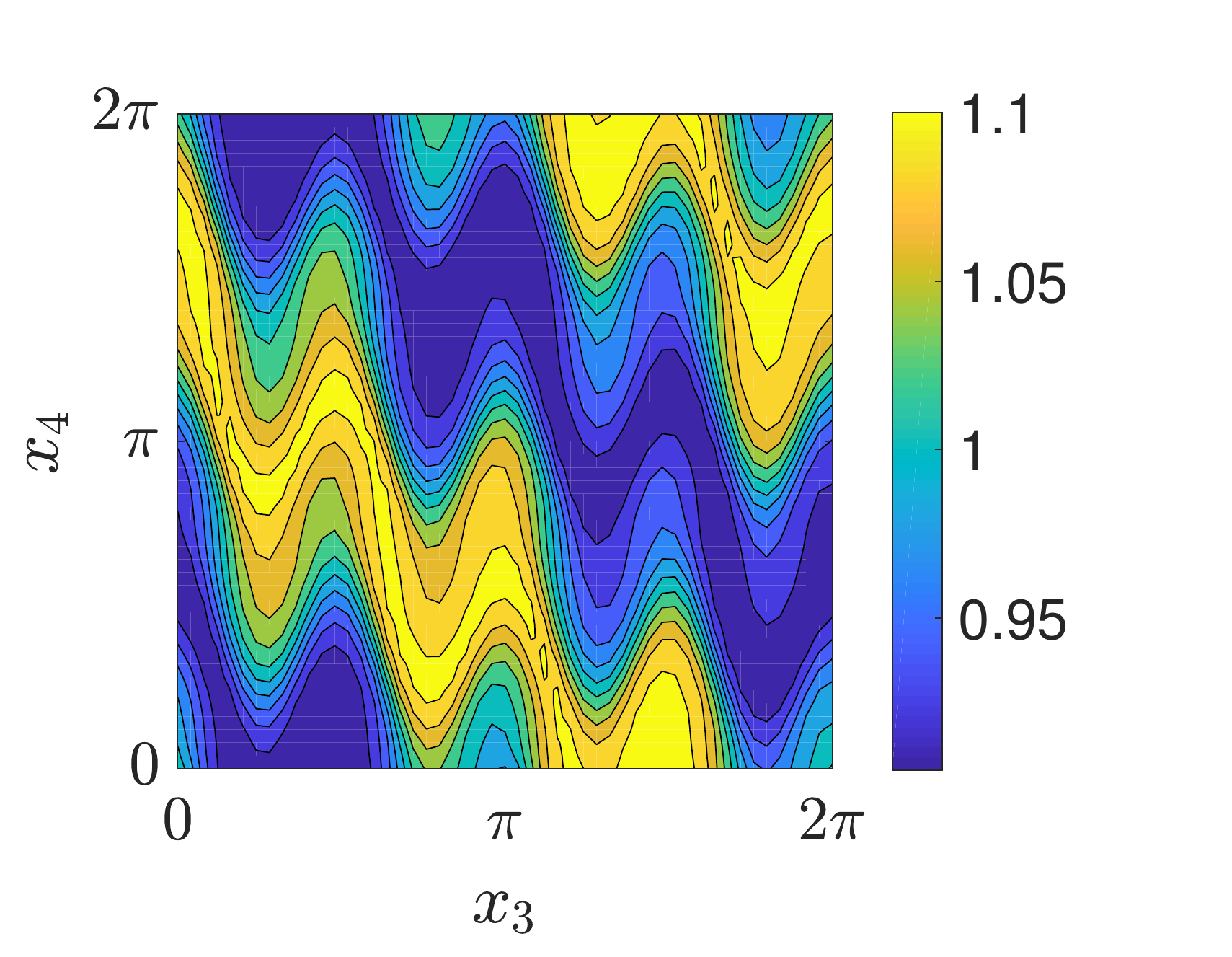}	
}
	
\hspace{0.5cm}	
\centerline{
	\rotatebox{90}{\hspace{1.8cm}  \footnotesize  DO-TT }
		\includegraphics[width=0.34\textwidth]{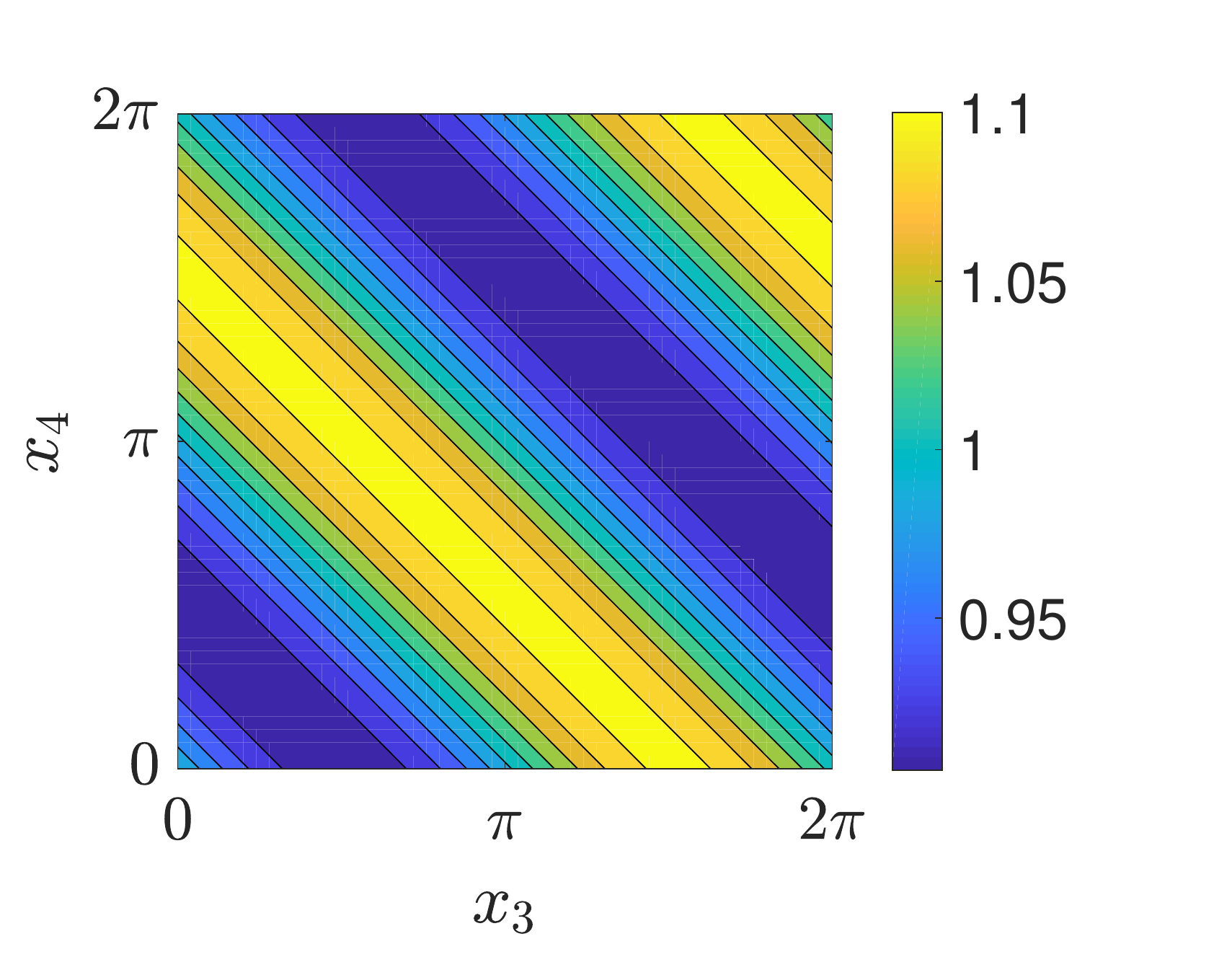}
		\includegraphics[width=0.34\textwidth]{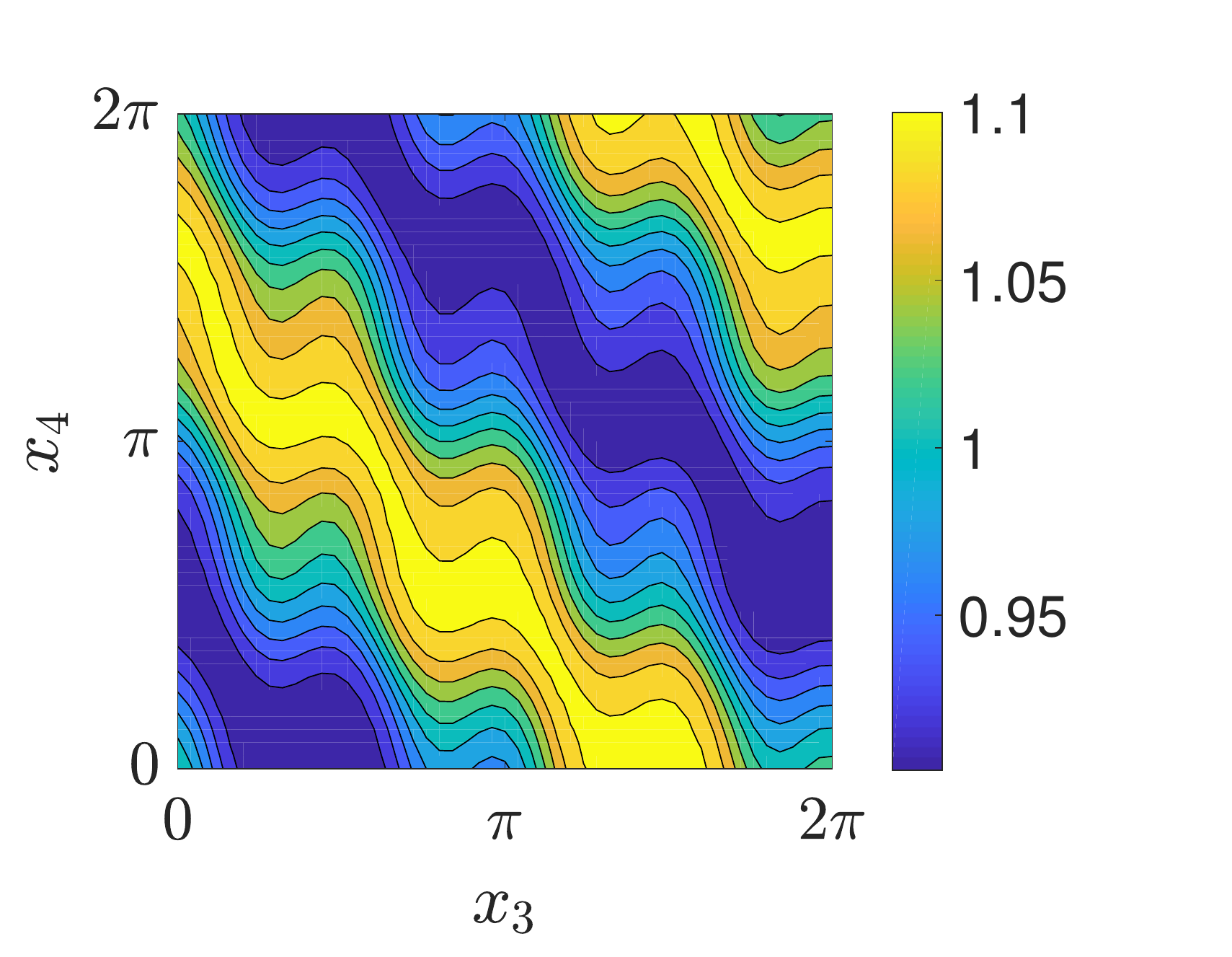}
		\includegraphics[width=0.34\textwidth]{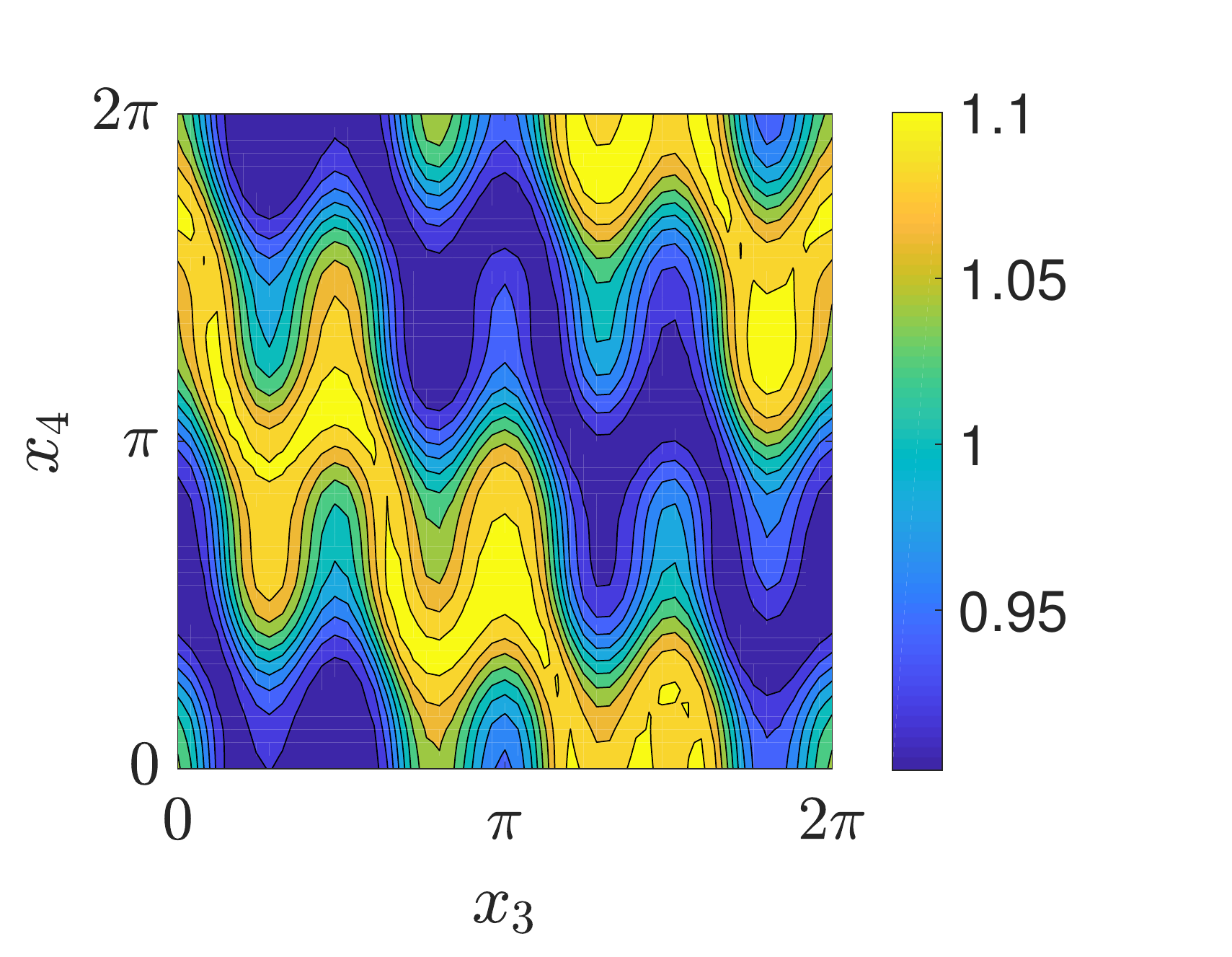}
}

\hspace{0.5cm}	
\centerline{
	\rotatebox{90}{\hspace{1.2cm}  \footnotesize Pointwise error}
        \includegraphics[width=0.34\textwidth]{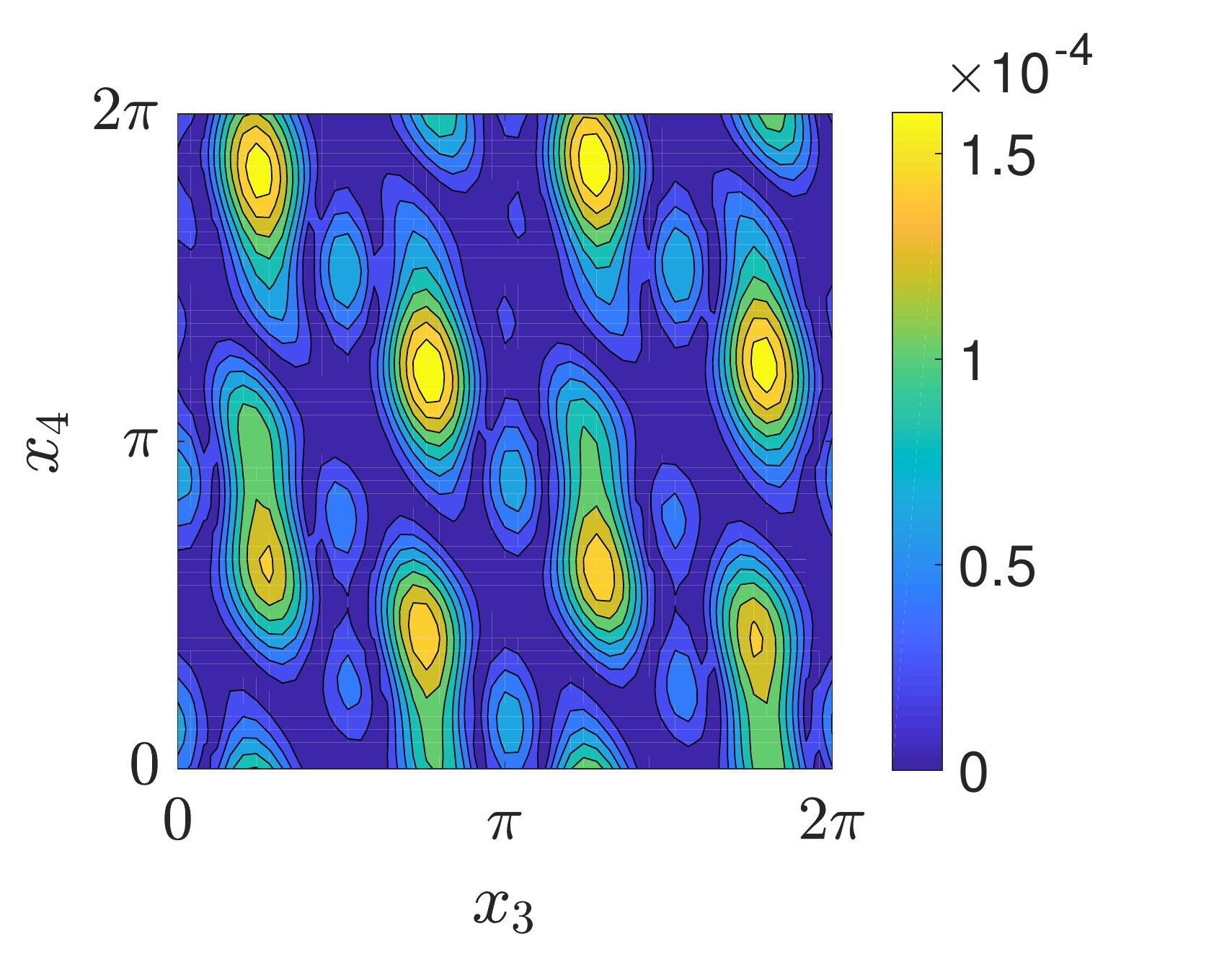}
	    \includegraphics[width=0.34\textwidth]{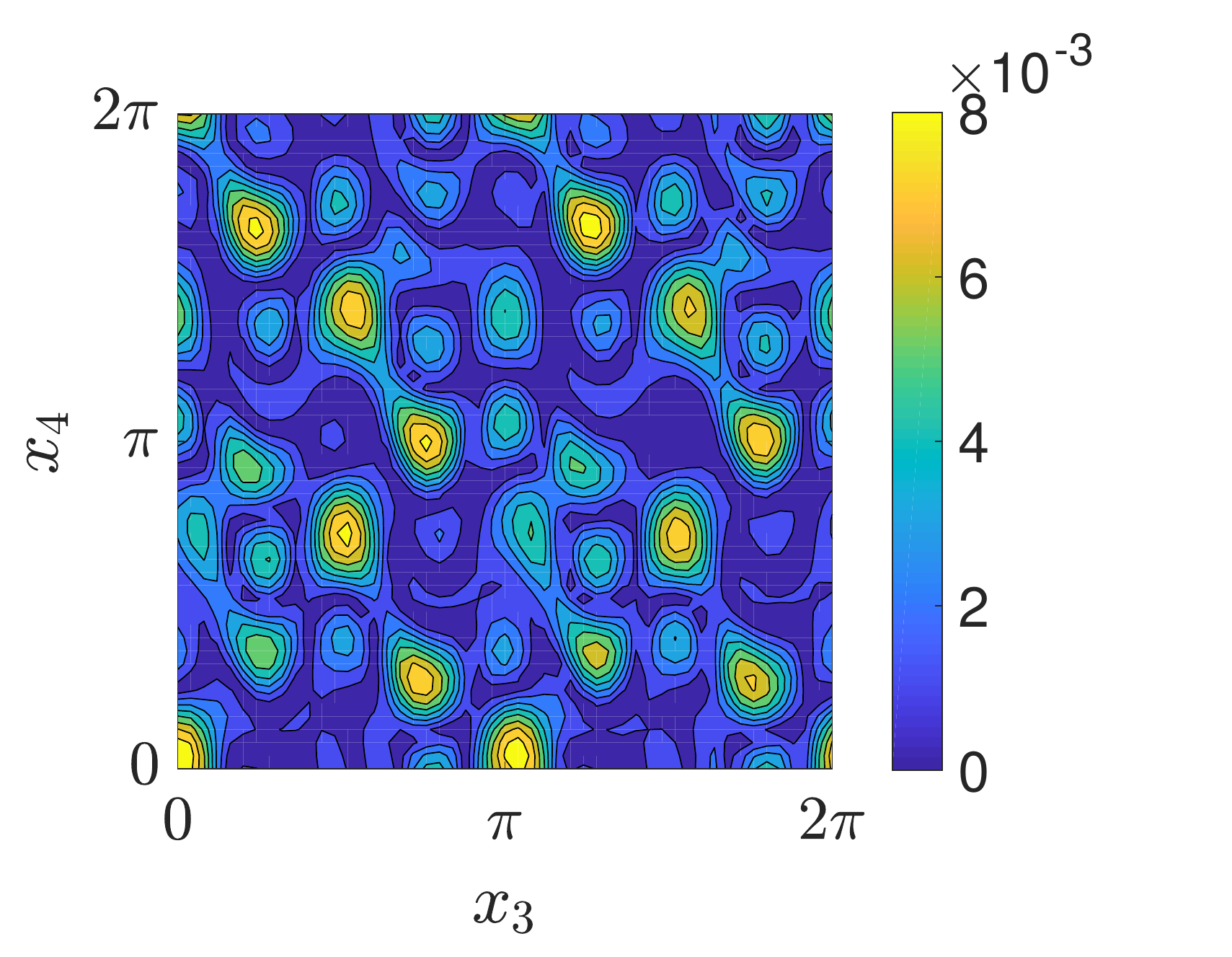}
	    \includegraphics[width=0.34\textwidth]{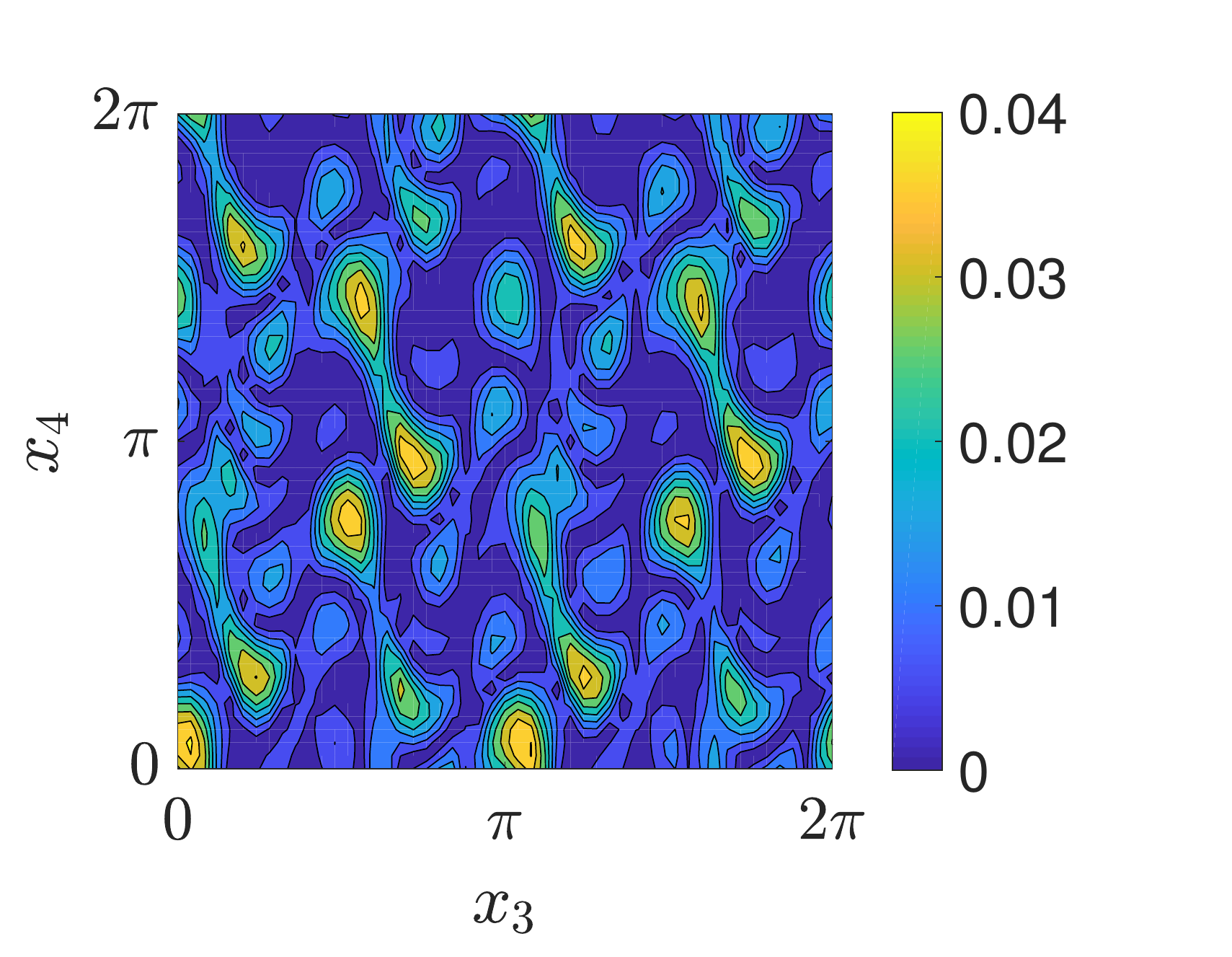}
}
\caption{Time snapshots of a 2D slice ($x_1= 2.9568$, $x_2=2.9568$) of the solution to the four-dimensional 
PDE \eqref{4d_pde} obtained using method of characteristics 
(benchmark solution),  and the DO-TT propagator with 
hierarchical ranks \eqref{ranks}. We also plot the maximum 
pointwise error between the two solutions.}
\label{fig:4d_solution}
\end{figure}

\begin{figure}
\centerline{\hspace{0.4cm}\footnotesize (a)\hspace{8cm} (b)}
\centerline{
\includegraphics[width=0.5\textwidth]{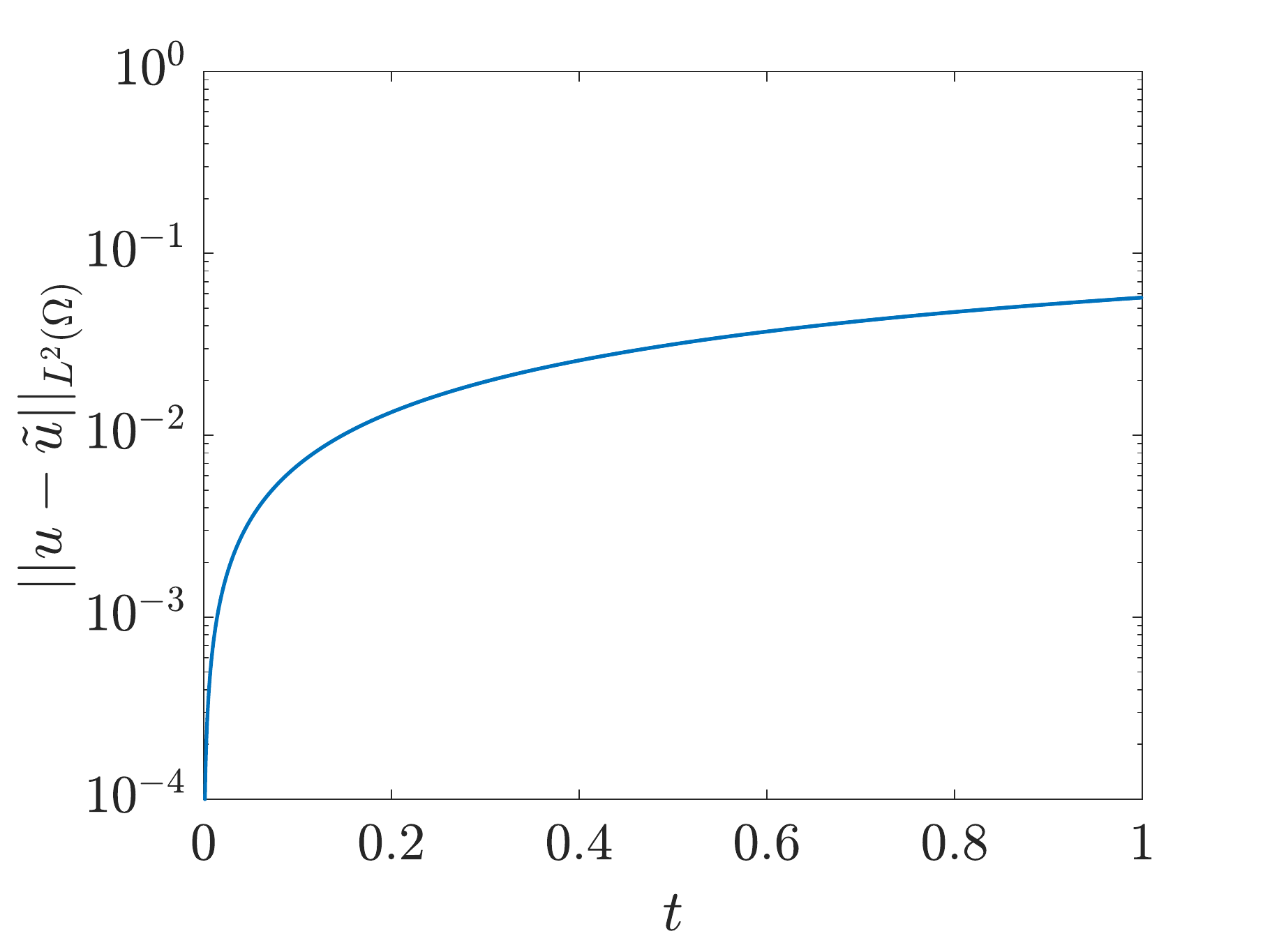}
\includegraphics[width=0.5\textwidth]{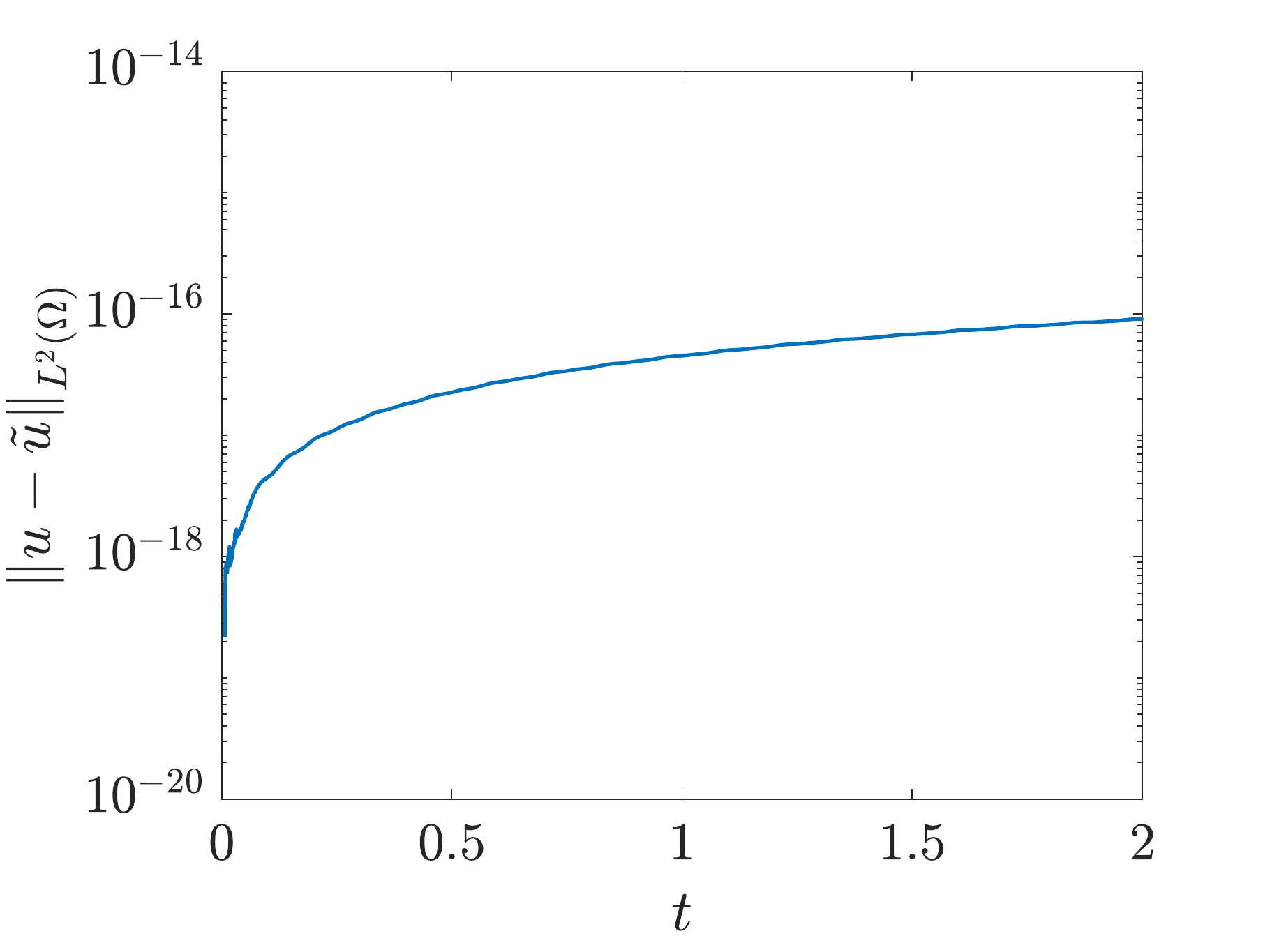}
}
\caption{Time-dependent $L^2(\Omega)$ error of the 
DO-TT  approximation of the solution to the 
four-dimensional PDE \eqref{4d_pde} (a), and 
the fifty-dimensional PDE \eqref{high_dim_pde} (b). 
In (a) we computed the DO-TT error relative to an accurate 
benchmark solution constructed with the method of 
characteristics. In (b) we used the analytical solution 
\eqref{error50d}.}
\label{fig:4d_error}
\end{figure}

% DA QUI
\subsection{Fifty-dimensional hyperbolic PDE}
Let us consider the following fifty-dimensional  
hyperbolic problem
\begin{equation}
 \label{high_dim_pde}
 \begin{cases}
\displaystyle\frac{\partial u(\bm x,t)}{\partial t} = \sum_{j = 1}^{50} f_j(x_j)\frac{\partial u(\bm x,t)}{\partial x_j} \vse\\
{u}(\bm x,0) = \displaystyle\prod_{j = 1}^{50} \psi_0^{(j)}(x_j)
\end{cases}
\end{equation}
subject to periodic boundary conditions in the 
hyper-cube $\Omega=[0, 2\pi]^{50}$.  It is easy to 
show that the solution to \eqref{high_dim_pde} is 
rank-one for all $t\geq 0$. 
Correspondingly, we look for a rank-one DO-TT solution
of the form
\begin{equation}
\tilde {u} = \prod_{j=1}^{50} \psi^{(j)}(t). 
\label{DO-TT-50d}
\end{equation} 
By substituting \eqref{DO-TT-50d} into \eqref{high_dim_pde} 
and imposing DO orthogonality conditions at each level of the 
binary tree yields the DO-TT propagator  
\begin{equation}
\label{do_tt_50d}
\begin{aligned}
\frac{\partial \psi^{(j)}}{\partial t} &= f_j(x_j) \frac{\partial \psi^{(j)}}{\partial x_j} - f_j(x_j)\psi^{(j)} \langle \frac{\partial \psi^{(j)}}{\partial x_j} \psi^{(j)} \rangle \ , \qquad j = 1,2, \ldots, 49\\
\frac{\partial \psi^{(50)}}{\partial t} &= \sum_{j = 1}^{49} f_j(x_j) \langle \frac{\partial \psi^{(j)}}{\partial x_j} \psi^{(j)} \rangle \psi^{(50)}.
\end{aligned}
\end{equation}
Note that this system is nonlinear. Specifically, the first 49 equations 
are uncoupled, while the 50th one is coupled with the entire system.  
We set the coefficients $f_j(x_j) = j$, i.e., constant,  
and the components of the initial condition as 
\begin{equation}
\begin{aligned}
\psi_0^{(j)}(x_j) = \frac{\sin(x_j)}{\sqrt{\pi}}, \quad j = 1,\ldots, 49, 
\qquad \psi_0^{(50)}(x_{50}) = 10^7(3 + \sin(x_{50})),
    \end{aligned}
\end{equation}
which satisfy the bi-orthogonality condition. The analytical solution to \eqref{high_dim_pde} is easily obtained as 
\begin{equation}
\label{hd_analytic}
u = \prod_{j = 1}^{50} \psi_0^{(j)}(x_j + jt).
\end{equation}
As before, we solved the DO-TT system \eqref{do_tt_50d} using 
a Fourier spectral collocation method with  
$60$ Fourier points in each variable $x_j$, and RK4 time integration 
with $\Delta t = 10^{-3}$. The temporal evolution of 
a few representative DO-TT modes is shown in Figure \ref{fig:high_dim_figures}. 
\begin{figure}[t]
\centerline{\footnotesize\hspace{-.2cm} $\psi^{(1)}$ \hspace{4.8cm} $\psi^{(25)}$   \hspace{4.8cm}  $\psi^{(50)}$   }
	\centerline{
	\rotatebox{90}{\hspace{1.3cm}  \footnotesize}
		\includegraphics[width=0.34\textwidth]{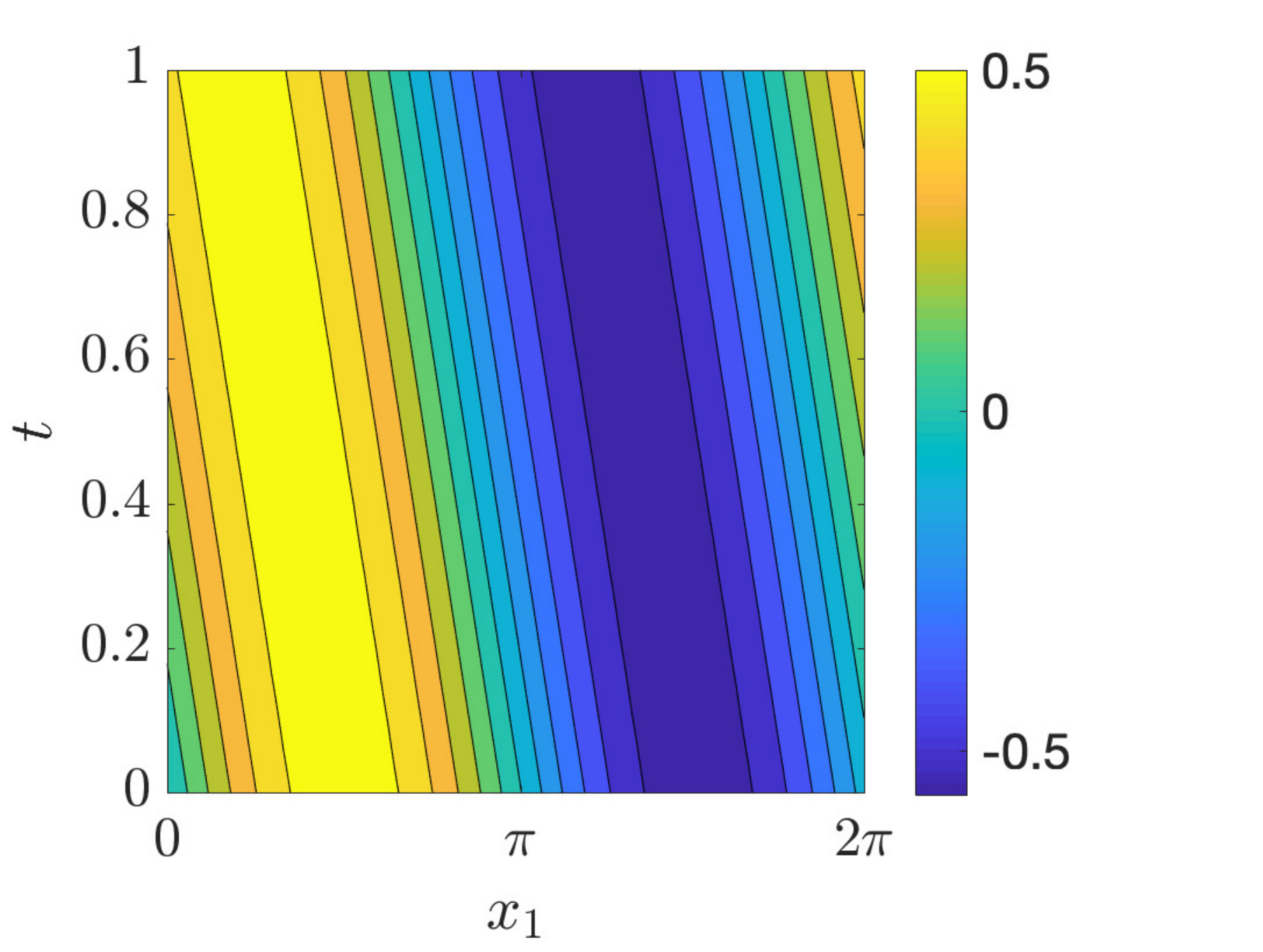}
		\includegraphics[width=0.34\textwidth]{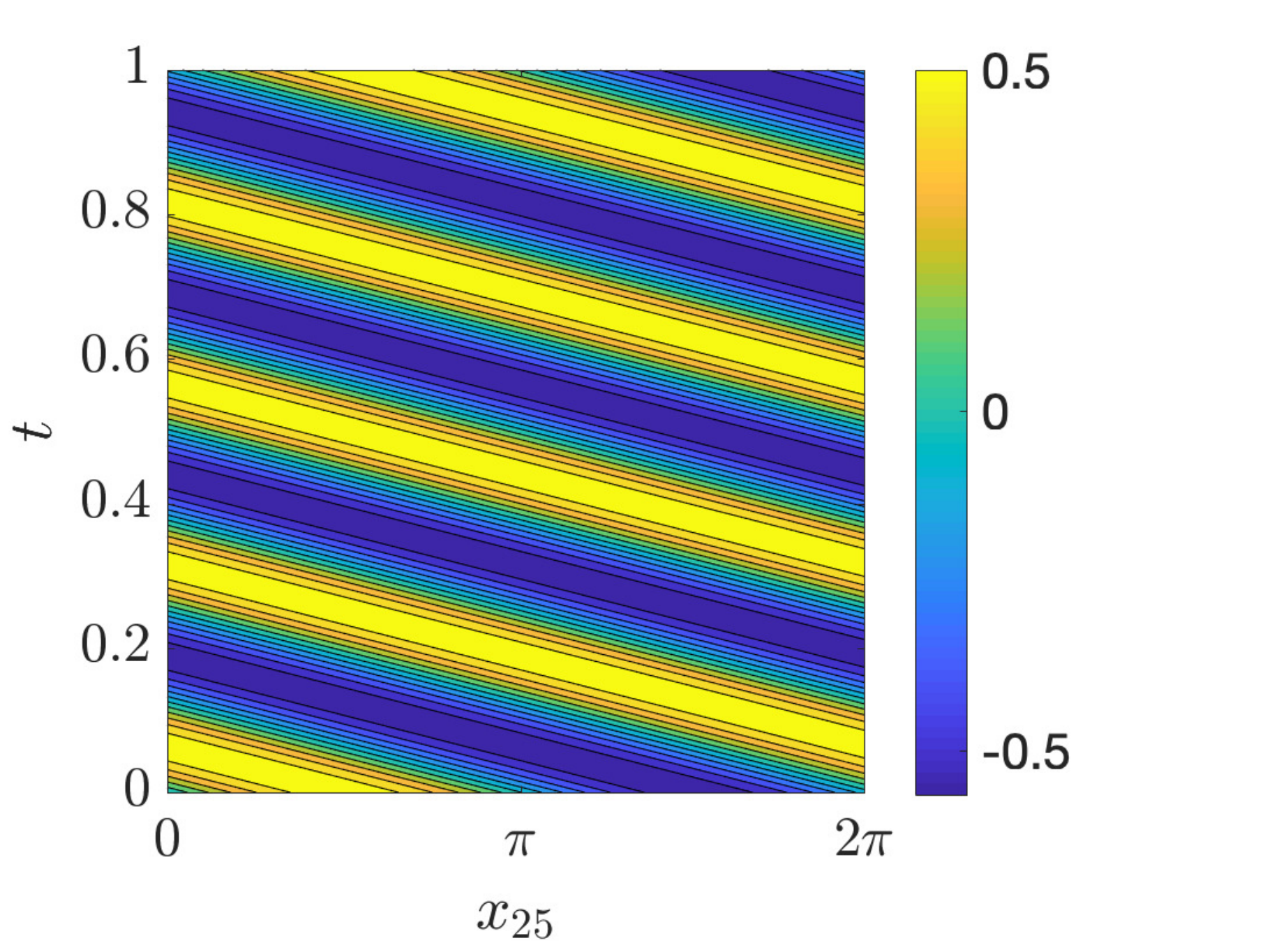}
		\includegraphics[width=0.34\textwidth]{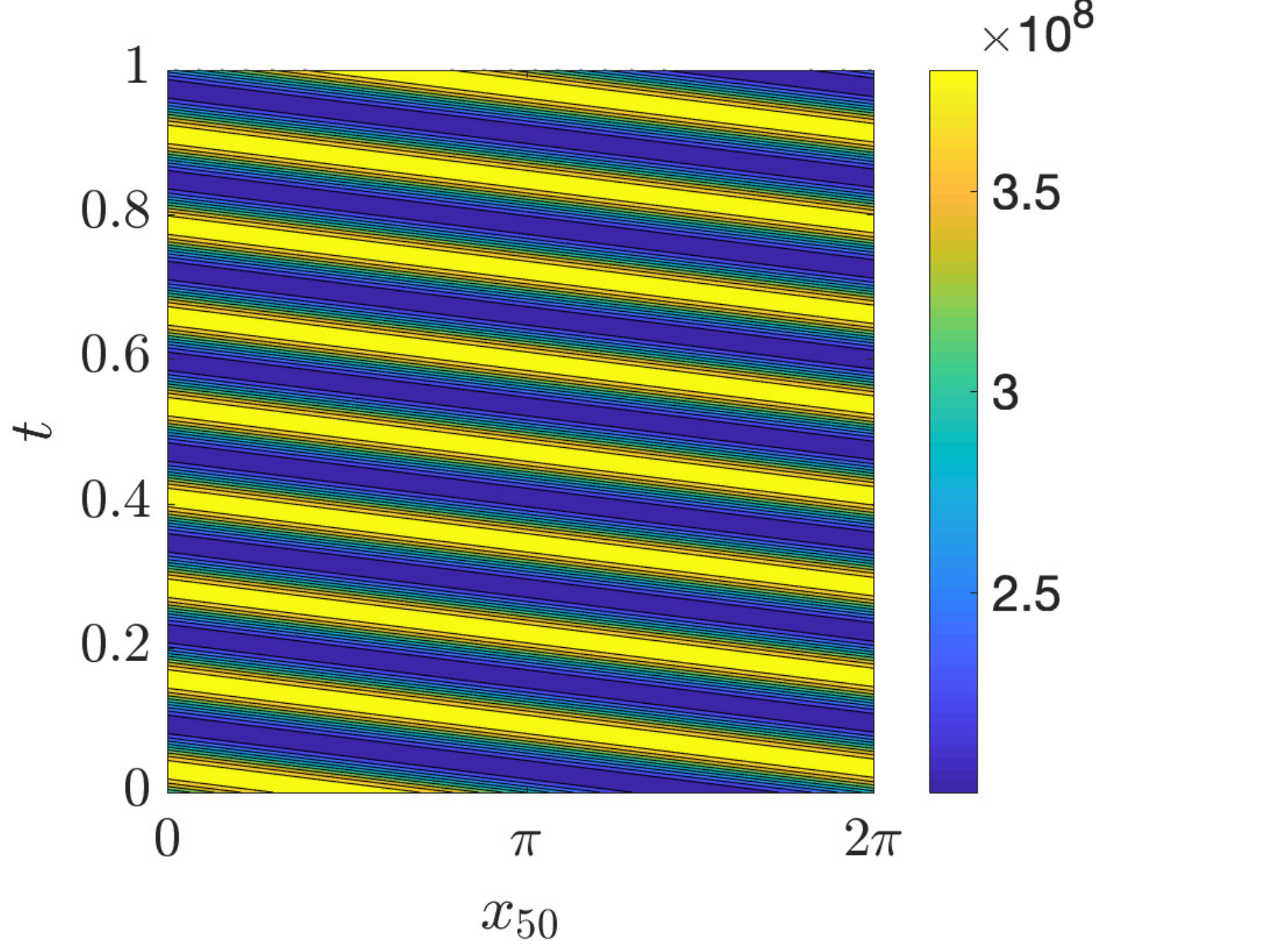}	
}
\caption{Time evolution of a few representative DO-TT modes solving 
nonlinear PDE system  \eqref{do_tt_50d}. It is seen that all modes 
are traveling waves.}
\label{fig:high_dim_figures}
\end{figure}
The $L^2(\Omega)$ error between the analytical 
solution \eqref{hd_analytic} and the DO-TT 
approximation \eqref{DO-TT-50d} can be represented in 
terms of one-dimensional integrals as follows
\begin{equation}
\begin{aligned}
\left\|u-\tilde{u}\right\|^2_{L^2(\Omega)} &= 
\int_{\Omega} \left(\prod_{j = 1}^{50} \psi_0^{(j)}(x_j + jt) - 
\prod_{j=1}^{50} \psi^{(j)}(t)\right)^2 dx_1\cdots dx_{50}\\
&= \prod_{j=1}^{50} \int_0^{2\pi} \psi_0^{(j)}(x_j + jt)^2 dx_j + 
\prod_{j=1}^{50} \int_0^{2\pi} \psi^{(j)}(t)^2 dx_j - 
2\prod_{j=1}^{50} \int_0^{2\pi} \psi^{(j)}_0(x_j + jt) 
\psi^{(j)}(t) dx_j.
\end{aligned}
\label{error50d}
\end{equation}
This error is plotted in Figure \ref{fig:4d_error}(b) versus time.
It is seen that the DO-TT propagator in this case is numerically 
exact.

\subsection{Parabolic PDEs}
\label{sec:parabolic}
In this Section we study DO-TT solution of 
constant-coefficients parabolic PDEs in periodic hypercubes. 
In contrast with the hyperbolic problems discussed in the previous 
Section, the solution in this case does not increase in rank 
throughout propagation. Rather, the solution smooths out
over time and the rank tends to decrease.  This implies 
that in order for the covariance matrices in the DO-TT system  
to remain invertible during propagation, a threshold must be set to 
dynamically remove modes, in particular when 
their energy becomes very small.

\subsubsection{Four-dimensional parabolic PDE}

Let us consider the four-dimensional initial value problem
\begin{equation}
 \label{4dpde_diff}
 \begin{cases}
\displaystyle\frac{\partial u(t,\bm x)}{\partial t} = 
\sum_{j=1}^4\frac{\partial^2 u(\bm x,t) }{\partial x_j^2} \vse\\
\displaystyle {u}(\bm x,0) = \exp\left[-\frac{1}{10}\sin(x_1 + x_2 + x_3 + x_4)\right]
\end{cases}
\end{equation}
in the spatial domain $\Omega=[0,2\pi]^4$, with 
periodic boundary conditions. Note that here we set the same 
initial condition as in the hyperbolic PDE \eqref{4d_pde}. 
In this case, however, the solution decays to zero because 
of diffusion. Hence, we expect 
that the DO-TT solution rank decays in time, instead 
of increasing. Applying the Fourier 
transform $\mathcal{F}[\cdot]$ to 
\eqref{4dpde_diff}, yields the linear ODE
\begin{equation}
 \label{4dpde_diff_transform}
 \begin{cases}
\displaystyle\frac{d\hat{u}(t,\bm \omega)}{d t} = - \sum_{j=1}^4 \omega_j^2 \hat{u}(t,\bm \omega) \vse\\
\displaystyle \hat{u}(\bm \omega,0) = \mathcal{F}\left[\exp\left(-\frac{1}{10} \sin(x_1 + x_2 + x_3 + x_4)\right)\right](\bm \omega)
\end{cases}
\end{equation}
which can be solved analytically or numerically
to obtain a benchmark solution to \eqref{4dpde_diff}. 
With such benchmark solution 
available, we can study the accuracy 
of the DO-TT propagator. To this end, we first 
compute the bi-orthogonal decomposition of the 
initial condition in \eqref{4dpde_diff} as we have 
done in Section \ref{sec:4dhyperbolic}, with threshold set 
to $\sigma = 10^{-10}$.  This yields the the 
hierarchical ranks \eqref{ranks}, as before. 
In Figure \ref{fig:4d_diff_error} we plot $L^2(\Omega)$ 
error of the DO-TT approximation of the solution 
to \eqref{4dpde_diff} relative to the benchmark solution 
obtained via Fourier transform. It is seen that, contrary 
to hyperbolic problems, diffusion promotes 
a low-rank structure of the solution.
\begin{figure}[t]
\centerline{\footnotesize (a)\hspace{8cm } (b)}
	\centerline{
	\rotatebox{90}{\hspace{1.3cm}  \footnotesize}
		\includegraphics[width=0.49\textwidth]{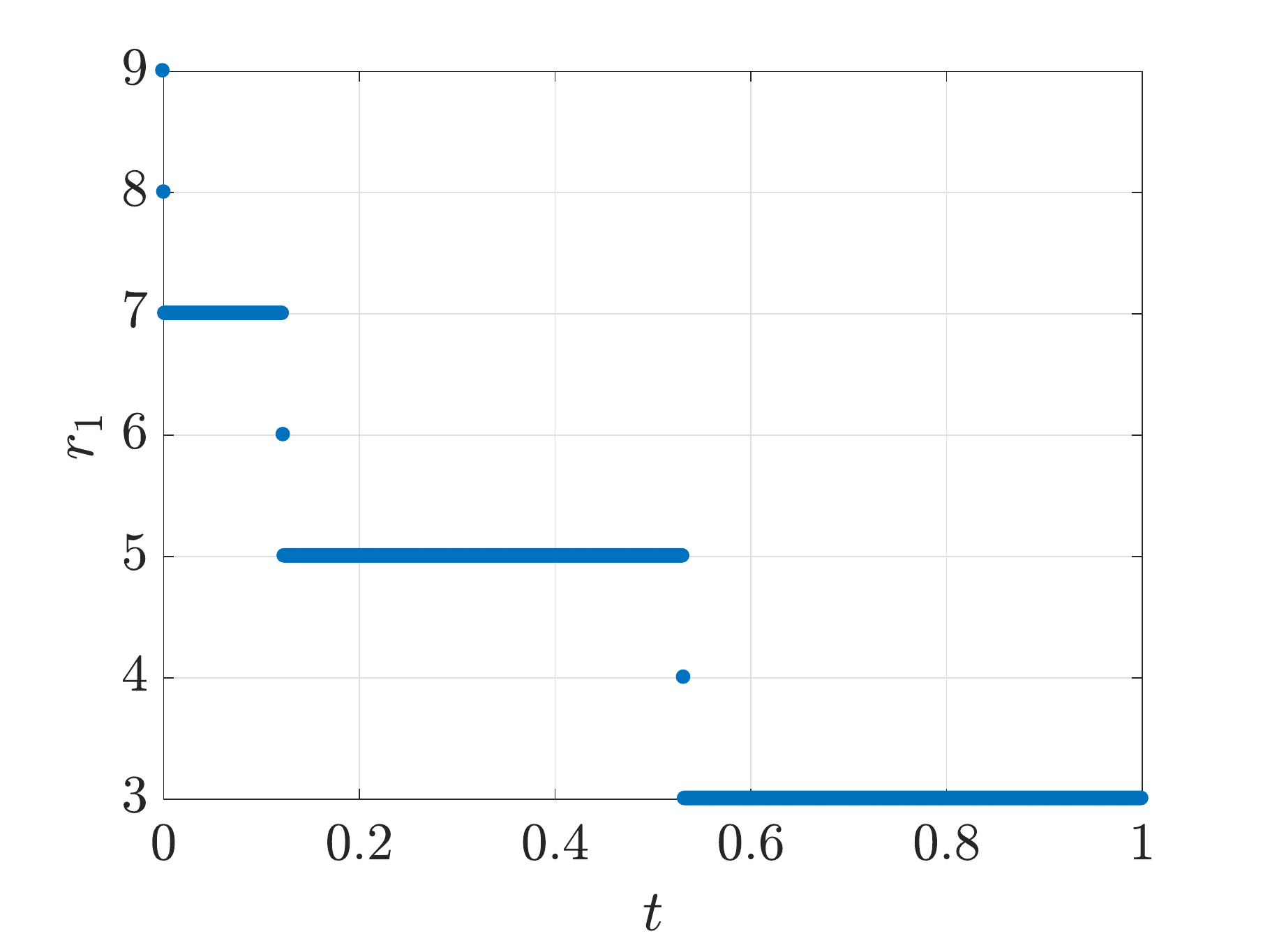}
		\includegraphics[width=0.49\textwidth]{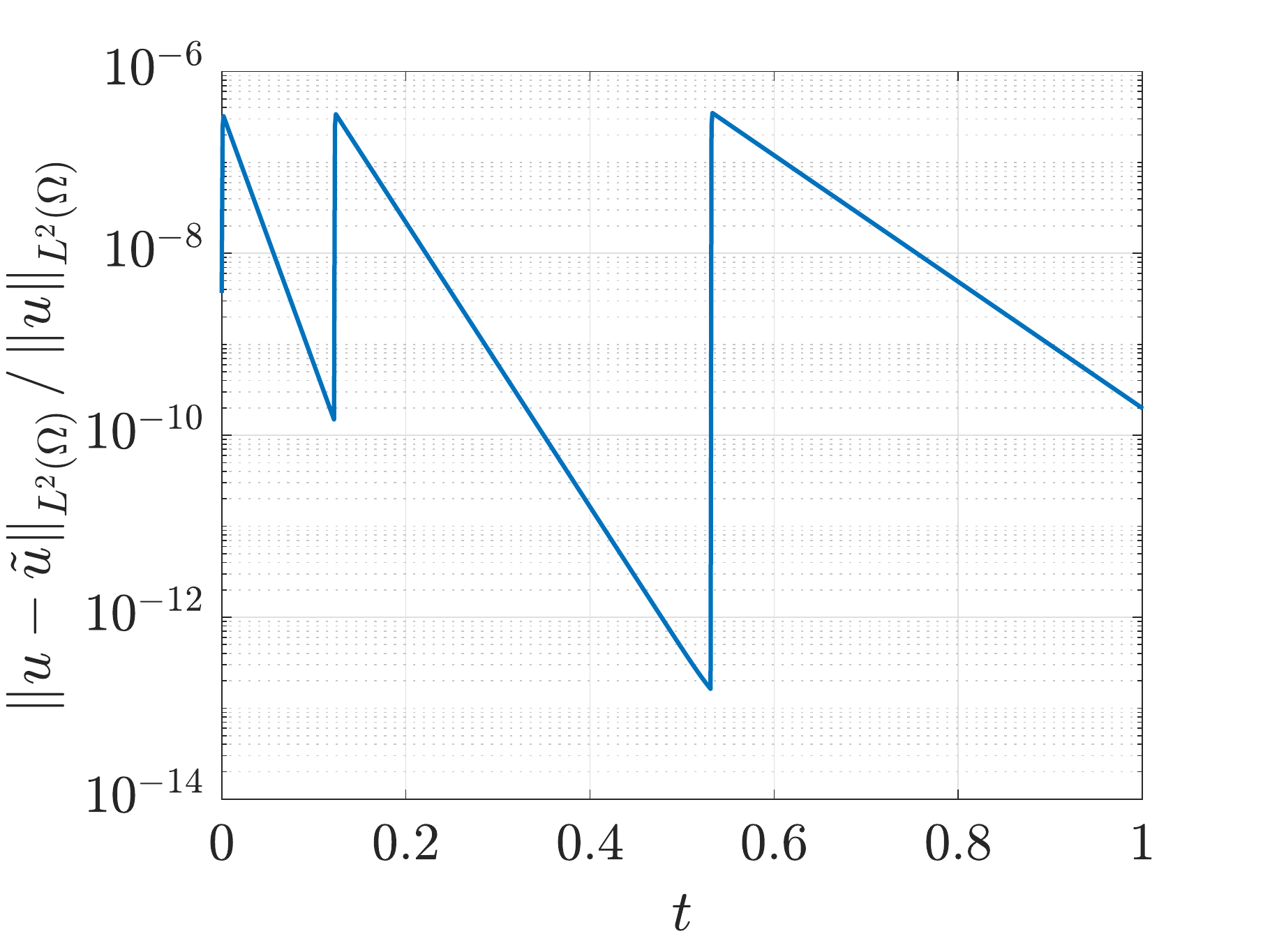}		
}
\caption{(a) Temporal evolution of the level-1 rank in the DO-TT  
approximation of the solution to  four-dimensional diffusion problem \eqref{4dpde_diff}. It is seen that that diffusion promotes 
a low-rank structure of the solution.  In (b)
we plot $L^2(\Omega)$ error of the DO-TT 
approximation relative to the benchmark solution, which is obtained 
via Fourier transform.}
\label{fig:4d_diff_error}
\end{figure}

This means that the energy of the DO-TT modes tends 
to decrease in time, which implies that the covariance 
matrix at the left hand side of the DO-TT propagator
can become numerically singular in time. To avoid such singularities 
we introduce a threshold $\epsilon = 10^{-10}$ on the
eigenvalues of the level-1 TT binary tree, i.e., $\lambda_{i_1}$, 
to determine whether we should decrease the leading rank 
$r_1$. Specifically, if $\lambda_{r_1} < \epsilon$ 
then $r_1 = r_1 - 1$.  For this example, 
thresholding only the first level is sufficient 
to avoid singular covariance matrices at all levels of 
the TT binary tree.  In Figure \ref{fig:4d_diff_error} 
we see that each time we remove one mode we have 
that the relative error jumps.  
Note however, that it immediately decays again with a 
rate that is inversely proportional to the number 
of the DO-TT modes kept in the series expansion.

\subsection{Fifty-dimensional parabolic PDE}
The last problem we consider is a fifty-dimensional 
constant-coefficient diffusion problem with rank-one 
 initial condition
\begin{equation}
 \label{high_dim_pde_diff}
 \begin{cases}
\displaystyle\frac{\partial u(\bm x,t) }{\partial t} = \sum_{j = 1}^{50} \frac{\partial^2 u(\bm x,t) }{\partial x_j^2}  \vse\\
{u}(\bm x,t) = \displaystyle\prod_{j = 1}^{50} \psi_0^{(j)}(x_j)
\end{cases}
\end{equation}
in the spatial domain $\Omega=[0,2\pi]^50$, with 
periodic boundary conditions. The analytical solution can 
be obtained by using the method of separation of variables \cite{Ozisik} as
\begin{equation}
\label{hd_analytic_diff}
u = \prod_{j = 1}^{50} \psi^{(j)}_{0}e^{-50 t}.
\end{equation}
Hence, it is a rank-one solution. Correspondingly, we 
look for a rank-one DO-TT solution of the form 
\begin{equation}
\tilde{u} = \prod_{j=1}^{50} \psi^{(j)}(t) \ , 
\label{rank-one-diff}
\end{equation}  
where the modes $\psi^{(j)}(t)$ satisfy the DO-TT system of
evolution equations
\begin{equation}
\label{do_tt_50d_diff}
\begin{aligned}
\frac{\partial \psi^{(j)}}{\partial t} &= 
\frac{\partial^2 \psi^{(j)}}{\partial x_j^2} - 
\psi^{(j)} \langle \frac{\partial^2 \psi^{(j)}}
{\partial x_j^2} \psi^{(j)}\rangle \ , \qquad j = 1,2, \ldots, 49 \ ,\\
\frac{\partial \psi^{(50)}}{\partial t} &= \sum_{j = 1}^{49} 
\langle \frac{\partial^2 \psi^{(j)}}{\partial x_j^2} 
\psi^{(j)} \rangle \psi^{(50)}.
\end{aligned}
\end{equation}
We set the initial condition in \eqref{high_dim_pde_diff} as 
\begin{equation}
\begin{aligned}
\psi_0^{(j)}= \frac{\sin(x_j)}{\sqrt{\pi}}, \quad j = 1,\ldots, 49 \ ,
\qquad \psi_0^{(50)} = 10^7\sin(x_{50}) \ ,
    \end{aligned}
\end{equation}
which satisfy the orthogonality conditions. As before, 
we solved the system \eqref{do_tt_50d_diff} using 
Fourier spectral collocation with 60 nodes in each variable $x_j$
and RK4 time integration with $\Delta t = 10^{-3}$. In Figure 
\ref{fig:high_dim_figures_diff} we plot the temporal evolution 
of a few representative DO-TT modes. As expected, 
they all decay to zero. 
\begin{figure}[t]
\centerline{\footnotesize\hspace{.2cm} $\psi^{(1)}$ \hspace{4.8cm} $\psi^{(25)}$   \hspace{4.8cm}  $\psi^{(50)}$   }
	\centerline{
	\rotatebox{90}{\hspace{1.3cm}  \footnotesize}
		\includegraphics[width=0.34\textwidth]{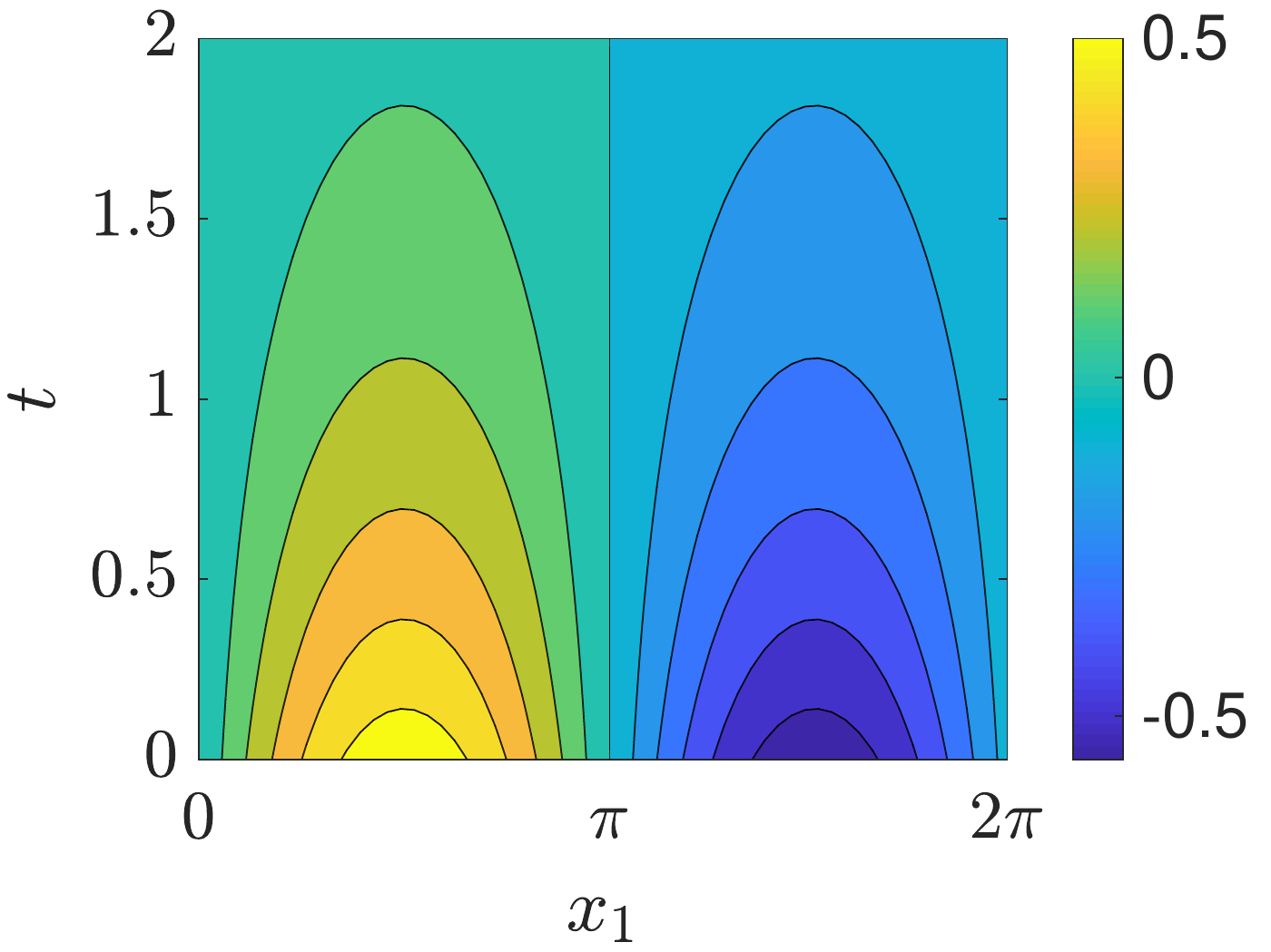}
		\includegraphics[width=0.34\textwidth]{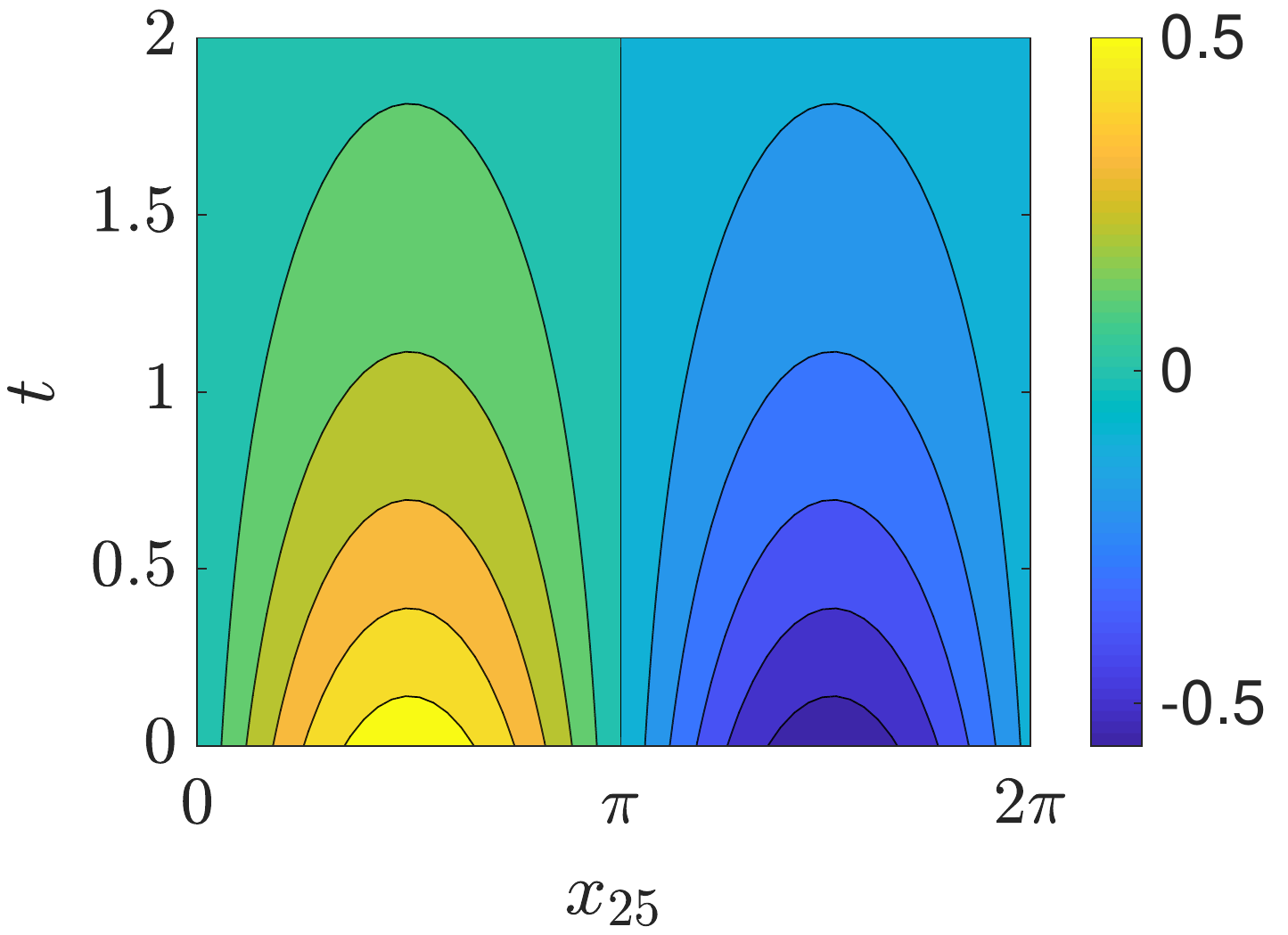}
		\includegraphics[width=0.34\textwidth]{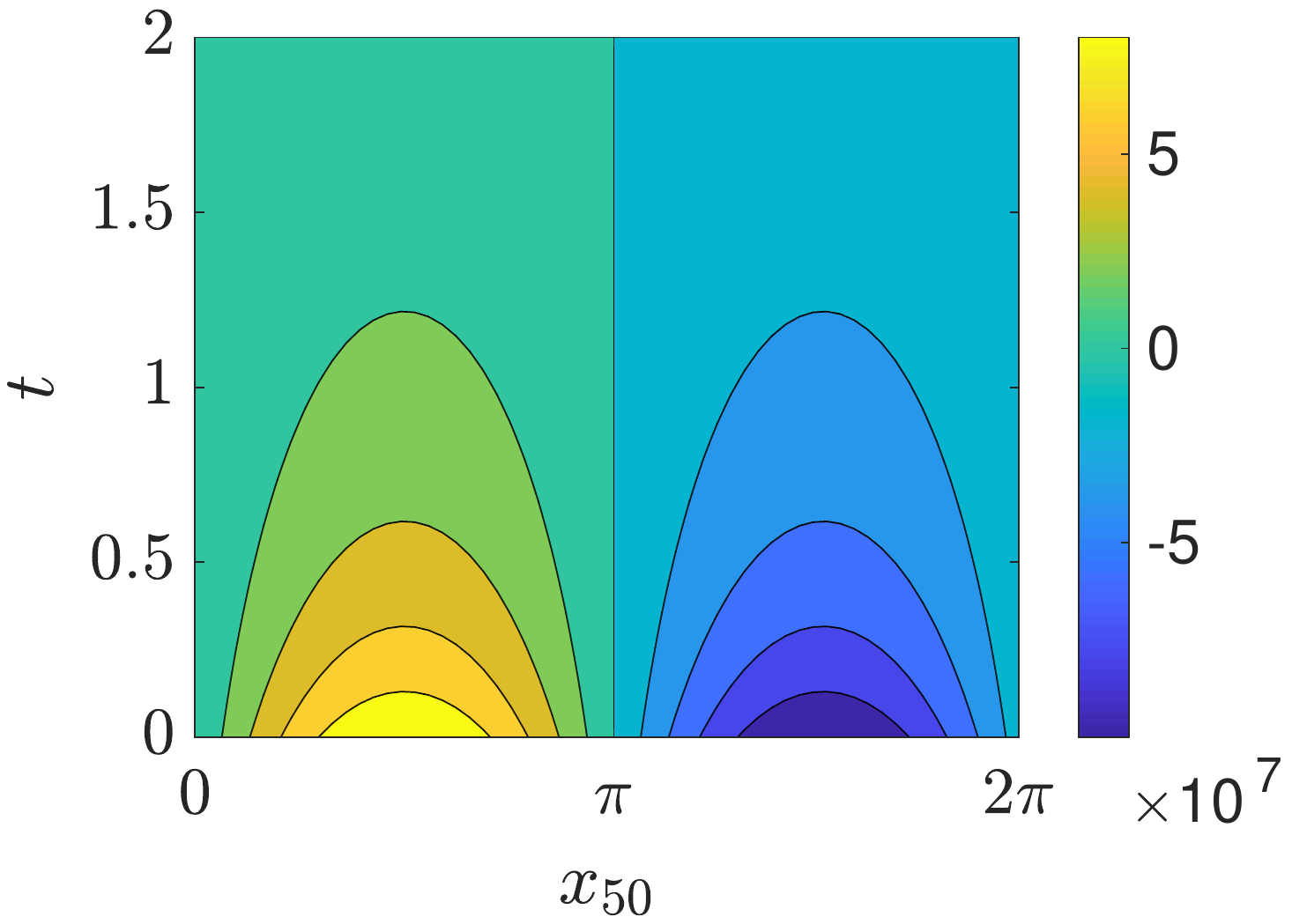}	
}
\caption{Fifty-dimensional diffusion problem \eqref{high_dim_pde_diff}. 
Time evolution of a few representative DO-TT modes  
appearing in the rank-one expansion of solution \eqref{rank-one-diff}. 
As expected, all modes decay to zero.}
\label{fig:high_dim_figures_diff}
\end{figure}
Finally in Figure \ref{fig:50d_diff_error} we plot the time-dependent 
$L^2(\Omega)$ error between the DO-TT solution \eqref{rank-one-diff} 
and the analytical solution \eqref{hd_analytic_diff}. Such error can be 
expressed in terms of one-dimensional integrals as follows
 
\begin{equation}
\begin{aligned}
\left\|u-\tilde{u}\right\|^2_{L^2(\Omega)} 
= &\prod_{j=1}^{50} \int_0^{2\pi} \psi_0^{(j)}(x_j)^2 dx_j e^{-100t} + \prod_{j=1}^{50} \int_0^{2\pi} \psi^{(j)}(t)^2 dx_j - \\
& 2\prod_{j=1}^{50} \int_0^{2\pi} \psi^{(j)}_0(x_j) \psi^{(j)}(t) dx_j e^{-50t}.
\end{aligned}
\end{equation}

\begin{figure}
\centerline{\hspace{0.3cm}\hspace{8cm}}
\centerline{
\includegraphics[width=0.55\textwidth]{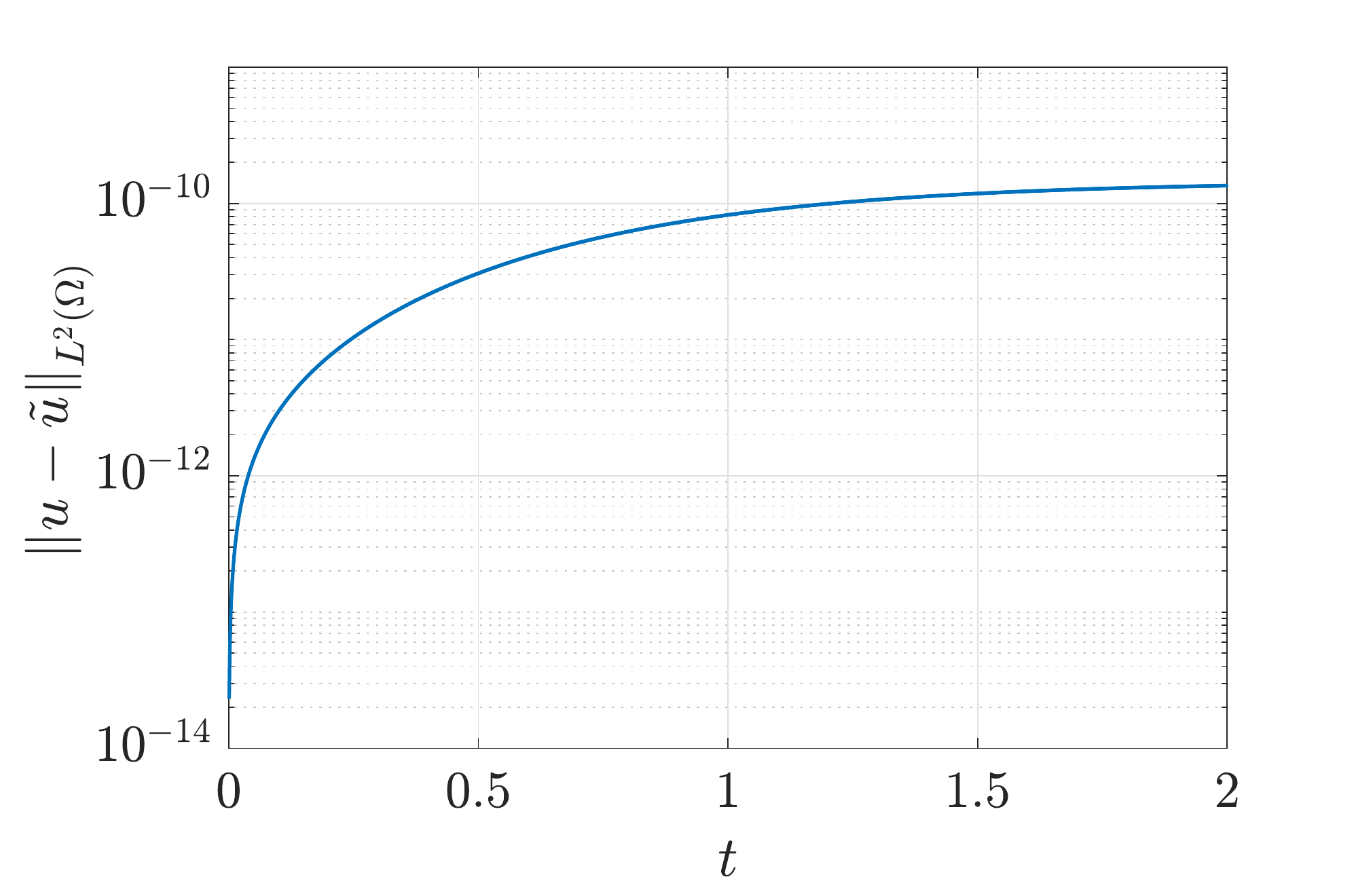}
}
\caption{Fifty-dimensional diffusion problem \eqref{high_dim_pde_diff}.
$L^2(\Omega)$ error between the DO-TT solution \eqref{rank-one-diff} 
and the analytical solution \eqref{hd_analytic_diff}.}
\label{fig:50d_diff_error}
\end{figure}

\section{Summary}
\label{sec:summary}

In this paper, we presented a new dynamically orthogonal 
tensor decomposition method to approximate multivariate 
functions and the solution to high-dimensional 
time-dependent nonlinear PDEs. 
The key idea relies on a hierarchical decomposition 
of the approximation space obtained by splitting 
the independent variables of the problem into 
disjoint subsets.
This process, which can be conveniently 
be visualized in terms of binary trees, yields series expansions 
analogous to the classical tensor-train and hierarchical Tucker 
tensor formats.
By enforcing dynamic orthogonality conditions at each 
level of binary tree, we obtained evolution 
equations for the orthogonal modes spanning each of the nested 
subspaces in the hierarchical decomposition. 
This allowed us to represent the temporal dynamics of 
high-dimensional functions, and compute the solution 
to high-dimensional time-dependent PDEs on a tensor 
manifold with constant rank. We also proposed a new 
algorithm for dynamic addition and removal of modes 
within each subspace, and demonstrated its 
effectiveness in numerical applications to hyperbolic 
and parabolic PDEs. 
The mathematical techniques and algorithms 
we presented in this paper can be readily applied to 
more general high-dimensional nonlinear systems and PDEs, 
such as the finite-dimensional approximation 
of nonlinear functionals and functional 
differential equations  \cite{venturi2018numerical}.

\vs\vs
\noindent 
{\bf Acknowledgements} 
This research was supported by the U.S. Army Research Office  
grant W911NF1810309.

%\newpage
%\section*{References}
\bibliographystyle{plain}
%\bibliography{paper}

\end{document}